\documentclass{article}
\usepackage{amsmath,amsfonts,amssymb,amsthm,hyperref,enumerate,algpseudocode,algorithm,subcaption,booktabs,fullpage}
\usepackage{float}
\newfloat{algorithm}{t}{lop}

\newtheorem{lemma}{Lemma}[section]
\newtheorem{theorem}{Theorem}[section]
\newtheorem{corollary}{Corollary}[section]
\theoremstyle{definition}
\newtheorem{definition}{Definition}[section]

\author{Riley Badenbroek \and Etienne de Klerk}
\title{Complexity Analysis of a Sampling-Based Interior Point Method for Convex Optimization}

\DeclareMathOperator{\diff}{d\!}
\DeclareMathOperator{\interior}{int}

\DeclareMathOperator{\tr}{tr}

\DeclareMathOperator{\dom}{dom}

\algnewcommand\Input{\item[\bf Input:] }
\algnewcommand\Output{\item[\bf Output:] }
\newcommand\algmultiline[1]{\parbox[t]{\dimexpr\textwidth-\leftmargin-\labelsep-\labelwidth}{#1\strut}}

\begin{document}

\maketitle

\begin{abstract}
	We develop a short-step interior point method to optimize a linear function over a convex body assuming that one only knows a membership oracle for this body. The approach is based on  Abernethy and Hazan's sketch of a universal interior point method using the so-called entropic barrier [arXiv 1507.02528v2, 2015]. It is well-known that the gradient and Hessian of the entropic barrier can be approximated by sampling from Boltzmann-Gibbs distributions, and the entropic barrier was shown to be self-concordant by Bubeck and Eldan [arXiv 1412.1587v3, 2015]. The analysis of our algorithm uses properties of the entropic barrier, mixing times for hit-and-run random walks by Lov\'asz and Vempala [Foundations of Computer Science, 2006], approximation quality guarantees for the mean and covariance of a log-concave distribution, and results from de Klerk, Glineur and Taylor on inexact Newton-type methods [arXiv 1709.0519, 2017].
\end{abstract}

\section{Introduction}
The interior point revolution, to used a phrase coined by Wright \cite{Wright}, was the introduction of polynomial-time logarithmic barrier methods for
linear programming, and their subsequent extension to convex programming.
In their seminal work on the extension of interior point methods to convex programming, Nesterov and Nemirovskii \cite{nesterov1994interior}  proved that every open convex set that does not contain an affine subspace is the domain of a  self-concordant barrier, called the universal barrier. While this is an important theoretical result, its practical applicability  is limited to cases where a barrier is known in closed form, and where its
gradient and Hessian may be computed efficiently.

The most practical interior point software deals with self-dual cones, where self-concordant barriers are known; e.g. MOSEK \cite{mosek8}, SDPT3 \cite{toh1999sdpt3}, and SeDuMi \cite{sturm1999using}. A promising recent development for more general cones is the primal-dual algorithm developed by Skajaa and Ye \cite{skajaa2015homogeneous} and implemented in the software \emph{alfonso} by Papp and Y{\i}ld{\i}z \cite{alfonso2018,papp2018homogeneous}, which only requires an efficiently computable self-concordant barrier of the primal cone. However, there are many convex bodies  where one can solve the membership problem in polynomial time, but where no efficiently computable self-concordant barrier is known, such as the subtour elimination polytope.

Abernethy and Hazan \cite{abernethy2015faster} recently connected the field of simulated annealing (Kirkpatrick et al.\ \cite{kirkpatrick1983optimization}) to the study of interior point methods using the entropic barrier of Bubeck and Eldan \cite{bubeck2014entropic}. One important property of this barrier is that its gradient and Hessian may be approximated through sampling, even though the barrier may not be known in closed form. This opens up the possibility of using interior point methods on sets for which we do not know efficiently computable self-concordant barriers. Interestingly, G\"uler \cite{guler1996barrier} showed that the universal and entropic barriers coincide (up to an additive constant) if the domain of the barrier is a homogeneous cone, i.e.\ a cone with a transitive automorphism group.

The interior point method Abernethy and Hazan proposed was shown to converge in polynomial time by De Klerk, Glineur and Taylor \cite{deklerk2017worst}, provided one can approximate the gradients and Hessians of the barrier sufficiently well. Our aim is to investigate when the approximations that may be obtained through sampling satisfy the requirements from \cite{deklerk2017worst}. In other words, we aim to show that we may approximate the gradient and Hessian sufficiently well in polynomial time, with high probability. The sampling algorithm that we will use is the Markov chain Monte Carlo method known as hit-and-run sampling, first suggested by Smith  \cite{smith1984efficient}. The mixing properties of this method, established by Lov\'asz and Vempala \cite{lovasz2006fast}, allow us to show the gradients and Hessians of the entropic barrier can be approximated to the desired accuracy in polynomial time, with high probability.
Our hope is that this analysis will contribute to an extension of interior point methods to convex bodies where the  membership problem is `easy', but no efficiently computable barrier is known.

\subsubsection*{Outline of this paper}
The outline of this paper is as follows. After some preliminary definitions in Section \ref{sec:Preliminaries}, we prove some useful properties of the entropic barrier in Section \ref{sec:EntropicBarrier}. Then, we give a hit-and-run mixing time theorem using self-concordance in Section \ref{sec:HitandRun}. With this result established, we can show in Section \ref{sec:SamplingQuality} how hit-and-run might be applied to approximate means and covariances of Boltzmann distributions. Such approximations can then be used in Section \ref{sec:ShortStepIPM} to analyze the aforementioned interior point method.

Much of the analysis presented here is of a technical nature, and the reader may wish to skip the proofs on a first reading.

\section{Preliminaries}
\label{sec:Preliminaries}
We are interested in the problem
\begin{equation}
	\label{eq:MainProblem}
	\min_{x \in K} \langle c, x \rangle,
\end{equation}
where $K \subseteq \mathbb{R}^n$ is a convex body, and $\langle \cdot, \cdot \rangle$ is a reference inner product on $\mathbb{R}^n$. We may assume that $\| c \| = \sqrt{\langle c, c \rangle} = 1$, and that $K$ contains a ball of radius $r > 0$ and is contained in a ball of radius $R \geq r$.
For any self-adjoint, positive definite linear operator $A$, we can define the inner product $\langle \cdot, \cdot \rangle_A$ by $\langle x, y \rangle_A := \langle x, Ay \rangle$. The reference inner product induces the norm $\| \cdot \|$, and the inner product $\langle \cdot, \cdot \rangle_A$ induces the norm $\| \cdot \|_A$.

\subsection{Self-Concordant Functions and the Entropic Barrier}
The following discussion is condensed from Renegar \cite{renegar2001mathematical}.
Let $\langle \cdot, \cdot \rangle$ be any inner product on $\mathbb{R}^n$.
A function $f$ from $\dom f \subseteq \mathbb{R}^n$ to $\mathbb{R}$ is differentiable at $\theta \in \dom f$ if there exists a vector $g(\theta) \in \mathbb{R}^n$ such that
\begin{equation*}
	\lim_{\| \Delta \theta \| \to 0} \frac{f(\theta + \Delta \theta) - f(\theta) - \langle g(\theta), \Delta \theta \rangle}{\| \Delta \theta \|} = 0.
\end{equation*}
The vector $g(\theta)$ is called the gradient of $f$ at $\theta$ with respect to $\langle \cdot, \cdot \rangle$.

Furthermore, the function $f$ is twice differentiable at $\theta \in \dom f$ if it is continuously differentiable at $f$ and there exists a linear operator $H(\theta): \mathbb{R}^n \to \mathbb{R}^n$ such that
\begin{equation*}
	\lim_{\| \Delta \theta \| \to 0} \frac{\| g(\theta + \Delta \theta) - g(\theta) - H(\theta) \Delta \theta\|}{\| \Delta \theta \|} = 0.
\end{equation*}
The linear operator $H(\theta)$ is called the Hessian of $f$ at $\theta$ with respect to $\langle \cdot, \cdot \rangle$.

We denote the gradient and Hessian with respect to some other inner product $\langle \cdot, \cdot \rangle_A$ by $g_A$ and $H_A$ respectively, and it can be shown that $g_A(\theta) = A^{-1} g(\theta)$ and $H_A(\theta) = A^{-1} H(\theta)$ (see e.g. Theorems 1.2.1 and 1.3.1 in Renegar \cite{renegar2001mathematical}). For brevity, define $\langle \cdot, \cdot \rangle_\theta := \langle \cdot, \cdot \rangle_{H(\theta)}$ and $\| \cdot \|_\theta := \| \cdot \|_{H(\theta)}$ for any $\theta \in \mathbb{R}^n$ and let $g_\theta$ and $H_\theta$ be the gradient and Hessian of $f$ with respect to the local inner product $\langle \cdot, \cdot \rangle_\theta$.

The class of self-concordant functions plays an important role in the theory of interior point methods. We use the definition by Renegar \cite{renegar2001mathematical}.
\begin{definition}
	\label{def:SC}
	A function $f$ is \emph{self-concordant} if for all $\theta_0 \in \dom f$ and $\theta_1 \in \mathbb{R}^n$ such that $\| \theta_1 - \theta_0 \|_{\theta_0} < 1$, we have $\theta_1 \in \dom f$ and the following inequalities hold for all non-zero $v \in \mathbb{R}^n$:
	\begin{equation}
		\label{eq:SCGeneral}
		1 - \| \theta_1 - \theta_0 \|_{\theta_0} \leq \frac{\| v \|_{\theta_1}}{\| v \|_{\theta_0}} \leq \frac{1}{1 - \| \theta_1 - \theta_0 \|_{\theta_0}}.
	\end{equation}
\end{definition}
Some self-concordant functions have the additional property that the local norm of their gradients are bounded. Such functions are called barriers.
\begin{definition}
	\label{def:Barrier}
	A function $f$ is a \emph{barrier} if it is self-concordant and
	\begin{equation*}
		\vartheta := \sup_{\theta \in \dom f} \| g_\theta(\theta) \|_\theta^2 < \infty.
	\end{equation*}
	The value $\vartheta$ is called the \emph{complexity parameter} of $f$.
\end{definition}

%\subsection{Entropic Barrier}
In the rest of this paper, the function $f$ will always denote the log partition function associated with $K$, which is defined for any $\theta \in \mathbb{R}^n$ by
\begin{equation*}
	f(\theta) := \ln \int_{K} e^{\langle \theta, x \rangle} \diff x.
\end{equation*}
Moreover, denote the expectation of a Boltzmann distribution over $K$ with parameter $\theta \in \mathbb{R}^n$ by
\begin{equation*}
%	x(\theta) :=
	\mathbb{E}_\theta[X] = \frac{\int_K x e^{\langle \theta, x \rangle} \diff x}{\int_K e^{\langle \theta, x \rangle} \diff x}.
\end{equation*}
%As in Renegar \cite{renegar2001mathematical}, let the vector $g(\theta)$ and the linear operator $H(\theta)$ be the gradient and Hessian of $f$ at $\theta \in \mathbb{R}^n$ (in the sense of Fr\'echet derivatives) with respect to the reference inner product.
It is not hard to see that for all $v \in \mathbb{R}^n$,
\begin{equation}
	\label{eq:DerivativesF}
	g(\theta) = \mathbb{E}_\theta[X], \qquad H(\theta)v = \mathbb{E}_\theta[ \langle X - \mathbb{E}_\theta[ X], v \rangle (X - \mathbb{E}_\theta [X]) ].
\end{equation}
If $\langle \cdot, \cdot \rangle$ were the Euclidean inner product, $H(\theta)$ can be represented by the covariance matrix $\mathbb{E}_{{\theta}} [(X - \mathbb{E}_{{\theta}}[X]) (X - \mathbb{E}_{{\theta}}[X])^\top]$ of a Boltzmann distribution with parameter $\theta$ over $K$. To emphasize this fact, we will write $\Sigma(\theta)$ instead of $H(\theta)$ where appropriate.
%let $g$ and $H$ be the gradient and Hessian of $f$ (in the sense of Fr\'echet derivatives) with respect to the reference inner product.
%It is not hard to see that for all $v,w \in \mathbb{R}^n$,
%\begin{equation*}
%	g(\theta)[w] = \langle \mathbb{E}_\theta[X], w \rangle, \qquad H(\theta)[v,w] = \mathbb{E}_\theta[ \langle X - \mathbb{E}_\theta[ X], v \rangle \langle X - \mathbb{E}_\theta [X], w \rangle ].
%\end{equation*}
%One could represent the linear operators $g(\theta)$ and $H(\theta)$ as a vector and a matrix respectively. If $\langle \cdot, \cdot \rangle$ were the Euclidean dot product, $g(\theta)$ can be thought of as $\mathbb{E}_\theta[X]$, the expectation of a Boltzmann distribution with parameter $\theta$ over $K$. Similarly,

It was shown by Bubeck and Eldan \cite{bubeck2014entropic} that $f$ is self-concordant. In this case, Definition \ref{def:SC} guarantees
that for all $\theta_0, \theta_1 \in \mathbb{R}^n$ such that $\| \theta_1 - \theta_0 \|_{\theta_0} < 1$, the following inequalities hold for all non-zero $v \in \mathbb{R}^n$:
\begin{equation}
	\label{eq:SCf}
	1 - \| \theta_1 - \theta_0 \|_{\theta_0} \leq \frac{\| v \|_{\theta_1}}{\| v \|_{\theta_0}} \leq \frac{1}{1 - \| \theta_1 - \theta_0 \|_{\theta_0}}.
\end{equation}

Moreover, let $f^*$ be the Fenchel conjugate of $f$, defined in the usual manner:
\begin{equation*}
	f^*(x) = \sup_{\theta \in \mathbb{R}^n} \{ \langle\theta, x \rangle - f(\theta) \},
\end{equation*}
where $x \in \interior K$. (The reason that $\dom f^* = \interior K$ will be discussed shortly.)
The function $f^*$ is called the \emph{entropic barrier} for $K$. Again borrowing notation from Renegar \cite{renegar2001mathematical}, let $g^*$ and $H^*$ be the gradient and Hessian of $f^*$ (all linear operators in this paper are self-adjoint, so we will not use an asterisk to refer to an adjoint).
Define $\langle \cdot, \cdot \rangle^*_x := \langle \cdot, \cdot \rangle_{H^*(x)}$ for all $x \in \interior K$, and let $\| \cdot \|_x^*$ be the local norm induced by this inner product.

Since it was shown by Bubeck and Eldan \cite{bubeck2014entropic} that $f$ is self-concordant, $f^*$ is self-concordant as well. Here, Definition \ref{def:SC} shows that for all $x_0 \in \interior K$ and $x_1 \in \mathbb{R}^n$ such that $\| x_1 - x_0 \|_{x_0}^* < 1$, it holds that $x_1 \in \interior K$ and for all non-zero $v \in \mathbb{R}^n$,
\begin{equation}
	\label{eq:SCf*}
	1 - \| x_1 - x_0 \|_{x_0}^* \leq \frac{\| v \|_{x_1}^*}{\| v \|_{x_0}^*} \leq \frac{1}{1 - \| x_1 - x_0 \|_{x_0}^*}.
\end{equation}

The following is known about the domain of the conjugate of self-concordant functions.
\begin{lemma}[Proposition 3.3.3 in \cite{renegar2001mathematical}]
	\label{lemma:DomainConjugate}
	For any self-concordant $f$, we have $\dom f^* = \{ g(\theta): \theta \in \mathbb{R}^n \}$.
\end{lemma}
For the log partition function specifically, we have an explicit description of $\{ g(\theta): \theta \in \mathbb{R}^n \}$.
\begin{lemma}[Lemma 3.1 in \cite{klartag2006convex}]
	\label{lemma:InteriorKImageG}
	If $f$ is the log partition function associated with a convex body $K$, then $\{ g(\theta) : \theta \in \mathbb{R}^n \} = \interior K$.
\end{lemma}

Hence, $\dom f^* = \interior K$, as one would expect from a barrier for $K$.
%It was shown to be a self-concordant barrier for $K$ by \cite{bubeck2014entropic}, using the fact that $f$ is self-concordant.
The following is known about the derivatives of $f^*$.
\begin{lemma}[Theorem 3.3.4 in \cite{renegar2001mathematical}]
	\label{lemma:PropertiesDerivativesConjugate}
	For any self-concordant $f$, we have for all $\theta \in \mathbb{R}^n$,
	\begin{equation*}
		g^*(g(\theta)) = \theta, \qquad H^*(g(\theta)) = H(\theta)^{-1}.
	\end{equation*}
\end{lemma}
%The gradient and Hessian of the entropic barrier are thus intimately connected to the expectation and covariance of a family of Boltzmann distributions. In fact, the expectations that this family of distributions can attain are exactly equal to the interior of $K$.

To clarify that $g$ assigns to a $\theta\in \mathbb{R}^n$ a point $x \in K$, we will sometimes write $x(\theta)$ for $g(\theta)$.
Lemmas \ref{lemma:DomainConjugate}, \ref{lemma:InteriorKImageG} and \ref{lemma:PropertiesDerivativesConjugate} imply that $g^*$ assigns to every $x \in \interior K$ a vector $\theta \in \mathbb{R}^n$ such that $g(\theta) = x(\theta) = x$. For this reason, we will often write $\theta(x)$ for $g^*(x)$ to keep the notation intuitive. The notation from this section is summarized in Table \ref{tab:Notation}.

\begin{table}[!ht]
	\centering
	\begin{tabular}{l @{\hskip 30pt} l}
		\toprule
		Log partition function $f$ & Entropic barrier $f^*$\\
		\midrule
		$f(\theta) = \ln \int_K e^{\langle \theta, x \rangle} \diff x$
		& $f^*(x) = \sup_{\theta \in \mathbb{R}^n} \{ \langle \theta, x \rangle - f(\theta) \}$\\ [1ex]
		$g(\theta) = \mathbb{E}_{\theta}[X] = x(\theta)$
		& $g^*(x) = \theta(x)$ such that $g(\theta(x)) = x$\\ [1ex]
		$H(\theta) = \Sigma(\theta)$
		& $H^*(x) = H(g^*(x))^{-1} = \Sigma(\theta(x))^{-1}$\\ [1ex]
		Domain: $\theta \in \mathbb{R}^n$ &
		Domain: $x \in \interior K$\\
		\bottomrule
	\end{tabular}
	\caption{Overview of the properties of $f$ and $f^*$}
	\label{tab:Notation}
\end{table}

Finally, it was shown by Bubeck and Eldan \cite{bubeck2014entropic} that $\vartheta = n + o(n)$, where $\vartheta$ is the complexity parameter of the entropic barrier defined in Definition \ref{def:Barrier}.
%\todo{Is it useful to note that $f$ does not have a complexity parameter since $\langle g(\theta_0), \theta_1 - \theta_0 \rangle$ is unbounded if we only know $\theta_0, \theta_1 \in \mathbb{R}^n$?}
%As a final piece of notation, define $\langle \cdot, \cdot \rangle^*_x := \langle \cdot, \cdot \rangle_{H^*(x)}$ for all $x \in \interior K$, and let $\| \cdot \|_x^*$ be the local norm induced by this inner product.

\subsection{Interior Point Method by Abernethy-Hazan}
\label{subsec:AbernethyHazan}

The short-step interior point method proposed by Abernethy and Hazan \cite{abernethy2015faster} is sketched in Algorithm \ref{alg:AbernethyHazan}.

\begin{algorithm}[!ht]
	\caption{Sketch of the interior point method by Abernethy and Hazan \cite{abernethy2015faster}}
	\label{alg:AbernethyHazan}
	\begin{algorithmic}[1]
		\Input Tolerances $\epsilon, \widetilde{\epsilon}, \bar{\epsilon} > 0$; entropic barrier parameter $\vartheta \leq n + o(n)$; objective $c \in \mathbb{R}^n$; central path proximity parameter $\delta > 0$; an $x_0 \in K$ and $\eta_0 > 0$ such that $\| x_0 - x(-\eta_0 c)\|_{x(-\eta_0 c)}^* \leq \frac{1}{2} \delta$; growth rate $\beta > 0$; sample size $N \in \mathbb{N}$.
		\Output $x_k$ such that $\langle c, x_k \rangle - \min_{x \in K} \langle c, x \rangle \leq \bar{\epsilon}$ at terminal iteration $k$.
		\Statex
		
		\State $\widehat{\epsilon} \gets \widetilde{\epsilon} \sqrt{\frac{1+\epsilon}{1-\epsilon}} + \sqrt{\frac{2 \epsilon}{1 - \epsilon}}$
		\State $\gamma \gets \frac{2(1-\delta)^4 - \widehat{\epsilon} (1+(1-\delta)^4)}{(1-\widehat{\epsilon})(1-\delta)^2(1+(1-\delta)^4)}$
		\State $k \gets 0$
		\While{$\vartheta(1+\frac{1}{2}\delta) / \eta_k > \bar{\epsilon}$}
		\State Generate samples $Y^{(1)}, ..., Y^{(N)}$ from the Boltzmann distribution with parameter $-\eta_k c$
		\State Find approximation $\widehat{\Sigma}(-\eta_k c)$ of $\Sigma(-\eta_k c)$ using samples $Y^{(1)}, ..., Y^{(N)}$
		\State Find approximation $\widehat{\theta}(x_k)$ of $\theta(x_k)$ by minimizing $\Psi_k(\theta) = f(\theta) - \langle \theta, x_k \rangle$ over $\mathbb{R}^n$
		\State $x_{k+1} \gets x_{k} - \gamma \widehat{\Sigma}(-\eta_k c) [\eta_{k+1} c + \widehat{\theta}(x_{k})]$ \Comment{$x_{k+1} \approx x(-\eta_{k+1} c)$}
		\State $\eta_{k+1} \gets (1 + \frac{\beta}{\sqrt{\vartheta}}) \eta_k$
		\State $k \gets k+1$
		\EndWhile
		\State \Return $x_k$
	\end{algorithmic}
\end{algorithm}

In every iteration $k$, we would like to know (an approximation of) $H^*(x_k)^{-1}$. Because the function $f^*$ is self-concordant, $H^*(x_k)$ is not too different from $H^*(x(-\eta_k c))$ if $x_k$ is close to $x(-\eta_k c)$ in some well-defined sense. Thus, the algorithm can also be proven to work if we find (an approximation of) $H^*(x(-\eta_k c))^{-1} = \Sigma(-\eta_k c)$. In practice, we can find an approximation of $\Sigma(-\eta_k c)$ by generating sufficiently many samples from the Boltzmann distribution with parameter $-\eta_k c$ and computing the empirical covariance matrix.

To find $g^*(x_k) = \theta(x_k)$, we will use the approach proposed by Abernethy and Hazan \cite{abernethy2015faster}. Note that the function $\Psi_k(\theta) = f(\theta) - \langle \theta, x_k \rangle$ has Fr\'echet derivatives
\begin{equation}
	D \Psi_k(\theta) = g(\theta) - x_k = \mathbb{E}_\theta[X] - x_k, \qquad D^2 \Psi_k(\theta) = \Sigma(\theta).
	\label{eq:DerivativesPsi}
\end{equation}
In other words, $\Psi_k$ is a convex function which is minimized at the $\theta \in \mathbb{R}^n$ such that $g(\theta) = x_k$, which is equal to $\theta(x_k)$. Therefore, to approximate $\theta(x_k)$, it suffices to minimize $\Psi_k$ over $\theta \in \mathbb{R}^n$. As is clear from \eqref{eq:DerivativesPsi}, the gradient and Hessian of $\Psi_k$ can be approximated at some particular $\theta \in \mathbb{R}^n$ by generating sufficiently many samples of the Boltzmann distribution with parameter $\theta$, and consequently computing the empirical mean c.q. covariance matrix.

One might wonder what quality guarantees should be satisfied by the approximations $\widehat{\Sigma}(-\eta_k c)$ and $\widehat{\theta}(x_k)$ such that the algorithm still provably works. This question was answered by the following theorem from de Klerk, Glineur and Taylor \cite{deklerk2017worst}.
\begin{theorem}[Theorem 7.6 in \cite{deklerk2017worst}]
	Consider Algorithm \ref{alg:AbernethyHazan} with the input settings $\beta = \frac{1}{32}$ and $\delta = \frac{1}{4}$, and let $\epsilon, \widetilde{\epsilon} > 0$ such that
	$\widehat{\epsilon} = \widetilde{\epsilon} \sqrt{\frac{1+\epsilon}{1-\epsilon}} + \sqrt{\frac{2 \epsilon}{1 - \epsilon}} \leq \frac{1}{6}$.
	Suppose that in every iteration of Algorithm \ref{alg:AbernethyHazan}, the approximation $\widehat{\Sigma}(-\eta_k c)$ satisfies
	\begin{align*}
		(1-\epsilon) y^\top \widehat{\Sigma}(-\eta_k c) y &\leq y^\top \Sigma(-\eta_k c) y \leq (1+\epsilon) y^\top \widehat{\Sigma}(-\eta_k c) y && \forall y \in \mathbb{R}^n,\\
		(1-\epsilon) y^\top \widehat{\Sigma}(-\eta_k c)^{-1} y &\leq y^\top \Sigma(-\eta_k c)^{-1} y \leq (1+\epsilon) y^\top \widehat{\Sigma}(-\eta_k c)^{-1} y && \forall y \in \mathbb{R}^n,
	\end{align*}
	and that the approximation $\widehat{\theta}(x_k)$ satisfies
	\begin{equation*}
		\left\| \widehat{\theta}(x_k) - \theta(x_k) \right\|_{\Sigma(-\eta_{k+1} c)} \leq \widetilde{\epsilon} \left\| -\eta_{k+1} c - \theta(x_k) \right\|_{\Sigma(-\eta_{k+1} c)}.
	\end{equation*}
	
	If the algorithm is initialized with an $x_0 \in K$ and $\eta_0 > 0$ such that $\| x_0 - x(-\eta_0 c)\|_{x(-\eta_0 c)}^* \leq \frac{1}{2} \delta$, then it terminates after
	\begin{equation*}
		k = \left\lceil 40 \sqrt{\vartheta} \ln \left( \frac{\vartheta (1 + \frac{1}{2} \delta)}{\eta_0 \bar{\epsilon}} \right) \right\rceil
	\end{equation*}
	iterations. The result is an $x_k$ such that
	\begin{equation*}
		\langle c, x_k \rangle - \min_{x \in K} \langle c, x \rangle \leq \bar{\epsilon}.
	\end{equation*}
\end{theorem}
The main purpose of this paper is to give a detailed description of how one can approximate $\Sigma(-\eta_k c)$ and $\theta(x_k)$ in practice such that quality requirements similar to the ones in the theorem above are satisfied. In order to do this, we need some results from probability theory.

\subsection{Log-Concavity and Divergence of Probability Distributions}
The probability density function of a Boltzmann distribution belongs to the well-studied class of log-concave functions.
%Readers familiar with log-concavity, the $L_2$-norm of a probability distribution with respect to another probability distribution, and total variation distance may wish to skip this section.
%
We start by recalling the definition of log-concavity.
\begin{definition}
	A function $h : \mathbb{R}^n \to \mathbb{R}_{+}$ is \emph{log-concave} if for any two $x,y \in \mathbb{R}^n$ and $\lambda \in (0,1)$,
	\begin{equation*}
		h(\lambda x + (1-\lambda) y) \geq h(x)^\lambda h(y)^{1-\lambda}.
	\end{equation*}
	In other words, if $h$ is strictly positive, then it is log-concave if and only if $x \mapsto \log(h(x))$ is concave.
%	\label{def:logConcavity}
\end{definition}
We will need the following concentration result for log-concave distributions. Note that it is stronger than Chebyshev's inequality.
\begin{lemma}[Lemma 3.3 from \cite{lovasz2006simulated}]
	\label{lemma:LogConcaveTail}
	Let $X$ be a random variable with a log-concave distribution, and let $\| \cdot \|$ be the Euclidean norm. Denote $\mathbb{E}[\|X - \mathbb{E}[X] \|^2] =: \sigma^2$. Then for all $t > 1$,
	\begin{equation*}
	\mathbb{P}\{ \|X - \mathbb{E}[X] \| > t \sigma \} \leq e^{1-t}.
	\end{equation*}
\end{lemma}
Next, we define level sets for general probability density functions.
\begin{definition}
	\label{def:LevelSet}
	Let $h: \mathbb{R}^n \to \mathbb{R}$ be a probability density function supported on $K \subseteq \mathbb{R}^n$. Then, the \emph{level set} $L_p$ of (the distribution with density) $h$ is $\{ x \in K : h(x) \geq \alpha_p \}$, where $\alpha_p$ is chosen such that $\int_{L_p} h(x) \diff x = p$.
\end{definition}
Note that the level sets of log-concave distributions are convex (see e.g. Section 3.5 in Boyd and Vandenberghe \cite{boyd2004convex}).

We will generate samples from the family of Boltzmann distributions with a random walk method known as hit-and-run sampling, to be defined later.
Hit-and-run sampling is only guaranteed to work if the distribution of the starting point and the distribution one would like to sample from are ``close''. Even then, the result of the hit-and-run walk does not have the correct distribution, but a distribution which is again ``close'' to the desired distribution. To make these statements exact, we will need two measures of divergence between probability distributions. Before we can define them, we recall the definition of absolute continuity.
\begin{definition}
	Let $(K, \mathcal{E})$ be a measurable space, and let $\nu$ and $\mu$ be measures on this space. Then, $\nu$ is \emph{absolutely continuous} with respect to $\mu$ if $\mu(A) = 0$ implies $\nu(A) = 0$ for all $A \in \mathcal{E}$.
	We write this property as $\nu \ll \mu$.
\end{definition}

The first measure of divergence between probability distributions is the $L_2$-norm.
\begin{definition}
	Let $(K, \mathcal{E})$ be a measurable space. Let $\nu$ and $\mu$ be two probability distributions over this space, such that $\nu \ll \mu$. Then,
	the \emph{$L_2$-norm of $\nu$ with respect to $\mu$} is
	\begin{equation*}
		\|\nu / \mu \| := \int_{K} \frac{\diff \nu}{\diff \mu} \diff \nu = \int_{K} \left(\frac{\diff \nu}{\diff \mu} \right)^2 \diff \mu,
	\end{equation*}
	where $\frac{\diff \nu}{\diff \mu}$ is the Radon-Nikodym derivative of $\mu$ with respect to $\nu$.
\end{definition}
If $K \subseteq \mathbb{R}^n$, and $\nu$ and $\mu$ have probability densities $h_\nu$ and $h_\mu$, respectively, with respect to the Lebesgue measure, it can be shown that
\begin{equation*}
	\| \nu / \mu \| = \int_K \frac{h_\nu(x)}{h_\mu(x)} h_\nu(x) \diff x.
%	\| \mu / \nu \| = \mathbb{E}_{\mu} \left( \frac{h_1(X)}{h_2(X)} \right) = \int_K \frac{h_1(x)}{h_2(x)} h_1(x) \diff x.
%	\label{eq:L2ExpectationQuotientDensities}
\end{equation*}

The second way in which we will measure distance between probability distributions is by total variation distance.
\begin{definition}
	\label{def:TotalVariationDistance}
	Let $(K, \mathcal{E})$ be a measurable space. For two probability distributions $\mu$ and $\nu$ over this space, their \emph{total variation distance} is
	\begin{equation*}
		\|\mu - \nu \| := \sup_{A \in \mathcal{E}} | \mu(A) - \nu(A)|.
	\end{equation*}
\end{definition}

A useful property of the total variation distance is that it allows coupling of random variables, as the following lemma asserts.
\begin{lemma}[e.g. Proposition 4.7 in \cite{levin2017markov}]
	Let $X$ be a
%	$(K, \mathcal{E})$-valued
	random variable on $K \subseteq \mathbb{R}^n$
%	$(\Omega, \mathcal{F}, \mathbb{P})$
	with distribution $\mu$, and let $\nu$ be a different probability distribution on $K$.
%	over $(K, \mathcal{E})$
%	Assume that
%	\begin{equation*}
%		B = \left\{ x \in K : \frac{\diff \nu}{\diff \lambda}(x) > \frac{\diff \mu}{\diff \lambda}(x) \right\} \in \mathcal{E},
%	\end{equation*}
%	and that there exists some $A_0 \in \mathcal{F}$ such that
%	\begin{equation*}
%		\mathbb{P}\{ \omega \in A_0 : X(\omega) \in B' \}= \nu(B') \qquad  \forall B' \in \mathcal{E}, B' \subseteq K \setminus B.
%	\end{equation*}
	If $\|\mu - \nu\| = \alpha$, we can construct another
%	$(K, \mathcal{E})$-valued
	random variable $Y$ on $K$ distributed according to $\nu$ such that $\mathbb{P}\{X = Y\} = 1-\alpha$.
%	The event $\{X \neq Y\}$ is referred to as \emph{divine intervention}.
	\label{lemma:DivineIntervention}
\end{lemma}

%\begin{proof}
%	See Appendix \ref{app:Proofs}.
%\end{proof}
%In the remainder of this paper, we will assume that the $\sigma$-algebras underlying our problem are sufficiently rich to satisfy the conditions in Lemma \ref{lemma:DivineIntervention}.

\subsection{Near-Independence}
The end point of a random walk depends on the starting point of the walk, but as the walk length increases, this dependence starts to vanish. We will use the notion of near-independence to quantify this.
\begin{definition}
	\label{def:NearIndependence}
	Two random variables $X$ and $Y$ taking values in measurable space $(K, \mathcal{E})$ are \emph{near-independent} or \emph{$q$-independent} if for all $A, B \in \mathcal{E}$,
	\begin{equation*}
		|\mathbb{P}\{ X \in A \wedge Y \in B \} - \mathbb{P}\{ X \in A \} \mathbb{P}\{ Y \in B \}| \leq q.
	\end{equation*}
\end{definition}

Before we can analyze the near-independence of starting and end points of a random walk, we need the formal machinery of Markov kernels. Intuitively, a Markov kernel assigns to any point in $K$ a probability distribution over $K$. Its analogue for discrete space Markov chains is a transition probability matrix.
\begin{definition}
	\label{def:MarkovKernel}
	Let $(K, \mathcal{E})$ be a measurable space, and let $\mathcal{B}[0,1]$ be the Borel-$\sigma$-algebra over $[0,1]$. A \emph{Markov kernel} is a map $Q: K \times \mathcal{E} \to [0,1]$ with the properties
	\begin{enumerate}[(i)]
		\item For every $x \in K$, the map $B \mapsto Q(x,B)$ for $B \in \mathcal{E}$ is a probability measure on $(K,\mathcal{E})$;
		\item For every $B \in \mathcal{E}$, the map $x \mapsto Q(x,B)$ for $x \in K$ is $(\mathcal{E}, \mathcal{B}[0,1])$-measurable.
	\end{enumerate}
\end{definition}
Since $Q(x, \cdot)$ is a measure for any fixed $x \in K$, we can integrate a function $\phi$ over $K$ with respect to this measure. This integral will be denoted by $\int_K \phi(y) Q(x, \diff y)$. We emphasize again that for any $x \in K$, this expression is just a Lebesgue integral.

Suppose the Markov kernel $Q$ corresponds to one step of a random walk, i.e. after one step from $x \in K$, the probability of ending up in $B \in \mathcal{E}$ is $Q(x,B)$.
%We may then define $Q^m$ for $m \in \mathbb{N}$ as the Markov kernel with the property that, starting from $x \in K$, one ends up in $B \in \mathcal{E}$ after $m$ of these steps with probability $Q^m(x,B)$.
%
%Moreover, we could also consider a random starting point for our random walk, rather than a fixed point. If the starting point of a random walk has distribution $\mu$ over $K$, then we denote the distribution of the current point after $m$ steps of the random walk by $\mu Q^m$.
The probability that after $m \geq 1$ steps a random walk starting at $x \in K$ ends up in $B \in \mathcal{E}$ is then given by
\begin{equation*}
	Q^m(x,B) := \int_K Q(y,B) Q^{m-1}(x, \diff y),
\end{equation*}
where $Q^1 := Q$. Another interpretation of $Q^m(x,B)$ is the probability of a random walk ending in $B$, conditional on the random starting point $X$ of the walk taking value $x$.
If moreover the starting point of the random walk is not fixed, but follows a probability distribution $\nu$, then the end point of the random walk after $m$ steps follows distribution $\nu Q^m$, defined by
\begin{equation*}
	(\nu Q^m)(B) = \int_K Q^m(x, B) \diff \nu(x),
\end{equation*}
for all $B \in \mathcal{E}$.
%Finally, the distribution of the end point of a random walk conditional on the starting point $X$ taking value $x \in K$ is defined by
%\begin{equation*}
%	Q^m(X,B) := \int_K Q(y,B) Q^{m-1}(x, \diff y),
%\end{equation*}
%for all $B \in \mathcal{E}$.

The following lemma connects total variation distance to near-independence. It will ensure that if the distribution of the end point $Y$ of a random walk approaches some fixed desired distribution $\mu$, then the start point $X$ of this random walk and $Y$ are near-independent.
A similar relation was established by Lov\'asz and Vempala \cite{lovasz2006simulated}, but we will use a version that does not assume $Y$ follows the desired distribution $\mu$.

\begin{lemma}[cf. Lemma 4.3(a) in \cite{lovasz2006simulated}]
	\label{lemma:NearIndependenceWalk}
	Fix a probability distribution $\mu$ over a set $K \subseteq \mathbb{R}^n$. Let $Q$ be a Markov kernel on $K$, and let $\ell: \mathbb{R}_+ \to \mathbb{N}$.
	Suppose that for any $\bar{M} \geq 0$, $\bar{q} > 0$ and any
%	random variable $\bar{X}$ with
	distribution $\bar{\nu}$ satisfying $\bar{\nu} \ll \mu$ and $\| \bar{\nu} / \mu \| \leq \bar{M}$, it holds that
%	$\| Q^{\ell(\bar{M}/\bar{q}^2)}(\bar{X}, \cdot) - \mu \| \leq \bar{q}$.
	$\| \bar{\nu} Q^{\ell(\bar{M}/\bar{q}^2)} - \mu \| \leq \bar{q}$.
	Let $M \geq 0, q > 0$, and let $\nu$ be a distribution such that $\nu \ll \mu$ and $\| \nu / \mu \| \leq M$. If $X$ is a random variable with distribution $\nu$, and $Y$ is a random variable with distribution conditional on $X=x$ given by $Q^{\ell(M/q^2)}(x, \cdot)$ for any $x \in K$, then $X$ and $Y$ are $3q$-independent.
\end{lemma}
\begin{proof}
	Let $A$ and $B$ be measurable subsets of $K$. As noted in Lov\'asz and Vempala \cite[relation (4)]{lovasz2006simulated}, one has the elementary relation
	\begin{equation*}
		\Big|\mathbb{P} \{ Y \in B \wedge X \in A \} - \mathbb{P}\{ Y \in B\} \mathbb{P}\{ X \in A \}\Big|
		= \Big|\mathbb{P} \{ Y \in B \wedge X \notin A \} - \mathbb{P}\{ Y \in B\} \mathbb{P}\{ X \notin A \} \Big|.
	\end{equation*}
	We may therefore assume $\mathbb{P}\{ X \in A \} = \nu(A) \geq \frac{1}{2}$.
	
	The marginal distribution of $Y$ satisfies
	\begin{equation}
		\label{eq:ProofNearIndependenceWalk1}
		\mathbb{P} \{ Y \in B \} = \int_K Q^{\ell(M/q^2)}(x, B) \diff \nu(x) = \nu Q^{\ell(M/q^2)}(B).
	\end{equation}
	Consider the restriction $\nu_A$ of $\nu$ to $A$, scaled to be a probability measure. Then,
	\begin{align}
		\mathbb{P} \{ Y \in B | X \in A \}
		&= \frac{\mathbb{P} \{ Y \in B \wedge X \in A \}}{\mathbb{P} \{ X \in A \}} \nonumber\\
		&= \frac{\int_A Q^{\ell(M/q^2)}(x, B) \diff \nu(x)}{\nu(A)} \nonumber\\
		&= \int_K Q^{\ell(M/q^2)}(x, B) \diff \nu_A(x) \nonumber\\
		&= \nu_A Q^{\ell(M/q^2)}(B).
		\label{eq:ProofNearIndependenceWalk2}
	\end{align}
	
	Since $\nu(A) \geq \frac{1}{2}$, we have $\frac{\diff \nu_A}{\diff \nu}(x) \leq 2$ for $\nu$-almost all $x \in K$. Then,
	\begin{equation*}
		\| \nu_A / \mu \| = \int_K \left( \frac{\diff \nu_A}{\diff \mu} \right)^2 \diff \mu = \int_K \left( \frac{\diff \nu_A}{\diff \nu} \right)^2 \left( \frac{\diff \nu}{\diff \mu} \right)^2 \diff \mu \leq 4 \| \nu / \mu \| \leq 4M.
	\end{equation*}
%	where the second equality holds $\mu$-almost everywhere.
%	Let $X_A$ be a random variable with distribution $\nu_A$ that equals $X$ on $\{ X \in A \}$. By assumption, $\| Q^{\ell(M/q^2)}(X_A, \cdot) - \mu \| = \| Q^{\ell(4M/(2q)^2)}(X_A, \cdot) - \mu \| \leq 2q$.
%	
%	Note that $\mathbb{P} \{ Y \in B | X \in A \}$ can be computed by evaluating $\{ Y\in B \}$ in the joint probability distribution of $X_A$ and $Y$.
	Therefore, $\| \nu_A Q^{\ell(M/q^2)} - \mu \| = \| \nu_A Q^{\ell(4M/(2q)^2)} - \mu \| \leq 2q$ by assumption. Since $\| \nu / \mu \| \leq M$, we also have $\| \nu Q^{\ell(M/q^2)} - \mu \| \leq q$ by assumption.
	Hence, by combining \eqref{eq:ProofNearIndependenceWalk1} and \eqref{eq:ProofNearIndependenceWalk2} with the triangle inequality and Definition \ref{def:TotalVariationDistance}, it follows that
%	\begin{align*}
%		\Big| \mathbb{P} \{ Y \in B | X \in A \} - \mathbb{P} \{ Y \in B \} \Big|
%		&= \Big| \nu_A Q^{\ell(M/q^2)} (B) - \nu Q^{\ell(M/q^2)} (B) \Big|\\
%		&\leq \Big| \nu_A Q^{\ell(M/q^2)} (B) - \mu(B) \Big| + \Big| \mu(B) - \nu Q^{\ell(M/q^2)} (B) \Big| \\
%		&\leq \| Q^{\ell(M/q^2)}(X_A, \cdot) - \mu \| + \| Q^{\ell(M/q^2)}(X, \cdot) - \mu \|\\
%		&\leq 2q + q = 3q.
%	\end{align*}
	\begin{align*}
		\Big| \mathbb{P} \{ Y \in B | X \in A \} - \mathbb{P} \{ Y \in B \} \Big| &= \Big| \nu_A Q^{\ell(M/q^2)} (B) - \nu Q^{\ell(M/q^2)} (B) \Big|\\
		&\leq \Big| \nu_A Q^{\ell(M/q^2)} (B) - \mu(B) \Big| + \Big| \mu(B) - \nu Q^{\ell(M/q^2)} (B) \Big| \\
		&\leq \| \nu_A Q^{\ell(M/q^2)} - \mu \| + \| \nu Q^{\ell(M/q^2)} - \mu \|\\
		&\leq 2q + q = 3q.
	\end{align*}
	Multiplying both sides of the outermost inequality by $\mathbb{P} \{X \in A\}$ shows
	\begin{equation*}
		\Big| \mathbb{P} \{ Y \in B \wedge X \in A \} - \mathbb{P} \{ Y \in B \} \mathbb{P}\{X\in A \} \Big| \leq 3q  \mathbb{P}\{X\in A \} \leq 3q,
	\end{equation*}
	which completes the proof.
\end{proof}

Having shown that the start and end point of a random walk are near-independent, we continue by proving the near-independence of the result of two independent random walks with the same starting point.
\begin{lemma}
	Let $Y_1$ and $Y_2$ be random variables that are both $q$-independent of a random variable $X$. Assume that $Y_1$ and $Y_2$ are conditionally independent given $X$ and that for all measurable events $\{Y_1 \in A\}$ and $\{ Y_2 \in B \}$, the following sets are measurable:
	\begin{align*}
		&\{ x \in K : \mathbb{P}\{ Y_1 \in A | X = x \} \geq \mathbb{P}\{ Y_1 \in A\} \} \\
		&\{ x \in K : \mathbb{P}\{ Y_2 \in B | X = x \} \geq \mathbb{P}\{ Y_2 \in B\} \}.
	\end{align*}
%	and that $\{Z = z\}$ is measurable for all $z$ in the codomain $K$ of $Z$.
	Then, $Y_1$ and $Y_2$ are $2q$-independent.
	\label{lemma:NearIndependenceSample}
\end{lemma}
\begin{proof}[Proof of Lemma \ref{lemma:NearIndependenceSample}]
	Denote the probability distribution of $X$ by $\mu$.
	We want to bound the following term.
	\begin{align}
	&\Big|\mathbb{P}\{ Y_1 \in A \wedge Y_2 \in B \} - \mathbb{P}\{ Y_1 \in A\} \mathbb{P}\{ Y_2 \in B \} \Big| \nonumber\\
	&= \left| \int_K (\mathbb{P}\{ Y_1 \in A \wedge Y_2 \in B | X = x \} - \mathbb{P}\{ Y_1 \in A\} \mathbb{P}\{Y_2 \in B \}) \diff \mu(x) \right| \nonumber\\
	&= \left| \int_K (\mathbb{P}\{ Y_1 \in A | X = x \} \mathbb{P}\{ Y_2 \in B | X = x \} - \mathbb{P}\{ Y_1 \in A\} \mathbb{P}\{Y_2 \in B \})  \diff \mu(x) \right|,
	\label{eq:NearIndependenceProof1}
	\end{align}
	where the last equality holds by the conditional independence of $Y_1$ and $Y_2$.
	We will use the identity
	%	\begin{equation}
	$ab-cd = (a-c)(b-d) + (a - c)d + (b - d)c$,
	%		\label{eq:NearIndependenceProofIdentity}
	%	\end{equation}
	where $a,b,c,d \in \mathbb{R}$. Select
	\begin{equation*}
	a = \mathbb{P}\{ Y_1 \in A | X = x \}, \qquad b = \mathbb{P}\{ Y_2 \in B | X = x \}, \qquad c = \mathbb{P}\{ Y_1 \in A\}, \qquad d = \mathbb{P}\{Y_2 \in B \}.
	\end{equation*}
	This allows us to expand \eqref{eq:NearIndependenceProof1} in an obvious manner. The triangle inequality then gives
	\begin{align}
	&\Big|\mathbb{P}\{ Y_1 \in A \wedge Y_2 \in B \} - \mathbb{P}\{ Y_1 \in A\} \mathbb{P}\{Y_2 \in B \} \Big| \nonumber\\
	%		\begin{aligned}
	&\leq \left| \int_K (\mathbb{P}\{ Y_1 \in A | X = x \} - \mathbb{P}\{ Y_1 \in A\}) (\mathbb{P}\{ Y_2 \in B | X = x \} - \mathbb{P}\{Y_2 \in B \}) \diff \mu(x) \right| \nonumber \\
	& + \left| \int_K (\mathbb{P}\{ Y_1 \in A | X = x \} - \mathbb{P}\{ Y_1 \in A\}) \mathbb{P}\{Y_2 \in B \} \diff \mu(x) \right|
	\label{eq:NearIndependenceProofTriangle}\\
	& + \left| \int_K (\mathbb{P}\{ Y_2 \in B | X = x \} - \mathbb{P}\{ Y_2 \in B \}) \mathbb{P}\{Y_1 \in A \} \diff \mu(x) \right|. \nonumber
	%		\end{aligned}
	\end{align}
	We will upper bound each of these terms.
	
	For the first term, we can use H\"older's inequality as follows.
	\begin{align}
	& \left| \int_K (\mathbb{P}\{ Y_1 \in A | X = x \} - \mathbb{P}\{ Y_1 \in A\}) (\mathbb{P}\{ Y_2 \in B | X = x \} - \mathbb{P}\{Y_2 \in B \}) \diff \mu(x) \right| \nonumber\\
	&\leq \sqrt{ \int_K (\mathbb{P}\{ Y_1 \in A | X = x \} - \mathbb{P}\{ Y_1 \in A\})^2 \diff \mu(x) \int_K (\mathbb{P}\{ Y_2 \in B | X = x \} - \mathbb{P}\{Y_2 \in B \})^2 \diff \mu(x) }.
	\label{eq:NearIndependenceProofTerm1}
	\end{align}
	Define
	\begin{equation*}
		C := \{ x \in K : \mathbb{P}\{ Y_1 \in A | X = x \} \geq \mathbb{P}\{ Y_1 \in A\} \}.
	\end{equation*}
	Since both $\mathbb{P}\{ Y_1 \in A | X = x \}$ and $\mathbb{P}\{ Y_1 \in A\}$ lie in $[0,1]$, the square of their difference can be upper bounded by their absolute difference. Therefore,
	\begin{align*}
	&\int_K (\mathbb{P}\{ Y_1 \in A | X = x \} - \mathbb{P}\{ Y_1 \in A\})^2 \diff \mu(x)\\
	&\leq \int_K \Big|\mathbb{P}\{ Y_1 \in A | X = x \} - \mathbb{P}\{ Y_1 \in A\} \Big| \diff \mu(x)\\
	&= \int_C (\mathbb{P}\{ Y_1 \in A | X = x \} - \mathbb{P}\{ Y_1 \in A\}) \diff \mu(x)
	+ \int_{K \setminus C} ( \mathbb{P}\{ Y_1 \in A\} - \mathbb{P}\{ Y_1 \in A | X = x \}) \diff \mu(x)\\
	&= \mathbb{P}\{ Y_1 \in A \wedge X \in C \} - \mathbb{P}\{ Y_1 \in A\} \mathbb{P}\{ X \in C \}
	+  \mathbb{P}\{ Y_1 \in A\}\mathbb{P}\{ X \notin C \} - \mathbb{P}\{ Y_1 \in A \wedge X \notin C \}\\
	&\leq 2q,
	\end{align*}
	since $X$ and $Y_1$ are near-independent.
	Because the same holds for $\mathbb{P}\{ Y_2 \in B | X = x \}$ and $\mathbb{P}\{ Y_2 \in B\}$, \eqref{eq:NearIndependenceProofTerm1} is upper bounded by $2q$.
	
	For the second term in \eqref{eq:NearIndependenceProofTriangle}, observe that
	\begin{align*}
	&\left| \int_K (\mathbb{P}\{ Y_1 \in A | X = x \} - \mathbb{P}\{ Y_1 \in A\}) \mathbb{P}\{Y_2 \in B \} \diff \mu(x) \right|\\
	&= \mathbb{P}\{Y_2 \in B \} \left| \int_K ( \mathbb{P}\{ Y_1 \in A | X = x \} - \mathbb{P}\{ Y_1 \in A\})  \diff \mu(x) \right|\\
	&= \mathbb{P}\{Y_2 \in B \} \Big| \mathbb{P}\{ Y_1 \in A \} - \mathbb{P}\{ Y_1 \in A\} \Big| = 0.
	\end{align*}
	The same clearly holds for the third term in \eqref{eq:NearIndependenceProofTriangle}. Hence,
	\begin{equation*}
	\Big|\mathbb{P}\{ Y_1 \in A \wedge Y_2 \in B \} - \mathbb{P}\{ Y_1 \in A\} \mathbb{P}\{Y_2 \in B \} \Big| \leq 2q.
	\end{equation*}
\end{proof}

For near-independent vector-valued random variables, the products of some of the entries in the respective vectors are also near-independent, as the following lemma shows.
\begin{lemma}
	\label{lemma:NearIndependentProduct}
	Let $X = (X_1, ..., X_n)$ and $Y = (Y_1, ..., Y_n)$ be $q$-independent random variables with values in $\mathbb{R}^n$, and let $\mathcal{S} \subseteq \{1,...,n\}$. Suppose the function $x = (x_1, ..., x_n) \mapsto \prod_{i \in \mathcal{S}} x_i$ is measurable. Then, the random variables $\prod_{i \in \mathcal{S}} X_i$ and $\prod_{i \in \mathcal{S}} Y_i$ are $q$-independent.
\end{lemma}
\begin{proof}
	The result follows from Lemma 3.5 in Lov\'asz and Vempala \cite{lovasz2006simulated} applied to $x \mapsto \prod_{i \in \mathcal{S}} x_i$.
%	Let $A$ and $B$ be measurable sets, and define $A_{\mathcal{S}} := \{\prod_{i \in \mathcal{S}} x_i : x \in A\}$.
%	Then,
%	\begin{align*}
%		q &\geq \Big| \mathbb{P}\{ X \in A \wedge Y \in B \} - \mathbb{P}\{ X \in A \} \mathbb{P}\{Y \in B\} \Big|\\
%		&\geq \left| \mathbb{P}\left\{ \prod_{i \in \mathcal{S}} X_i \in A_{\mathcal{S}} \wedge \prod_{i \in \mathcal{S}} Y_i \in B_{\mathcal{S}} \right\} - \mathbb{P}\left\{ \prod_{i \in \mathcal{S}} X_i \in A_{\mathcal{S}} \right\} \mathbb{P} \left\{\prod_{i \in \mathcal{S}} Y_i \in B_{\mathcal{S}} \right\} \right|,
%	\end{align*}
%	because the map $A \mapsto A_{\mathcal{S}}$ is surjective.
\end{proof}
The measurability conditions in Lemmas \ref{lemma:NearIndependenceSample} and \ref{lemma:NearIndependentProduct} are satisfied for sufficiently detailed $\sigma$-algebras. We assume these conditions to hold in the remainder of this paper.

To close this section, we cite the following result on the expectation of the product of near-independent real-valued random variables.
\begin{lemma}[Lemma 2.7 from \cite{kannan1997random}]
	\label{lemma:NearIndependenceExpectation}
	Let $X$ and $Y$ be $q$-independent random variables such that $|X|\leq a$ and $|Y| \leq b$. Then
	\begin{equation*}
		|\mathbb{E}[XY] - \mathbb{E}[X] \mathbb{E}[Y]| \leq 4 q a b.
	\end{equation*}
\end{lemma}

\section{Entropic Barrier Properties}
\label{sec:EntropicBarrier}
The self-concordance of $f$ and $f^*$ can be used to show two results that we will need for the analysis of an interior point method that uses the entropic barrier. First, it will turn out that we will need a lower bound on $\| \theta \|_\theta$ for all $\theta \in \mathbb{R}^n$, which requires an investigation of the spectrum (with respect to the Euclidean inner product) of the covariance matrix of a Boltzmann distribution. Second, we will show that if $x, y \in \interior K$ are close, then $\theta(x)$ and $\theta(y)$ are also close, in a well-defined sense.

\subsection{Spectra of Boltzmann Covariance Matrices}
To analyze the spectra of the Boltzmann covariance matrices, we will need information about the spectrum of the covariance matrix of the uniform distribution. We will denote the smallest and largest eigenvalue of a self-adjoint linear operator $A$ with respect to the reference inner product by $\lambda_{\min}(A)$ and $\lambda_{\max}(A)$. Recall that for positive semidefinite linear operators $A$, we have $\lambda_{\min}(A) = \min_{v: \| v \| = 1}  \langle v, Av \rangle$ and $\lambda_{\max}(A) = \max_{v: \| v \| = 1} \langle v, Av \rangle = \| A \|$.
%, we have for all $x \in \mathbb{R}^n$ that
%\begin{equation}
%	\label{eq:EigenvaluesOperatorNorm}
%	\lambda_{\min}(S) \|x\| \leq \| Sx \| \leq \lambda_{\max}(S) \|x\|.
%\end{equation}

One should note that an upper bound of the spectrum of $\Sigma(\theta)$ is trivial to derive for any $\theta \in \mathbb{R}^n$. If $K$ is contained in a ball with radius $R$, i.e. the diameter of $K$ is at most $2R$,
\begin{equation}
	\lambda_{\max}(\Sigma(\theta)) = \max_{v: \| v \| = 1} \mathbb{E}_\theta[\langle X - \mathbb{E}_\theta[X], v \rangle^2] \leq (2R)^2,
\label{eq:UpperBoundSigma}
\end{equation}
where the equality uses \eqref{eq:DerivativesF}.
Thus, we will focus on bounding the smallest eigenvalue of $\Sigma(\theta)$ from below. Our starting point is the following result from Kannan, Lov\'asz and Simonovits \cite{kannan1995isoperimetric}.

\begin{lemma}[Theorem 4.1 in \cite{kannan1995isoperimetric}]
	\label{lemma:BallsIsotropicUniform}
	Let $K \subseteq \mathbb{R}^n$ be a convex body, and recall that $\Sigma(0)$ denotes the covariance matrix of the uniform distribution over $K$. If $\Sigma(0) = I$, then $K$
%	strictly contains a Euclidean ball with radius $1$ and
	is contained in a Euclidean ball with radius $n+1$.
\end{lemma}

We can use this result to bound the spectrum of $\Sigma(0)$ from below.
\begin{lemma}
	\label{lemma:SpectrumUniform}
	Let $K \subseteq \mathbb{R}^n$ be a convex body that contains a Euclidean ball of radius $r$. Then, $\lambda_{\min}(\Sigma(0)) \geq \frac{1}{4} (\frac{r}{n+1})^2$.
\end{lemma}
\begin{proof}
	The convex body $K' = \Sigma(0)^{-1/2} K$ has the property that the uniform distribution over $K'$ has identity covariance. By Lemma \ref{lemma:BallsIsotropicUniform}, $K'$ is contained in a ball of radius $n+1$.
	
	Let $x \in K$ be the center of the ball with radius $r$ contained in $K$, and let $v$ be a unit vector such that $\Sigma(0)^{-1/2} v = \lambda_{\max}(\Sigma(0)^{-1/2}) v$. Since $v$ is a unit vector, the point $x + r v$ lies in $K$. Because $\Sigma(0)^{-1/2}x$ and $\Sigma(0)^{-1/2}(x+rv)$ lie in $K'$, we find $\| \Sigma(0)^{-1/2} ((x + rv) - x) \| \leq 2(n+1)$, where $2(n+1)$ is the diameter of a ball containing $K'$. In conclusion,
	\begin{equation*}
		2(n+1) \geq \| r \Sigma(0)^{-1/2} v \| = r \lambda_{\max}(\Sigma(0)^{-1/2}) = \frac{r}{\sqrt{\lambda_{\min}(\Sigma(0))}},
	\end{equation*}
	which proves $\lambda_{\min}(\Sigma(0)) \geq \frac{1}{4} (\frac{r}{n+1})^2$.
%	Thus, we can always find a point $\Sigma(0)^{-1/2} y \in K'$ such that $\| \Sigma(0)^{-1/2}(x_R - y) \| \geq 1$. It then follows from \eqref{eq:EigenvaluesOperatorNorm} that
%	\begin{equation*}
%		1 \leq \| \Sigma(0)^{-1/2}(x_R - y) \| \leq \lambda_{\max}(\Sigma(0)^{-1/2}) \| x_R - y \| \leq \lambda_{\max}(\Sigma(0)^{-1/2}) R,
%	\end{equation*}
%	which implies that $\lambda_{\max}(\Sigma(0)^{-1/2}) \geq \frac{1}{R}$. Hence, $\lambda_{\min}(\Sigma(0)) \leq R^2$.
%	
%	Similarly, since $K'$ is contained in a Euclidean ball of radius $n+1$, we can find a point $\Sigma(0)^{-1/2} z \notin K'$ such that $\| \Sigma(0)^{-1/2}(x - z) \| \leq n+1$. Again using \eqref{eq:EigenvaluesOperatorNorm},
%	\begin{equation*}
%		n+1 \geq \| \Sigma(0)^{-1/2}(x - z) \| \geq \lambda_{\min}(\Sigma(0)^{-1/2}) \| x - z \| \geq \lambda_{\min}(\Sigma(0)^{-1/2}) r,
%	\end{equation*}
%	showing that $\lambda_{\max}(\Sigma(0)) \geq (\frac{r}{n+1})^2$.
\end{proof}

%On a side note, it follows from Corollary 2.3.5 in Renegar \cite{renegar2001mathematical} that any convex body $K$ contains a Euclidean ball of radius $\sqrt{\lambda_{\min}(\Sigma(0))}$ and is contained in a Euclidean ball of radius $(4\vartheta + 1)\sqrt{\lambda_{\max}(\Sigma(0))}$.

With the spectrum of the uniform covariance matrix bounded, we can continue to analyze $\Sigma(\theta)$, where $\theta \in \mathbb{R}^n$.
Using Lemma \ref{lemma:PropertiesDerivativesConjugate}, we get for every $\theta \in \mathbb{R}^n$,
%\begin{align}
%	\| \theta \|_{\theta}^2 &= \langle g^*(g(\theta)), H(\theta) g^*(g(\theta)) \rangle \nonumber\\
%	&= \langle g^*_{g(\theta)}(g(\theta)), H^*(g(\theta)) g^*_{g(\theta)}(g(\theta)) \rangle = (\| g^*_{g(\theta)}(g(\theta)) \|^*_{g(\theta)})^2 \leq \vartheta,
%	\label{eq:UpperBoundLocalTheta}
%\end{align}
\begin{equation}
	\| \theta \|_{\theta}^2 = \langle g^*(g(\theta)), H(\theta) g^*(g(\theta)) \rangle = \langle g^*_{g(\theta)}(g(\theta)), H^*(g(\theta)) g^*_{g(\theta)}(g(\theta)) \rangle = (\| g^*_{g(\theta)}(g(\theta)) \|^*_{g(\theta)})^2 \leq \vartheta,
	\label{eq:UpperBoundLocalTheta}
\end{equation}
where the inequality follows from Lemma \ref{lemma:InteriorKImageG} and the definition of the complexity parameter $\vartheta$ from Renegar \cite{renegar2001mathematical} (see Definition \ref{def:Barrier}).
With this inequality, we can now prove a bound on the smallest eigenvalue of $\Sigma(\theta)$ for all $\theta \in \mathbb{R}^n$.
\begin{theorem}
	\label{thm:LowerBoundSigma}
	Let $K \subseteq \mathbb{R}^n$ be a convex body that contains a Euclidean ball of radius $r$ and is contained in a Euclidean ball of radius $R$.
	Define $f$ as the log partition function $f(\theta) = \ln \int_{K} e^{\langle \theta, x \rangle} \diff x$, where $\langle \cdot, \cdot \rangle$ is the Euclidean inner product, and denote its Hessian by $\Sigma(\theta)$. Let $f^*$ be the entropic barrier for $K$ with complexity parameter $\vartheta$.
	Let $\lambda_{\min}(\Sigma(\theta))$ be the smallest eigenvalue of $\Sigma(\theta)$. Then, for any $\theta \in \mathbb{R}^n$ with $\| \theta \| \leq \frac{1}{4R}$,
	\begin{equation*}
		\lambda_{\min}(\Sigma(\theta)) \geq \frac{1}{16} \left(\frac{r}{n+1} \right)^2,
	\end{equation*}
	and for all $\theta \in \mathbb{R}^n$ with $\| \theta \| > \frac{1}{4R}$,
	\begin{equation*}
		\lambda_{\min}(\Sigma(\theta)) \geq \frac{1}{64} \left( \frac{1}{4 R \| \theta \|} \right)^{4 \sqrt{\vartheta} + 2} \left(\frac{r}{n+1} \right)^2.
	\end{equation*}
\end{theorem}
\begin{proof}
	We want to find a lower bound on $\| v \|_\theta$, where $\| v \| = 1$. The idea is to use the self-concordance properties of $f$ to move from $\Sigma(\theta)$ to the covariance matrix of the uniform distribution, and then apply Lemma \ref{lemma:SpectrumUniform}.
	
	If $\| \theta \| \leq \frac{1}{4R}$, then \eqref{eq:UpperBoundSigma} shows
	\begin{equation*}
		\| \theta - 0 \|_0 \leq \sqrt{\lambda_{\max}(\Sigma(0))} \| \theta \| \leq 2R \| \theta \| \leq \frac{1}{2} < 1,
	\end{equation*}
	and thus we may apply the first inequality in \eqref{eq:SCf} and Lemma \ref{lemma:SpectrumUniform} to show that
	\begin{equation}
		\label{eq:ProofLowerBoundSigma1}
		\| v \|_\theta \geq \| v \|_0 (1 - \| \theta - 0 \|_0) \geq \frac{1}{2} \sqrt{\lambda_{\min}(\Sigma(0))} \| v \| \geq \frac{1}{4} \frac{r}{n+1} \| v \|.
	\end{equation}
	It then follows from \eqref{eq:ProofLowerBoundSigma1} that
	\begin{equation*}
		\lambda_{\min}(\Sigma(\theta)) = \min_{v: \| v \| = 1} \| v \|_\theta^2 \geq \frac{1}{16} \left(\frac{r}{n+1} \right)^2.
	\end{equation*}
	
	Next, suppose that $\| \theta \| > \frac{1}{4R}$. Let $\theta_0 = \theta$ and recursively define
	$\theta_k = (1 - \frac{1}{2\sqrt{\vartheta} + 1}) \theta_{k-1}$.
	Observe that by \eqref{eq:UpperBoundLocalTheta}, for all $k$,
	\begin{equation*}
		\| \theta_{k-1} - \theta_k \|_{\theta_k} = \frac{\| \theta_{k-1} \|_{\theta_k}}{2\sqrt{\vartheta} + 1} = \frac{\| \theta_k \|_{\theta_k}}{2\sqrt{\vartheta}} \leq \frac{\sqrt{\vartheta}}{2 \sqrt{\vartheta}} = \frac{1}{2} < 1.
	\end{equation*}
	Since $\theta_k$ and $\theta_{k-1}$ are close in the sense above, we can apply self-concordance. By the first inequality of \eqref{eq:SCf}, for all $k$,
%	\begin{equation*}
		$\| v \|_{\theta_{k-1}} \geq \left(1 - \| \theta_{k-1} - \theta_k \|_{\theta_k} \right) \| v \|_{\theta_k} \geq \tfrac{1}{2} \| v \|_{\theta_k}$.
%	\end{equation*}
	Thus, after $m$ steps, we have
	\begin{equation}
		\| v \|_\theta = \| v \|_{\theta_0} \geq 2^{-m} \| v \|_{\theta_m}.
		\label{eq:ProofLowerBoundSigma1.5}
	\end{equation}
	Setting
	\begin{equation*}
		m = \left\lceil \frac{\log_2 (\frac{1}{4 R \| \theta \|})}{\log_2(1 - \frac{1}{2\sqrt{\vartheta} + 1})} \right\rceil,
	\end{equation*}
	we obtain
	\begin{equation*}
		 \| \theta_m \| = \left(1 - \frac{1}{2\sqrt{\vartheta} + 1}\right)^m \| \theta_0 \| \leq \frac{1}{4 R \| \theta \|} \| \theta_0 \| = \frac{1}{4R}.
	\end{equation*}
	We may now apply \eqref{eq:ProofLowerBoundSigma1} to see that $\| v \|_{\theta_m} \geq \frac{1}{4} \frac{r}{n+1} \| v \|$. Combined with \eqref{eq:ProofLowerBoundSigma1.5}, it follows that
%	Note that for this particular value of $m$,
%	\begin{equation*}
%		\| \theta_m - 0 \|_0 \leq \sqrt{\lambda_{\max}(\Sigma(0))} \| \theta_m \| \leq 2R \frac{1}{4R} = \frac{1}{2} < 1,
%	\end{equation*}
%	and thus we may again apply self-concordance and Lemma \ref{lemma:SpectrumUniform} to show that
	\begin{equation}
		\| v \|_\theta \geq 2^{-m} \| v \|_{\theta_m} \geq \frac{2^{-m}}{4} \frac{r}{n+1} \| v \| = 2^{-m-2} \frac{r}{n+1} \| v \|.
		\label{eq:ProofLowerBoundSigma2}
	\end{equation}
	Because $m$ is an integer, we arrive at the following lower bound for $2^{-m-2}$:
	\begin{equation*}
	2^{-m-2} \geq \frac{1}{8} \left( 4 R \| \theta \| \right)^{1/\log_2(1 - \frac{1}{2\sqrt{\vartheta} + 1})}.
	\end{equation*}
	Since $4R \|\theta \| > 1$ by assumption, and $1/\log_2(1-t) \geq -1/t$ for all $t \in (0,1)$, this bound can be developed to
	\begin{equation*}
		2^{-m-2} \geq \frac{1}{8} \left( 4 R \| \theta \| \right)^{1/\log_2(1 - \frac{1}{2\sqrt{\vartheta} + 1})} \geq \frac{1}{8} \left( 4 R \| \theta \| \right)^{-2 \sqrt{\vartheta} - 1},
	\end{equation*}
	and we can conclude from \eqref{eq:ProofLowerBoundSigma2} that
	\begin{equation*}
		\lambda_{\min}(\Sigma(\theta)) = \min_{v: \| v \| = 1} \| v \|_\theta^2 \geq \frac{1}{64} \left( \frac{1}{4 R \| \theta \|} \right)^{4 \sqrt{\vartheta} + 2} \left(\frac{r}{n+1} \right)^2.
	\end{equation*}
\end{proof}

Note that this lower bound is exponential in $\vartheta = n + o(n)$.
For our analysis in Section \ref{sec:ShortStepIPM}, we will need a stronger lower bound on $\| \theta \|_\theta = \sqrt{ \langle \theta, \Sigma(\theta) \theta \rangle }$ than the one obtained from Theorem \ref{thm:LowerBoundSigma} by setting $v = \theta / \|\theta\|$. The following lemma gives such a lower bound that is not exponential in $n$.
%The lower bound in Lemma \ref{lemma:LowerBoundSigma} depends exponentially on $\vartheta = O(n)$. However, we can improve upon this result if one considers $\| \theta \|_\theta$ rather than $\| v\|_\theta$ for any $v \in \mathbb{R}^n$.
\begin{lemma}
	\label{lemma:LowerBoundTheta}
	Let $K \subseteq \mathbb{R}^n$ be a convex body that contains a Euclidean ball of radius $r$, and let $\langle \cdot, \cdot \rangle$ be the Euclidean inner product.
	Then, it holds for all $\theta \in \mathbb{R}^n$ that
	\begin{equation}
		\label{eq:LowerBoundThetaLocalNorm}
		\| \theta \|_\theta \geq \frac{r \| \theta \|}{2(n+1) + r \| \theta \|}.
	\end{equation}
\end{lemma}
\begin{proof}
	Note that the right hand side of \eqref{eq:LowerBoundThetaLocalNorm} is always strictly smaller than one. The claim therefore holds automatically for all $\theta$ with $\| \theta \|_\theta \geq 1$, and we can assume in the remainder that $\| \theta \|_\theta < 1$.
	By Lemma \ref{lemma:SpectrumUniform}, we have that $\lambda_{\min}(\Sigma(0)) \geq \frac{1}{4}(\frac{r}{n+1})^2$.
%	It then follows that $\| \theta \|_0 \geq \frac{r}{n+1} \| \theta \|_2$.
%	Note that to complete the proof, we only need to consider the $\theta \in \mathbb{R}^n$ with .
	The second inequality in \eqref{eq:SCf} then gives us
	\begin{equation*}
		\frac{1}{2}\frac{r}{n+1} \| \theta \| \leq \| \theta \|_0 \leq \frac{\| \theta \|_{\theta}}{1 - \|\theta \|_\theta},
	\end{equation*}
	or equivalently,
	\begin{equation*}
		\| \theta \|_\theta \geq \frac{\frac{1}{2}\frac{r}{n+1} \| \theta \|}{1 + \frac{1}{2}\frac{r}{n+1} \| \theta \|} = \frac{r \| \theta \|}{2(n+1) + r \| \theta \|}.
	\end{equation*}
\end{proof}

%One should note that an upper bound of the spectrum of $\Sigma(\theta)$ is much easier to derive. Under the same assumptions as in Lemma \ref{lemma:LowerBoundSigma},

\subsection{Parameter Proximity}
%One might expect that if $x_k$ is a good approximation of $z(\eta_k)$, then $\theta(z(\eta_k))$ and $\theta(x_k)$ are also not too different. The following lemma shows that this is indeed the case.
%An illustration is provided in Figure \ref{fig:ProximityTheta}.
Next, we show that if $x, y \in \interior K$ are ``close'', then so are $\theta(x)$ and $\theta(y)$, and vice versa.
\begin{lemma}
	\label{lemma:ProximityThetaGeneral}
	Let $K$ be a convex body, and let $x, y, z \in \interior K$. If $\| \theta(x) - \theta(y) \|_{\theta(z)} + \|\theta(y) - \theta(z)\|_{\theta(z)} < 1$, then
	\begin{equation*}
		\| x - y \|_z^* \leq \frac{1}{1 - \| \theta(y) - \theta(z) \|_{\theta(z)}} \left(\frac{ \| \theta(x) - \theta(y) \|_{\theta(z)}}{1 - \| \theta(x) - \theta(y) \|_{\theta(z)} - \|\theta(y) - \theta(z)\|_{\theta(z)}  } \right).
	\end{equation*}
	Similarly, if $\| x-y \|_z^* + \| y-z \|_z^* < 1$, then
	\begin{equation*}
		\| \theta(x) - \theta(y) \|_{\theta(z)} \leq \frac{1}{1 - \| y-z \|_{z}^*} \left(\frac{ \| x-y \|_{z}^*}{1 - \| x-y \|_{z}^* - \|y-z\|_{z}^* } \right).
	\end{equation*}
\end{lemma}
\begin{proof}
	We have
	\begin{align}
%		\begin{aligned}
			\| x - y \|_{z}^*
			&= \| g(\theta(x)) - g(\theta(y)) \|_{H^*(z)} \nonumber \\
			&= \| g(\theta(x)) - g(\theta(y)) \|_{H(\theta(z))^{-1}} \nonumber \\
			&= \| g_{\theta(z)}(\theta(x)) - g_{\theta(z)}(\theta(y)) \|_{H(\theta(z))} \nonumber \\
			&= \left\| \int_0^1 H_{\theta(z)}(\theta(y) + t[\theta(x) - \theta(y)]) [\theta(x) -\theta(y) ] \diff t \right\|_{\theta(z)},
%		\end{aligned}
		\label{eq:LemmaProximityTheta2First}
	\end{align}
	by the fundamental theorem of calculus (see e.g. Theorem 1.5.6 in Renegar \cite{renegar2001mathematical}). We have the following upper bound on \eqref{eq:LemmaProximityTheta2First}:
	\begin{align}
%		\begin{aligned}
			&\left\| \int_0^1 H_{\theta(z)}(\theta(y) + t[\theta(x) - \theta(y)]) [\theta(x) -\theta(y) ] \diff t \right\|_{\theta(z)} \nonumber\\
			&\leq \| \theta(x) - \theta(y) \|_{\theta(z)} \max_{u \in \mathbb{R}^n} \frac{\langle u, \int_0^1 H_{\theta(z)}(\theta(y) + t[\theta(x) - \theta(y)])u \diff t \rangle_{\theta(z)}}{\| u \|^2_{\theta(z)}} \nonumber\\
			&\leq \| \theta(x) - \theta(y) \|_{\theta(z)} \int_0^1 \max_{u \in \mathbb{R}^n} \frac{\langle u,  H_{\theta(z)}(\theta(y) + t[\theta(x) - \theta(y)])u \rangle_{\theta(z)}}{\| u \|^2_{\theta(z)}} \diff t \nonumber \\
			&= \| \theta(x) - \theta(y) \|_{\theta(z)} \int_0^1 \max_{u \in \mathbb{R}^n} \frac{\| u \|^2_{\theta(y) + t[\theta(x) - \theta(y)]} }{\| u \|^2_{\theta(z)}} \diff t.
%		\end{aligned}
		\label{eq:LemmaProximityTheta2Second}
	\end{align}
	
	If $\| \theta(x) - \theta(z) \|_{\theta(z)} < 1$ and $\| \theta(y) - \theta(z) \|_{\theta(z)} < 1$, then by the triangle inequality,
	\begin{equation*}
		\|\theta(z) - \theta(y) - t [\theta(x) - \theta(y)] \|_{\theta(z)} = \| t [\theta(z) - \theta(x) ] + (1-t) [\theta(z) - \theta(y)] \|_{\theta(z)} < 1,
	\end{equation*}
	for all $t \in [0,1]$. We can therefore apply the second inequality of \eqref{eq:SCf} as follows:
	\begin{align}
%		\begin{aligned}
			\int_0^1 \max_{u \in \mathbb{R}^n} \frac{\| u \|^2_{\theta(y) + t[\theta(x) - \theta(y)]} }{\| u \|^2_{\theta(z)}} \diff t
			&\leq \int_0^1 \left(\frac{1}{ 1 - \|\theta(z) - \theta(y) - t [\theta(x) - \theta(y)] \|_{\theta(z)}} \right)^2 \diff t \nonumber\\
			&\leq \int_0^1 \left(\frac{1}{ 1 - \|\theta(z) - \theta(y)\|_{\theta(z)} - t \| \theta(x) - \theta(y) \|_{\theta(z)}} \right)^2 \diff t \nonumber\\
			&= \frac{1}{1 - \|\theta(z) - \theta(y)\|_{\theta(z)}} \left(\frac{1}{1 -  \|\theta(z) - \theta(y)\|_{\theta(z)} - \| \theta(x) - \theta(y) \|_{\theta(z)}}\right).
%		\end{aligned}
		\label{eq:LemmaProximityTheta2Third}
	\end{align}
	The upper bound on $\| x-z\|_z^*$ thus follows from combining \eqref{eq:LemmaProximityTheta2First}, \eqref{eq:LemmaProximityTheta2Second} and \eqref{eq:LemmaProximityTheta2Third}.
%	Since $\|\theta(x) - \theta(z) \|_{\theta(z)} < 1$ and $\| x-z\|_z^* \geq 0$, this inequality can be rewritten as a lower bound on $\|\theta(x) - \theta(z) \|_{\theta(z)}$.
%	
	The upper bound on $\| \theta(x) - \theta(z) \|_{\theta(z)}$ can be derived in the same manner as the above by interchanging $x$ and $\theta(x)$, $y$ and $\theta(y)$, $z$ and $\theta(z)$, and $f$ and $f^*$.
\end{proof}
We will not always need this general lemma with three points $x$, $y$ and $z$. For easy reference, we will state the following corollary that only considers two points $x$ and $z$.
\begin{corollary}
	\label{cor:ProximityThetaTwoPoints}
	Let $K$ be a convex body, and let $x, z \in \interior K$. If $\| \theta(x) - \theta(z) \|_{\theta(z)} < 1$, then
	\begin{equation}
		\| x - z \|_z^* \leq \frac{\| \theta(x) - \theta(z) \|_{\theta(z)}}{1 - \| \theta(x) - \theta(z) \|_{\theta(z)}} \text{\quad and \quad } \frac{\| x - z \|_z^*}{1 + \| x - z \|_z^*} \leq \| \theta(x) - \theta(z) \|_{\theta(z)}.
		\label{eq:ProximityThetaTwoPointsX}
%		\| x - z \|_z^* \geq \| \theta(x) - \theta(z) \|_{\theta(z)} \left( 1 - \| \theta(x) - \theta(z) \|_{\theta(z)} + \frac{\| \theta(x) - \theta(z) \|_{\theta(z)}^2}{3} \right),
	\end{equation}
%	and if $\| \theta(x) - \theta(z) \|_{\theta(z)} \geq 1$, then
%	\begin{equation*}
%		$\| x - z \|_z^* \geq \tfrac{1}{3}$.
%	\end{equation*}
	Similarly, if $\| x - z \|_z^* < 1$, then
	\begin{equation}
		\label{eq:ProximityThetaTwoPointsTheta}
		\| \theta(x) - \theta(z) \|_{\theta(z)} \leq \frac{\| x - z \|_z^*}{1 - \| x - z \|_z^*} \text{\quad and \quad } \frac{\| \theta(x) - \theta(z) \|_{\theta(z)}}{1 + \| \theta(x) - \theta(z) \|_{\theta(z)}} \leq \| x - z \|_z^*.
%		\| \theta(x) - \theta(z) \|_{\theta(z)} \geq \| x - z \|_z^* \left( 1 - \| x - z \|_z^* + \frac{(\| x - z \|_z^*)^2}{3} \right),
	\end{equation}
%	and if $\| x - z \|_z^* \geq 1$, then
%	\begin{equation*}
%		$\| \theta(x) - \theta(z) \|_{\theta(z)} \geq \tfrac{1}{3}$.
%	\end{equation*}
\end{corollary}
\begin{proof}
	Substitution of $y=z$ in Lemma \ref{lemma:ProximityThetaGeneral} gives the upper bounds on $\| x-z\|_z^*$ and $\| \theta(x) - \theta(z) \|_{\theta(z)}$. These can be rewritten as lower bounds on $\| \theta(x) - \theta(z) \|_{\theta(z)}$ and $\| x-z\|_z^*$, respectively.
\end{proof}

\section{Hit-and-Run Sampling}
\label{sec:HitandRun}

The procedure we will use to generate samples is called \emph{hit-and-run} sampling. This routine was introduced for the uniform distribution by Smith \cite{smith1984efficient} and later generalized to absolutely continuous distributions (see for example B\'elisle et al. \cite{belisle1993hit}). We will use the version in Algorithm \ref{alg:HitAndRun}, based on Lov\'asz and Vempala \cite{lovasz2006fast}.

\begin{algorithm}[!ht]
	\begin{algorithmic}[1]
		%		\Function{HitAndRun}{$g, \mathcal{O}_K, V, x, k $}
		\Require probability density $h: \mathbb{R}^n \to \mathbb{R}_+$ with respect to the Lebesgue measure of the distribution to sample from (i.e. the target distribution); covariance matrix $\Sigma \in \mathbb{R}^{n \times n}$; starting point $x \in K$; number of hit-and-run steps $\ell \in \mathbb{N}$.
		\Statex
		
		\State $X_0 \gets x$
		\State Sample directions $D_1, ..., D_\ell$ i.i.d. from a $\mathcal{N}(0, \Sigma)$-distribution
		\State Sample $P_1, ..., P_\ell$ i.i.d. from a uniform distribution over $[0,1]$, independent from $D_1, ..., D_\ell$
		\For{$i \in \{1, ..., \ell \}$}
%		\State Sample direction $D_i$ from a $\mathcal{N}(0, \Sigma)$-distribution
%		(independent from all $D_j$ and $X_j - X_{j-1}$, $j \in \{1, ..., i-1, i+1, ... \ell\}$)}
		\State Determine the two end points $Y_i$ and $Z_i$ of the line segment $K \cap \{X_{i-1} + t D_i : t \in \mathbb{R} \}$
		\State Determine $s \in [0,1]$ such that $\int_0^s h(Y_i + t(Z_i-Y_i)) \diff t = P_i \int_0^1 h(Y_i + t(Z_i-Y_i)) \diff t$
		\State $X_i \gets Y_i + s(Z_i-Y_i)$
%		\State Sample $X_i$ from the density $h$ restricted to $K \cap \{X_{i-1} + t D_i : t \in \mathbb{R} \}$
		\label{line:SampleLineSegment}
		\EndFor
		\State \Return $X_\ell$
		%		\EndFunction
	\end{algorithmic}
	\caption{The hit-and-run sampling procedure}
	\label{alg:HitAndRun}
\end{algorithm}

This procedure samples a random direction $D_i$ from a normal distribution, and samples the next iterate $X_i$ from the desired distribution restricted to the line through $X_{i-1}$ in the direction $D_i$, intersected with $K$.
%repeats the following $\ell$ times: draw a random direction $D \in \mathbb{R}^n$ from a suitable distribution, and let $L_D$ be the line through the current point in the direction $D$. Then, sample the next point from the desired distribution restricted to $L_D \cap K$.
Effectively, this reduces a high-dimensional sampling problem to a sequence of one-dimensional sampling problems.

%For any fixed covariance matrix $\Sigma$ and probability density $h$, consider a hit-and-run random walk starting at $x \in K$ with $\ell$ steps. We will make the (natural) assumption that the distribution of the end point of this random walk is given by $Q^\ell(x, \cdot)$, where $Q$ is a Markov kernel associated with a hit-and-run step. Note that $Q$ depends on the choice of $\Sigma$ and $h$.

The following theorem from Lov\'asz and Vempala \cite{lovasz2006fast} is the starting point of our analysis.
\begin{theorem}[Theorem 1.1 in \cite{lovasz2006fast}]
	\label{thm:MixingTime}
	Let $\mu$ be a log-concave probability distribution supported on a convex body $K \subseteq \mathbb{R}^n$, and let $q > 0$.
%	Let $\mu$ be the probability distribution with density proportional to $h$, and let $\delta > 0$.
	Consider a hit-and-run random walk as in Algorithm \ref{alg:HitAndRun} with respect to the target distribution $\mu$ from a random starting point with distribution $\nu$ supported on $K$.
	Assume that the following holds:
	\begin{enumerate}[(i)]
		\item the level set of $\mu$ with probability $\frac{1}{8}$ (see Definition \ref{def:LevelSet}) contains a ball of radius $s$ with respect to $\| \cdot \|$;
		\item \label{con:RDDerivativeBound} $\frac{\diff \nu}{\diff \mu}(x) \leq M'$ for all $x \in K \setminus A$ for some set $A \subseteq K$ with $\nu(A) \leq q$;
		\item $\mathbb{E}_\mu[\| X - \mathbb{E}_\mu[X]\|^2] \leq S^2$.
	\end{enumerate}
	Let $\nu^{(\ell)}$ be the distribution of the current hit-and-run point after $\ell$ steps of hit-and-run sampling applied to $\mu$, where the directions are chosen from a $\mathcal{N}(0,I)$-distribution. Then, after
	\begin{equation*}
	\ell = \left\lceil 10^{30} \frac{n^2 S^2}{s^2} \ln^2 \left( \frac{2M'nS}{sq} \right) \ln^3 \left( \frac{2M'}{q} \right) \right\rceil,
	\end{equation*}
	hit-and-run steps, we have $\| \nu^{(\ell)} - \mu \| \leq q$.
\end{theorem}
Suppose that rather than \eqref{con:RDDerivativeBound}, we know $\| \nu / \mu \| \leq M$, i.e. $\int_K \frac{\diff \nu}{\diff \mu}(x) \diff \nu(x) \leq M$. If $A = \{x \in K: \frac{\diff \nu}{\diff \mu}(x) > M / q \}$, then
\begin{equation*}
M \geq \int_A \frac{\diff \nu}{\diff \mu}(x) \diff \nu(x) \geq \frac{M}{q} \nu(A),
\end{equation*}
and thus we have $\nu(A) \leq q$. (This construction was also applied by Lov\'asz and Vempala \cite[page 10]{lovasz2006hit}.) We can therefore set $M' = M / q$ in the theorem. If one additionally considers a transformation $x \mapsto \Sigma^{-1/2} x$ for some invertible matrix $\Sigma$ applied to $K$, before Theorem \ref{thm:MixingTime} is applied, we arrive at the following corollary.
\begin{corollary}
	\label{cor:MixingTime}
	Let $\mu$ be a log-concave probability distribution supported on a convex body $K \subseteq \mathbb{R}^n$, and let $q > 0$. Consider a hit-and-run random walk as in Algorithm \ref{alg:HitAndRun} with respect to the target distribution $\mu$ from a random starting point with distribution $\nu$ supported on $K$.
	Assume that the following holds for some invertible matrix $\Sigma$:
	\begin{enumerate}[(i)]
		\item \label{con:ContainedBall} the level set of $\mu$ with probability $\frac{1}{8}$ (see Definition \ref{def:LevelSet}) contains a ball of radius $s$ with respect to $\| \cdot \|_{\Sigma^{-1}}$;
		\item \label{con:L2Norm} $\| \nu / \mu \| \leq M$;
		\item \label{con:LocalVariance} $\mathbb{E}_\mu[\| X - \mathbb{E}_\mu[X]\|_{\Sigma^{-1}}^2] \leq S^2$.
	\end{enumerate}
	Let $\nu^{(\ell)}$ be the distribution of the hit-and-run point after $\ell$ steps of hit-and-run sampling applied to $\mu$, where the directions are drawn from a $\mathcal{N}(0,\Sigma)$-distribution. Then, after
	\begin{equation}
		\ell = \left\lceil 10^{30} \frac{n^2 S^2}{s^2} \ln^2 \left( \frac{2 MnS}{sq^2} \right) \ln^3 \left( \frac{2 M}{q^2} \right) \right\rceil,
		\label{eq:WalkLengthLV}
	\end{equation}
	hit-and-run steps, we have $\| \nu^{(\ell)} - \mu \| \leq q$.
\end{corollary}

This corollary can be used to show that two hit-and-run samples with the same starting point are near-independent.
\begin{lemma}
	\label{lemma:NearIndependenceHitAndRun}
	Let $\mu$ be a log-concave probability distribution supported on a convex body $K \subseteq \mathbb{R}^n$, and let $q > 0$. Suppose the conditions of Corollary \ref{cor:MixingTime} are satisfied for some $\Sigma$, $s$, $M$, and $S$. Let $X$ be a random variable with distribution $\nu$ supported on $K$. Consider two hit-and-run random walks as in Algorithm \ref{alg:HitAndRun}, applied to $\mu$, both starting from the same realization of $X$.
	Let the number of steps $\ell$ of both walks be given by \eqref{eq:WalkLengthLV}, and call the resulting end points $Y_1$ and $Y_2$.
	Then, $Y_1$ and $Y_2$ are $6q$-independent.
\end{lemma}
\begin{proof}
	Let $Q$ be the Markov kernel of a hit-and-run step, where directions are chosen from $\mathcal{N}(0, \Sigma)$, and the iterates are drawn from $\mu$ restricted to appropriate line segments, as defined in Algorithm \ref{alg:HitAndRun}. Note that the only dependence of \eqref{eq:WalkLengthLV} on $M$ and $q$ is through the fraction $M/q^2$. Thus, the conditions in Lemma \ref{lemma:NearIndependenceWalk} are satisfied. It follows that $X$ and $Y_1$ are $3q$-independent, and $X$ and $Y_2$ are $3q$-independent.
	Since the $D_i$ and $P_i$ in the random walks are independent, $Y_1$ and $Y_2$ are conditionally independent given $X$. Therefore, Lemma \ref{lemma:NearIndependenceSample} shows the result.
\end{proof}

In the remainder of this section, we aim to show that the conditions of Corollary \ref{cor:MixingTime} are satisfied if $\nu$ and $\mu$ are Boltzmann distributions with parameters $\theta_0$ and $\theta_1$, respectively, such that $\| \theta_1 - \theta_0 \|_{\theta_0}$ is sufficiently small.
Note that Kalai and Vempala \cite{kalai2006simulated} only show these conditions to be satisfied if $\theta_0$ and $\theta_1$ are collinear. In studying interior point methods, we are also interested in (small) deviations from the central path, so it is important to know that the mixing conditions can be shown to hold for these cases.

%\subsection{Contained Ball Radius}

We begin with condition \eqref{con:ContainedBall} from Corollary \ref{cor:MixingTime}.
\begin{lemma}
	\label{lemma:ContainedBall}
	Let $\theta_0, \theta_1 \in \mathbb{R}^n$ and $p \in (0,1)$. Let $h: \mathbb{R}^n \to \mathbb{R}$ be the density of the Boltzmann distribution with parameter $\theta_1$ over a convex body $K \subseteq \mathbb{R}^n$. Let $L$ be the level set of $h$ with probability $p$.
%	$L$ to be the level set of probability $p$ with respect to $h$, i.e.
%	$L := \{ x \in K : h(x) \geq \alpha_p \}$, where $\alpha_p$ is chosen such that $\int_{L} h(x) \diff x = p$.
	Then, $L$ contains a closed $\| \cdot \|_{\Sigma(\theta_0)^{-1}}$-ball with radius
	\begin{equation*}
		\frac{p}{e} \left(1-\| x(\theta_0) - x(\theta_1) \|_{x(\theta_0)}^* \right).
	\end{equation*}
\end{lemma}
\begin{proof}

	Lemma 5.13 from Lov\'asz and Vempala \cite{lovasz2007geometry} shows that $L$ contains a $\langle \cdot, \cdot \rangle_{\Sigma(\theta_1)^{-1}}$-ball with radius $p/e$.
	In other words, there exists some $z \in L$ such that for all $y \in \mathbb{R}^n$ with $\| y - z \|_{\Sigma(\theta_1)^{-1}} \leq p/e$ it holds that $y \in L$. Thus, for all $y \in \mathbb{R}^n$ with $\| y - z \|_{\Sigma(\theta_0)^{-1}} \leq (1 - \| x(\theta_0) - x(\theta_1) \|_{x(\theta_0)}^*)p/e$, the second inequality in \eqref{eq:SCf*} and Lemma \ref{lemma:PropertiesDerivativesConjugate} show
	\begin{equation*}
		\| y - z \|_{\Sigma(\theta_1)^{-1}} = \| y - z \|_{x(\theta_1)}^* \leq \frac{\| y - z \|_{x(\theta_0)}^*}{1-\| x(\theta_0) - x(\theta_1) \|_{x(\theta_0)}^*}
		= \frac{\| y - z \|_{\Sigma(\theta_0)^{-1}}}{1-\| x(\theta_0) - x(\theta_1) \|_{x(\theta_0)}^*}
		\leq \frac{p}{e},
	\end{equation*}
	which proves that all such $x$ lie in $L$.
\end{proof}

%\subsection{Warm Start}
Next, we prove upper and lower bounds on the $L_2$ norm of two Boltzmann distributions. This corresponds to \eqref{con:L2Norm} in Corollary \ref{cor:MixingTime}.
\begin{lemma}
	\label{lemma:L2UpperBound}
	Let $\mu_0$ and $\mu_1$ be Boltzmann distributions supported on a convex body $K \subseteq \mathbb{R}^n$ with parameters $\theta_0$ and $\theta_1$ respectively. Then, if $\| \theta_1 - \theta_0 \|_{\theta_0} < 1$,
	\begin{equation*}
		\exp\Big( \| \theta_1 - \theta_0 \|_{\theta_0}^2 \left[1 + \tfrac{1}{6} \| \theta_1 - \theta_0 \|_{\theta_0}^2 - \tfrac{2}{3} \| \theta_1 - \theta_0 \|_{\theta_0} \right] \Big)  \leq \| \mu_0 / \mu_1 \| \leq \frac{\exp(-2 \| \theta_1 - \theta_0 \|_{\theta_0})}{(1- \| \theta_1 - \theta_0 \|_{\theta_0})^2}.
	\end{equation*}
\end{lemma}
\begin{proof}
	For ease of notation, let $\theta := \theta_0$ and $u := \theta_1 - \theta_0$.
	By definition,
	\begin{equation}
		\| \mu_0 / \mu_1 \| = \mathbb{E}_{\theta_0} \left[ \frac{\diff \mu_0}{\diff \mu_1} \right] = \int_K \frac{e^{\langle 2 \theta, x \rangle}}{e^{\langle \theta + u, x \rangle}} \diff x \frac{\int_K e^{\langle \theta + u, x \rangle} \diff x}{\left(\int_K e^{\langle \theta, x \rangle} \diff x \right)^2}
		= \frac{\int_K e^{\langle \theta - u, x \rangle} \diff x}{\int_K e^{\langle \theta, x \rangle} \diff x } \frac{\int_K e^{\langle \theta + u, x \rangle} \diff x}{\int_K e^{\langle \theta, x \rangle} \diff x }.
		\label{eq:ProofThmL2Def}
	\end{equation}
	The key observation is that the natural logarithm of this expression equals
	\begin{align}
		 &f(\theta + u) - f(\theta) + f(\theta - u) - f(\theta) \nonumber\\
		 &= \int_0^1 \langle g(\theta + tu), u \rangle \diff t - \int_0^1 \langle g(\theta - tu), u \rangle \diff t \nonumber\\
		 &= \int_0^1 \left\langle g(\theta) + \int_0^1 H(\theta + stu)(tu) \diff s, u \right\rangle \diff t - \int_0^1 \left\langle g(\theta) + \int_0^1 H(\theta - stu)(-tu) \diff s, u \right\rangle \diff t \nonumber\\
		 &= \int_0^1 t \int_0^1 \langle H(\theta + stu)u, u \rangle \diff s \diff t + \int_0^1 t \int_0^1 \langle H(\theta - stu)u, u \rangle \diff s \diff t,
		 \label{eq:ProofThmL21}
	\end{align}
	where we used the fundamental theorem of calculus twice. By the second inequality in \eqref{eq:SCf},
	\begin{equation*}
		\langle H(\theta + stu)u, u \rangle = \| u \|_{\theta + stu}^2 \leq \frac{\| u \|_\theta^2}{(1 - st \| u \|_\theta)^2},
	\end{equation*}
	and the same upper bound holds for $\langle H(\theta - stu)u, u \rangle$. Then, \eqref{eq:ProofThmL21} can be bounded above by
	\begin{equation*}
		2\int_0^1 t \int_0^1 \frac{\| u \|_\theta^2}{(1 - st \| u \|_\theta)^2} \diff s \diff t = - 2 [\| u \|_\theta + \ln(1 - \| u \|_\theta)],
	\end{equation*}
	which is non-negative for $0 \leq \| u \|_\theta < 1$. Since \eqref{eq:ProofThmL21} is the natural logarithm of \eqref{eq:ProofThmL2Def},
	\begin{equation*}
		\| \mu_0 / \mu_1 \| \leq \exp( - 2 [\| u \|_\theta + \ln(1 - \| u \|_\theta)] ) = \frac{\exp(-2 \| u \|_\theta)}{(1- \| u \|_\theta)^2}.
	\end{equation*}
	The lower bound on $\| \mu_0 / \mu_1 \|$ follows similarly after noting
	\begin{equation*}
		\langle H(\theta + stu)u, u \rangle = \| u \|_{\theta + stu}^2 \geq \| u \|_\theta^2 (1 - st \| u \|_\theta)^2.
	\end{equation*}
\end{proof}
Both the lower and upper bound in Lemma \ref{lemma:L2UpperBound} have Taylor approximations $1 + \| \theta_1 - \theta_0 \|_{\theta_0}^2 + O(\|\theta_1 - \theta_0\|_{\theta_0}^3)$ at $\| \theta_1 - \theta_0 \|_{\theta_0} = 0$.

The bounds in this theorem are more general than the ones used in Kalai and Vempala \cite[Lemma 4.4]{kalai2006simulated}. They consider the case where $\theta_1 = (1 + \alpha)\theta_0$ for some $\alpha \in (-1,1)$. By using the log-concavity of the Boltzmann distribution, they show
\begin{equation}
	\| \mu_0 / \mu_1 \| \leq \frac{1}{(1+\alpha)^n (1-\alpha)^n},
	\label{eq:KalaiVempalaL2}
\end{equation}
where $\mu_0$ and $\mu_1$ are Boltzmann distributions with parameters $\theta_0$ and $\theta_1$, respectively. Since \eqref{eq:UpperBoundLocalTheta} shows $\| \theta_1 - \theta_0 \|_{\theta_0} = \alpha \| \theta_0 \|_{\theta_0} \leq \alpha \sqrt{\vartheta} = \alpha \sqrt{n + o(n)}$, the upper bound from Lemma \ref{lemma:L2UpperBound} is better than \eqref{eq:KalaiVempalaL2} for sufficiently large $n$ and $\alpha \sqrt{\vartheta} < 1$.

%\subsection{Local Variance}
Finally, we show that condition \eqref{con:LocalVariance} in Corollary \ref{cor:MixingTime} holds.
\begin{lemma}
	\label{lemma:UpperBoundLocalVariance}
	Let $\langle \cdot, \cdot \rangle$ be the Euclidean inner product, and suppose $\theta_0, \theta_1 \in \mathbb{R}^n$ satisfy $\| \theta_0 - \theta_1 \|_{\theta_0} < 1$. Then,
	\begin{equation*}
		\mathbb{E}_{\theta_1} \left[ \| X - \mathbb{E}_{\theta_1}[X] \|_{\Sigma(\theta_0)^{-1}}^2 \right] \leq \frac{n}{(1 - \| \theta_0 - \theta_1 \|_{\theta_0})^2}.
	\end{equation*}
\end{lemma}
\begin{proof}
	By using the cyclic permutation invariance of the trace,
	\begin{align*}
		\mathbb{E}_{\theta_1} \left[ \| X - \mathbb{E}_{\theta_1}[X] \|_{\Sigma(\theta_0)^{-1}}^2 \right] &= \mathbb{E}_{\theta_1} \left[ \tr \left[ (X - \mathbb{E}_{\theta_1}[X])^\top \Sigma(\theta_0)^{-1} (X - \mathbb{E}_{\theta_1}[X]) \right] \right]\\
		&= \tr \left[ \Sigma(\theta_1) \Sigma(\theta_0)^{-1} \right]\\
		&= \tr \left[ \Sigma(\theta_0)^{-1/2} \Sigma(\theta_1) \Sigma(\theta_0)^{-1/2} \right].
	\end{align*}
	We can upper bound $\tr [ \Sigma(\theta_0)^{-1/2} \Sigma(\theta_1) \Sigma(\theta_0)^{-1/2} ] = \tr [ H(\theta_0)^{-1/2} H(\theta_1) H(\theta_0)^{-1/2} ]$ by
	\begin{equation*}
		n \max_{u \in \mathbb{R}^n} \frac{\langle H(\theta_0)^{-1/2} u, H(\theta_1) H(\theta_0)^{-1/2} u] \rangle}{\| u \|^2} = n \max_{v \in \mathbb{R}^n} \frac{\langle v, H(\theta_1) v \rangle}{\| H(\theta_0)^{1/2} v \|^2} = n \max_{v \in \mathbb{R}^n} \frac{\| v \|^2_{\theta_1}}{\| v \|^2_{\theta_0}}.
	\end{equation*}
	The claim now follows from \eqref{eq:SCf}.
%	Since $\Sigma(\theta_1)\Sigma(\theta_0)^{-1} = H(\theta_1)H(\theta_0)^{-1}$, we can upper bound $\tr [ \Sigma(\theta_1) \Sigma(\theta_0)^{-1} ]$ by
%	\begin{equation*}
%		n \max_{v \in \mathbb{R}^n} \frac{\langle [H(\theta_0)^{1/2}v], H(\theta_1)H(\theta_0)^{-1} [H(\theta_0)^{1/2}v] \rangle}{\| H(\theta_0)^{1/2}v \|^2} = n \max_{v \in \mathbb{R}^n} \frac{\langle v, H(\theta_1) v \rangle}{\| v \|^2_{\theta_0}} = n \max_{v \in \mathbb{R}^n} \frac{\| v \|^2_{\theta_1}}{\| v \|^2_{\theta_0}},
%	\end{equation*}
%	where the first equality is due to the similarity of $H(\theta_0)^{1/2} H(\theta_1) H(\theta_0)^{-1/2}$ and $H(\theta_1)$. The claim now follows from \eqref{eq:SCf}.
\end{proof}

It is interesting to compare the upper bound in Lemma \ref{lemma:UpperBoundLocalVariance} with the one Kalai and Vempala \cite{kalai2006simulated} arrive at through a near-isotropy argument. It is shown by \cite[Lemma 4.2 and 4.3]{kalai2006simulated} that
\begin{equation}
	\mathbb{E}_{\theta_1} \left[ \| X - \mathbb{E}_{\theta_1}[X] \|_{\Sigma(\theta_0)^{-1}}^2 \right] \leq 16 n \| \mu_1 / \mu_0 \| \max_{v \in \mathbb{R}^n} \frac{\mathbb{E}_{\theta_0} \left [\langle v, X - \mathbb{E}_{\theta_0}[X]\rangle^2_{\Sigma(\theta_0)^{-1}} \right]}{\| v \|^2_{\Sigma(\theta_0)^{-1}}},
	\label{eq:UpperBoundLocalVarianceKalaiVempala}
\end{equation}
where $\mu_0$ and $\mu_1$ are Boltzmann distributions with parameters $\theta_0$ and $\theta_1$, respectively. Observe that for all $v \in \mathbb{R}^n$,
\begin{equation*}
	\frac{\mathbb{E}_{\theta_0} \left[\langle \Sigma(\theta_0)^{-1} v, X - \mathbb{E}_{\theta_0}[X]\rangle^2 \right]}{\| v \|^2_{\Sigma(\theta_0)^{-1}}}
	 = \frac{\langle \Sigma(\theta_0)^{-1} v, \Sigma(\theta_0) \Sigma(\theta_0)^{-1} v \rangle}{\| v \|^2_{\Sigma(\theta_0)^{-1}}} = 1,
\end{equation*}
and therefore the right hand side of \eqref{eq:UpperBoundLocalVarianceKalaiVempala} is just $16 n \| \mu_1 / \mu_0 \|$. If we upper bound this norm by Lemma \ref{lemma:L2UpperBound}, we find
\begin{equation}
	\mathbb{E}_{\theta_1} \left[ \| X - \mathbb{E}_{\theta_1}[X] \|_{\Sigma(\theta_0)^{-1}}^2 \right] \leq \frac{16 n \exp(-2 \| \theta_1 - \theta_0 \|_{\theta_1})}{(1- \| \theta_1 - \theta_0 \|_{\theta_1})^2}.
	\label{eq:UpperBoundLocalVarianceKalaiVempala2}
\end{equation}
By the second inequality in \eqref{eq:SCf}, we have
\begin{equation*}
	\frac{n}{(1 - \| \theta_0 - \theta_1 \|_{\theta_0})^2} \leq \frac{n}{\left(1 - \frac{\| \theta_0 - \theta_1 \|_{\theta_1}}{1-\| \theta_0 - \theta_1 \|_{\theta_1}} \right)^2} \leq \frac{16 n \exp(-2 \| \theta_1 - \theta_0 \|_{\theta_1})}{(1- \| \theta_1 - \theta_0 \|_{\theta_1})^2},
\end{equation*}
where the second inequality holds for $\| \theta_1 - \theta_0 \|_{\theta_1} \leq 0.438$. In this case, the bound in Lemma \ref{lemma:UpperBoundLocalVariance} is stronger than \eqref{eq:UpperBoundLocalVarianceKalaiVempala2}.
Alternatively, if $\theta_0 = (1 + \alpha)\theta_1$ for some $\alpha \in (-1,1)$, \eqref{eq:KalaiVempalaL2} shows that \eqref{eq:UpperBoundLocalVarianceKalaiVempala} can be bounded by
\begin{equation}
	\mathbb{E}_{\theta_1} \left[ \| X - \mathbb{E}_{\theta_1}[X] \|_{\Sigma(\theta_0)^{-1}}^2 \right] \leq \frac{ 16 n}{(1+\alpha)^n (1-\alpha)^n}.
	\label{eq:UpperBoundLocalVarianceKalaiVempala3}
\end{equation}
Since \eqref{eq:UpperBoundLocalTheta} shows $\| \theta_1 - \theta_0 \|_{\theta_0} = \alpha \| \theta_0 \|_{\theta_0} \leq \alpha \sqrt{\vartheta} = \alpha \sqrt{n + o(n)}$, the upper bound from Lemma \ref{lemma:UpperBoundLocalVariance} is better than \eqref{eq:UpperBoundLocalVarianceKalaiVempala3} for sufficiently large $n$ and $\alpha \sqrt{\vartheta} < 1$.

%\subsection{Hit-and-Run for Boltzmann Distributions}
The results from this section can be summarized as follows.
\begin{theorem}
	\label{thm:MixingTimeSC}
	Let $K \subseteq \mathbb{R}^n$ be a convex body, and let $\langle \cdot, \cdot \rangle$ be the Euclidean inner product. Let $q > 0$, and $\theta_0, \theta_1 \in \mathbb{R}^n$ such that $\Delta \theta := \| \theta_1 - \theta_0 \|_{\theta_0} < 1$ and $\Delta x := \| x(\theta_1) - x(\theta_0) \|_{x(\theta_0)}^* < 1$. Pick $\epsilon \in [0, 1)$, and suppose we have an invertible matrix $\widehat{\Sigma}(\theta_0)$ such that
	\begin{equation}
		(1-\epsilon) y^\top \widehat{\Sigma}(\theta_0)^{-1} y \leq y^\top \Sigma(\theta_0)^{-1} y \leq (1+\epsilon) y^\top \widehat{\Sigma}(\theta_0)^{-1} y \qquad \forall y \in \mathbb{R}^n. \label{eq:HitAndRunInvCovCondition}
	\end{equation}
	Consider a hit-and-run random walk as in Algorithm \ref{alg:HitAndRun} applied to the Boltzmann distribution $\mu$ with parameter $\theta_1$ from a random starting point drawn from a Boltzmann distribution with parameter $\theta_0$.
	Let $\nu^{(\ell)}$ be the distribution of the hit-and-run point after $\ell$ steps of hit-and-run sampling applied to $\mu$, where the directions are drawn from a $\mathcal{N}(0,\widehat{\Sigma}(\theta_0))$-distribution. Then, after
	\begin{equation}
		\label{eq:WalkLength}
		\ell = \left\lceil \frac{1+\epsilon}{1-\epsilon}\frac{64 e^2 n^3 10^{30}}{(1 - \Delta \theta)^2 (1- \Delta x )^2} \ln^2 \left( \sqrt{\frac{1+\epsilon}{1-\epsilon}} \frac{16 e n\sqrt{n} \exp(-2\Delta\theta)}{q^2 (1-\Delta\theta)^3(1-\Delta x)} \right) \ln^3 \left( \frac{2 \exp(-2\Delta\theta)}{q^2  (1-\Delta\theta)^2} \right) \right\rceil,
	\end{equation}
	hit-and-run steps, we have $\| \nu^{(\ell)} - \mu \| \leq q$.
\end{theorem}
\begin{proof}
	We will apply Corollary \ref{cor:MixingTime} with respect to $\widehat{\Sigma}(\theta_0)$.
	
	By Lemma \ref{lemma:ContainedBall}, the level set of $\mu$ with probability $\frac{1}{8}$ contains a $\| \cdot \|_{\Sigma(\theta_0)^{-1}}$-ball with radius $\frac{1}{8e} (1-\| x(\theta_0) - x(\theta_1) \|_{x(\theta_0)}^*)$. Denote the center of this ball by $z \in K$.
%	We claim that the level set of $\mu$ with probability $\frac{1}{8}$ contains a $\| \cdot \|_{\widehat{\Sigma}^{-1}}$-ball with radius $\frac{1}{8e\sqrt{1+\epsilon}} (1-\| g(\theta_0) - g(\theta_1) \|_{g(\theta_0)}^* )$.
	Then, for all $y \in K$ with $\| y - z \|_{\widehat{\Sigma}(\theta_0)^{-1}} \leq \frac{1}{8e\sqrt{1+\epsilon}} (1-\| x(\theta_0) - x(\theta_1) \|_{x(\theta_0)}^* )$, it can be seen from \eqref{eq:HitAndRunInvCovCondition} that
	\begin{equation*}
		\| y - z \|_{\Sigma(\theta_0)^{-1}} \leq \sqrt{1+\epsilon} \| y - z \|_{\widehat{\Sigma}(\theta_0)^{-1}} \leq \frac{1}{8e} \left(1-\| x(\theta_0) - x(\theta_1) \|_{x(\theta_0)}^* \right),
	\end{equation*}
	and thus $y$ lies in the level set. Therefore, the level set of $\mu$ with probability $\frac{1}{8}$ contains a $\| \cdot \|_{\widehat{\Sigma}(\theta_0)^{-1}}$-ball with radius $\frac{1}{8e\sqrt{1+\epsilon}} (1-\| x(\theta_0) - x(\theta_1) \|_{x(\theta_0)}^* )$.
	
	Moreover, \eqref{eq:HitAndRunInvCovCondition} and Lemma \ref{lemma:UpperBoundLocalVariance} show
	\begin{equation*}
		\mathbb{E}_{\theta_1}[\| X - \mathbb{E}_{\theta_1}[X]\|_{\widehat{\Sigma}(\theta_0)^{-1}}^2] \leq \frac{1}{1-\epsilon} \mathbb{E}_{\theta_1}[\| X - \mathbb{E}_{\theta_1}[X]\|_{\Sigma(\theta_0)^{-1}}^2] \leq \frac{n}{(1-\epsilon)(1 - \| \theta_0 - \theta_1 \|_{\theta_0})^2}.
	\end{equation*}
	
	Using $s = \frac{1}{8e\sqrt{1+\epsilon}} (1-\| x(\theta_0) - x(\theta_1) \|_{x(\theta_0)}^* )$, $S^2 = \frac{n}{(1-\epsilon)(1 - \| \theta_0 - \theta_1 \|_{\theta_0})^2}$, and Lemma \ref{lemma:L2UpperBound}, Corollary \ref{cor:MixingTime} now proves the result.
\end{proof}
Since Theorem \ref{thm:MixingTimeSC} is essentially an application of Corollary \ref{cor:MixingTime}, Lemma \ref{lemma:NearIndependenceHitAndRun} also holds in this setting. For ease of reference, we state this result below.
\begin{lemma}
	\label{lemma:NearIndependenceHitAndRunSC}
	Let $K \subseteq \mathbb{R}^n$ be a convex body, and let $q > 0$. Let $\theta_0, \theta_1 \in \mathbb{R}^n$ such that the conditions of Theorem \ref{thm:MixingTimeSC} are satisfied for some $\epsilon$ and $\widehat{\Sigma}(\theta_0)$. Let $X$ be a random variable following a Boltzmann distribution supported on $K$ with parameter $\theta_0$. Consider two hit-and-run random walks as in Algorithm \ref{alg:HitAndRun}, applied to the Boltzmann distribution with parameter $\theta_1$, both starting from the realization of $X$, where all $D_i$ and $P_i$ in one random walk are independent of all $D_i$ and $P_i$ in the other random walk. Let the number of steps $\ell$ of both walks be given by \eqref{eq:WalkLength}, and call the resulting end points $Y_1$ and $Y_2$. Then, $Y_1$ and $Y_2$ are $6q$-independent.
\end{lemma}

\section{Sampling Quality Guarantees}
\label{sec:SamplingQuality}
In this section, we provide probabilistic guarantees on the quality of the empirical mean and covariance estimates of a log-concave distribution $\mu$.
We will repeatedly use that for random variables $Y$ and $Z$ taking values in a set $K$ and a function $\phi$ on $K$,
\begin{equation}
	\mathbb{E}[\phi(Y)] = \mathbb{E}[\phi(Z)] + \mathbb{E}[\phi(Y) - \phi(Z)]
	= \mathbb{E}[\phi(Z)] + \mathbb{E}[\phi(Y) - \phi(Z) | Y \neq Z] \mathbb{P}\{ Y \neq Z \}.
	\label{eq:DivineInterventionExpectation}
\end{equation}

%\subsection{Expectation}
We start by analyzing the quality of the mean estimate.
\begin{theorem}
	\label{thm:SampleMeanNorm}
	Let $K \subseteq \mathbb{R}^n$ be a convex body, and let $\langle \cdot, \cdot \rangle$ be the Euclidean inner product. Suppose $K$ is contained in a Euclidean ball with radius $R > 0$. Let $\alpha > 0$, $p \in (0,1)$, and $\epsilon \in [0, 1)$.
	Let $\theta_0, \theta_1 \in \mathbb{R}^n$ such that $\Delta \theta := \| \theta_1 - \theta_0 \|_{\theta_0} < 1$ and $\Delta x := \| x(\theta_1) - x(\theta_0) \|_{x(\theta_0)}^* < 1$. Suppose we have an invertible matrix $\widehat{\Sigma}(\theta_0)$ such that
	\begin{equation}
		(1-\epsilon) y^\top \widehat{\Sigma}(\theta_0)^{-1} y \leq y^\top \Sigma(\theta_0)^{-1} y \leq (1+\epsilon) y^\top \widehat{\Sigma}(\theta_0)^{-1} y \qquad \forall y \in \mathbb{R}^n.
	\end{equation}
	Pick
	\begin{equation*}
		N \geq \frac{2n}{p \alpha^2}, \qquad q \leq \frac{p \alpha^2}{204 n R^2} \lambda_{\min}(\Sigma(\theta_1)).
	\end{equation*}
	Let $X_0$ be a random starting point drawn from a Boltzmann distribution with parameter $\theta_0$. Let $Y^{(1)}, ..., Y^{(N)}$ be the end points of $N$ hit-and-run random walks applied to the Boltzmann distribution with parameter $\theta_1$ having starting point $X_0$, where the directions are drawn from a $\mathcal{N}(0,\widehat{\Sigma}(\theta_0))$-distribution, and each walk has length $\ell$ given by \eqref{eq:WalkLength}. (Note that $\ell$ depends on $\epsilon$, $n$, $q$, $\Delta\theta$, and $\Delta x$.)
	Then, the empirical mean $\widehat{Y} := \frac{1}{N} \sum_{j=1}^N Y^{(j)}$ satisfies
	\begin{equation}
		\label{eq:SampleMeanNormStatement}
		\mathbb{P} \left\{ \| \widehat{Y} - \mathbb{E}_{\theta_1}[X] \|_{\Sigma(\theta_1)^{-1}} \leq \alpha \right\} \geq 1-p.
	\end{equation}
\end{theorem}
\begin{proof}
	Theorem \ref{thm:MixingTimeSC} ensures that the distributions of the samples $Y^{(1)}, ..., Y^{(N)}$ all have a total variation distance to the Boltzmann distribution with parameter $\theta_1$ of at most $q$. By Lemma \ref{lemma:NearIndependenceHitAndRunSC}, the samples are pairwise $6q$-independent. It therefore remains to be shown that $N$ pairwise $6q$-independent samples with total variation distance to the Boltzmann distribution with parameter $\theta_1$ of at most $q$ are enough to guarantee \eqref{eq:SampleMeanNormStatement}.
	
	We start by investigating an expression resembling the variance of $\widehat{Y}$ in the norm induced by $\Sigma(\theta_1)^{-1}$:
	\begin{align}
		\begin{aligned}
		\mathbb{E} [\| \widehat{Y} - \mathbb{E}_{\theta_1}[X] \|_{\Sigma(\theta_1)^{-1}}^2]
%		&=\mathbb{E} \left[\frac{1}{N^2} \sum_{j=1}^N \sum_{k =1}^N (Y^{(j)} - \mathbb{E}_{\theta_1}[X] )^\top \Sigma^{-1} (Y^{(k)} - \mathbb{E}_{\theta_1}[X]) \right]\\
		&= \frac{1}{N^2} \sum_{j=1}^N \mathbb{E} \left[ (Y^{(j)} - \mathbb{E}_{\theta_1}[X] )^\top \Sigma(\theta_1)^{-1} (Y^{(j)} - \mathbb{E}_{\theta_1}[X]) \right]\\
		&+ \frac{1}{N^2} \sum_{j=1}^N \sum_{k\neq j} \mathbb{E} \left[ (Y^{(j)} - \mathbb{E}_{\theta_1}[X] )^\top \Sigma(\theta_1)^{-1} (Y^{(k)} - \mathbb{E}_{\theta_1}[X]) \right].
%		&= \frac{1}{N} \sum_{j=1}^N \tr \left( \mathbb{E} \left[ (Y^{(j)} - \mathbb{E}_{\theta_1}[X] ) (Y^{(j)} - \mathbb{E}_{\theta_1}[X])^\top \right] \Sigma^{-1} \right)
%		&= \frac{1}{N^2} \sum_{j=1}^N \tr(\Sigma \Sigma^{-1}) + \frac{1}{N^2} \sum_{j=1}^N \sum_{k \neq j} \mathbb{E} \left[ (Y^{(j)} - \mathbb{E}_\mu[Y] )^\top \Sigma^{-1} (Y^{(k)} - \mathbb{E}_\mu[Y]) \right].
		\end{aligned}
		\label{eq:ProofSampleMeanNorm1}
	\end{align}
	The first term of \eqref{eq:ProofSampleMeanNorm1} can be bounded if one notes that Lemma \ref{lemma:DivineIntervention} guarantees that for each $Y^{(j)}$ there exists a $Z^{(j)}$ with Boltzmann distribution with parameter $\theta_1$ and $\mathbb{P}\{Y^{(j)} = Z^{(j)}\} \geq 1-q$.
	Using \eqref{eq:DivineInterventionExpectation}, we have for all $j \in \{1, ..., N\}$,
	\begin{align*}
		&\mathbb{E} \left[ (Y^{(j)} - \mathbb{E}_{\theta_1}[X] )^\top \Sigma(\theta_1)^{-1} (Y^{(j)} - \mathbb{E}_{\theta_1}[X]) \right] \\
		&\leq \mathbb{E} \left[ (Z^{(j)} - \mathbb{E}_{\theta_1}[X] )^\top \Sigma(\theta_1)^{-1} (Z^{(j)} - \mathbb{E}_{\theta_1}[X]) \right]  + q \lambda_{\max}(\Sigma(\theta_1)^{-1}) \| Y^{(j)} - \mathbb{E}_{\theta_1}[X] \|^2\\
		&= \tr \left(\mathbb{E} \left[ (Z^{(j)} - \mathbb{E}_{\theta_1}[X])(Z^{(j)} - \mathbb{E}_{\theta_1}[X] )^\top \right] \Sigma(\theta_1)^{-1} \right)  + q \frac{\| Y^{(j)} - \mathbb{E}_{\theta_1}[X] \|^2}{\lambda_{\min}(\Sigma(\theta_1))} \\
		&\leq n + q \frac{(2R)^2}{\lambda_{\min}(\Sigma(\theta_1))}.
	\end{align*}
	To bound the second term of \eqref{eq:ProofSampleMeanNorm1}, note that since $Y^{(j)}$ and $Y^{(k)}$ are $6q$-independent, so are $\Sigma(\theta_1)^{-1/2} (Y^{(j)} - \mathbb{E}_{\theta_1}[X])$ and $\Sigma(\theta_1)^{-1/2} (Y^{(k)} - \mathbb{E}_{\theta_1}[X])$. By Lemma \ref{lemma:NearIndependentProduct} and Lemma \ref{lemma:NearIndependenceExpectation}, we have for all $j \neq k$,
	\begin{align*}
		&\mathbb{E} \left[ (Y^{(j)} - \mathbb{E}_{\theta_1}[X] )^\top \Sigma(\theta_1)^{-1} (Y^{(k)} - \mathbb{E}_{\theta_1}[X]) \right] \\
		&= \sum_{i=1}^n \mathbb{E} \left[ \left( \Sigma(\theta_1)^{-1/2}(Y^{(j)} - \mathbb{E}_{\theta_1}[X] ) \right)_i \left(\Sigma(\theta_1)^{-1/2} (Y^{(k)} - \mathbb{E}_{\theta_1}[X]) \right)_i \right]\\
		&\leq \sum_{i=1}^n \left[ \left( \mathbb{E} \left[  \Sigma(\theta_1)^{-1/2}(Y^{(j)} - \mathbb{E}_{\theta_1}[X] ) \right] \right)_{i}
		\left( \mathbb{E} \left[ \Sigma(\theta_1)^{-1/2} (Y^{(k)} - \mathbb{E}_{\theta_1}[X]) \right] \right)_{i} + 4 (6q) (2R\lambda_{\max}(\Sigma(\theta_1)^{-1/2}))^2 \right].
	\end{align*}
	Using Lemma \ref{lemma:DivineIntervention} and \eqref{eq:DivineInterventionExpectation} in the same manner as before, we get for all $j$,
	\begin{equation*}
		\left| \left( \mathbb{E} \left[ \Sigma(\theta_1)^{-1/2}(Y^{(j)} - \mathbb{E}_{\theta_1}[X] ) \right] \right)_i \right| \leq 0 + q (2R) \lambda_{\max}(\Sigma(\theta_1)^{-1/2}) = \frac{2Rq}{\sqrt{\lambda_{\min}(\Sigma(\theta_1))}}.
	\end{equation*}
	In conclusion,
	\begin{equation*}
		\mathbb{E} \left[ (Y^{(j)} - \mathbb{E}_{\theta_1}[X] )^\top \Sigma(\theta_1)^{-1} (Y^{(j)} - \mathbb{E}_{\theta_1}[X]) \right]
		\leq \frac{n}{N} + \frac{4qR^2}{N \lambda_{\min}(\Sigma(\theta_1))} + \frac{96 q n  R^2}{\lambda_{\min}(\Sigma(\theta_1))} + \frac{4n R^2 q^2}{\lambda_{\min}(\Sigma(\theta_1))}
		\leq p \alpha^2.
	\end{equation*}
	
	The proof is completed by applying Markov's inequality:
	\begin{equation*}
		\mathbb{P} \left\{ \| \widehat{Y} - \mathbb{E}_{\theta_1}[X] \|_{\Sigma(\theta_1)^{-1}}^2 > \alpha^2 \right\} \leq \frac{\mathbb{E} [\| \widehat{Y} - \mathbb{E}_{\theta_1}[X] \|_{\Sigma(\theta_1)^{-1}}^2]}{\alpha^2} \leq p.
	\end{equation*}
\end{proof}

If $Y^{(1)}, ..., Y^{(N)}$ are random variables, we define the associated empirical covariance matrix as
\begin{equation}
	\widehat{\Sigma} := \frac{1}{N} \sum_{j=1}^{N} Y^{(j)} (Y^{(j)})^\top - \left( \frac{1}{N} \sum_{j=1}^{N} Y^{(j)} \right) \left( \frac{1}{N} \sum_{j=1}^{N} Y^{(j)} \right)^\top.
	\label{eq:DefinitionEmpiricalCovariance}
\end{equation}
Before we can prove a result similar to Theorem \ref{thm:SampleMeanNorm} for the empirical covariance, it will be helpful to prove that, without loss of generality, we can assume that the underlying distribution has identity covariance and mean zero.
\begin{lemma}
	\label{lemma:InvarianceAffineTransformation}
%	Let $K$ be a convex body and let $S$ be the empirical covariance matrix of a distribution $\mu$ over $K$ with density $h$ based on samples $X^{(1)}, ..., X^{(N)}$.
	Let $K$ be a convex body and let $\widehat{\Sigma}$ be the empirical covariance matrix of a distribution $\mu$ over $K$ with density $h$ based on samples $X^{(1)}, ..., X^{(N)}$ as in \eqref{eq:DefinitionEmpiricalCovariance}.
	Denote the set $\{ Ax+b : x \in K \}$ by $A K + b$, where $A \in \mathbb{R}^{n \times n}$ is of full rank.
%	Let $\overline{h}(y) = \det(A^{-1}) h(A^{-1}(y-b))$ be a probability density over $A K + b$ with induced distribution $\overline{\mu}$. Let $\overline{S}$ be the empirical covariance matrix of $\overline{\mu}$ based on the samples $Y^{(j)} = A X^{(j)} + b$ for $j \in \{1, ..., N\}$.
	Let $h'(y) = \det(A^{-1}) h(A^{-1}(y-b))$ be a probability density over $A K + b$ with induced distribution $\mu'$. Let $\widehat{\Sigma}'$ be the empirical covariance matrix of $\mu'$ based on the samples $Y^{(j)} = A X^{(j)} + b$ for $j \in \{1, ..., N\}$ as in \eqref{eq:DefinitionEmpiricalCovariance}.
	Let the true covariance matrix of $\mu$ be $\Sigma$, and the covariance matrix of $\mu'$ be $\Sigma'$. Then, for any $\varepsilon \in [0,1]$,
	\begin{equation}
		(1-\varepsilon) u^\top \widehat{\Sigma} u \leq u^\top \Sigma u \leq (1+\varepsilon) u^\top \widehat{\Sigma} u \qquad \forall u \in \mathbb{R}^n,
		\label{eq:LemmaInvarianceAffineTransformFirst}
	\end{equation}
	if and only if
	\begin{equation}
		(1-\varepsilon) v^\top \widehat{\Sigma}' v \leq v^\top \Sigma' v \leq (1+\varepsilon) v^\top \widehat{\Sigma}' v \qquad \forall v \in \mathbb{R}^n.
		\label{eq:LemmaInvarianceAffineTransformSecond}
	\end{equation}
\end{lemma}
\begin{proof}
	Let $\widehat{X} = \frac{1}{N} \sum_{j=1}^N X^{(j)}$ and $\widehat{\Sigma} = \frac{1}{N} \sum_{j=1}^N (X^{(j)} - \widehat{X}) (X^{(j)} - \widehat{X})^\top$. The empirical covariance matrix of $\mu'$ is
	\begin{align*}
		\widehat{\Sigma}' &= \frac{1}{N} \sum_{j=1}^N (A X^{(j)} + b - (A \widehat{X} + b)) (A X^{(j)} + b - (A \widehat{X} + b))^\top\\
		&= A \left(\frac{1}{N} \sum_{j=1}^N ( X^{(j)} - \widehat{X} ) (X^{(j)} - \widehat{X} )^\top \right) A^\top = A \widehat{\Sigma} A^\top.
	\end{align*}
	Similarly, if $X$ has distribution $\mu$ and $Y$ has distribution $\mu'$,
	\begin{align*}
		\Sigma' &= \int_{AK+b} (y - \mathbb{E}[Y] ) (y - \mathbb{E}[Y] )^\top h'(y) \diff y\\
		&= \int_{K} (Ax + b - (A\mathbb{E}[X] + b) ) (Ax + b - (A\mathbb{E}[X] + b) )^\top h(x) \diff x = A \Sigma A^\top.
	\end{align*}
	The equivalence of \eqref{eq:LemmaInvarianceAffineTransformFirst} and \eqref{eq:LemmaInvarianceAffineTransformSecond} follows by taking $u = Av$.
\end{proof}

With this lemma, we are ready to bound the number of samples required to find an approximation of the covariance matrix of $\mu$ satisfying a certain quality criterion.
\begin{theorem}
	\label{thm:CovQuality}	
	Let $K \subseteq \mathbb{R}^n$ be a convex body, and let $\langle \cdot, \cdot \rangle$ be the Euclidean inner product. Suppose $K$ is contained in a Euclidean ball with radius $R > 0$. Let $\varepsilon, p \in (0,1)$, and $\epsilon \in [0, 1)$.
	Let $\theta_0, \theta_1 \in \mathbb{R}^n$ such that $\Delta \theta := \| \theta_1 - \theta_0 \|_{\theta_0} < 1$ and $\Delta x := \| x(\theta_1) - x(\theta_0) \|_{x(\theta_0)}^* < 1$. Suppose we have an invertible matrix $\widehat{\Sigma}(\theta_0)$ such that
	\begin{equation*}
		(1-\epsilon) y^\top \widehat{\Sigma}(\theta_0)^{-1} y \leq y^\top \Sigma(\theta_0)^{-1} y \leq (1+\epsilon) y^\top \widehat{\Sigma}(\theta_0)^{-1} y \qquad \forall y \in \mathbb{R}^n.
	\end{equation*}
	Pick
	\begin{equation}
		\label{eq:CovQualityNumberSamples}
		N \geq \frac{490 n^2}{p \varepsilon^2}, \qquad q \leq \frac{p \varepsilon^2}{49980 n^2 R^4} \lambda_{\min}(\Sigma(\theta_1))^2.
	\end{equation}
	Let $X_0$ be a random starting point drawn from a Boltzmann distribution with parameter $\theta_0$. Let $Y^{(1)}, ..., Y^{(N)}$ be the end points of $N$ hit-and-run random walks applied to the Boltzmann distribution with parameter $\theta_1$ having starting point $X_0$, where the directions are drawn from a $\mathcal{N}(0,\widehat{\Sigma}(\theta_0))$-distribution, and each walk has length $\ell$ given by \eqref{eq:WalkLength}. (Note that $\ell$ depends on $\epsilon$, $n$, $q$, $\Delta\theta$, and $\Delta x$.)
	Then, the empirical covariance matrix $\widehat{\Sigma} \approx \Sigma(\theta_1)$ as defined in \eqref{eq:DefinitionEmpiricalCovariance} satisfies
	\begin{equation}
		\label{eq:CovQualityStatement}
		\mathbb{P} \left\{ (1-\varepsilon) v^\top \widehat{\Sigma} v \leq v^\top \Sigma(\theta_1) v \leq (1+\varepsilon) v^\top \widehat{\Sigma} v \qquad \forall v \in \mathbb{R}^n \right\} \geq 1-p.
	\end{equation}
\end{theorem}
\begin{proof}
	By the same argument as in Theorem \ref{thm:SampleMeanNorm}, $Y^{(1)}, ..., Y^{(N)}$ are pairwise $6q$-independent samples, each with a distribution that has total variation distance to the Boltzmann distribution with parameter $\theta_1$ of at most $q$.

	The remainder of the proof uses an approach similar to Theorem 5.11 from Kannan, Lov\'asz and Simonovits \cite{kannan1997random}, although their result only applies to the uniform distribution.
	As in Theorem \ref{thm:SampleMeanNorm}, define $\widehat{Y} = \frac{1}{N} \sum_{j=1}^N Y^{(j)}$.
	Lemma \ref{lemma:InvarianceAffineTransformation} shows that applying an affine transformation to $K$ does not affect the statement. We can therefore assume that the Boltzmann distribution with parameter $\theta_1$ is isotropic, i.e. has identity covariance, and that the mean of this distribution is the origin. However, the support of this distribution is now contained in a ball of radius $\lambda_{\max}(\Sigma(\theta_1)^{-1/2})R$.
	
	We want to prove that with probability at least $1 - p$, for every $v \in \mathbb{R}^n$,
	\begin{equation*}
		(1-\varepsilon) v^\top \widehat{\Sigma} v \leq \| v \|^2 \leq (1+\varepsilon) v^\top \widehat{\Sigma} v,
	\end{equation*}
	or equivalently,
	\begin{equation}
		\frac{1}{1+\varepsilon} \leq \frac{v^\top \widehat{\Sigma} v}{\|v\|^2} \leq \frac{1}{1-\varepsilon}.
		\label{eq:ProofCovQualityCondition1}
	\end{equation}
	We may therefore assume in the remainder that $\|v\| = 1$. Letting
%	\begin{equation*}
		$S := \frac{1}{N} \sum_{j=1}^N Y^{(j)} (Y^{(j)})^\top$,
%	\end{equation*}
	\eqref{eq:ProofCovQualityCondition1} is equivalent to
	\begin{equation}
		\frac{1}{1+\varepsilon} + (v^\top \widehat{Y})^2 \leq v^\top S v \leq \frac{1}{1-\varepsilon} + (v^\top \widehat{Y})^2.
		\label{eq:ProofCovQualityCondition2}
	\end{equation}
	We will use that
	\begin{equation*}
		v^\top S v = v^\top v + v^\top (S - I) v = 1 + v^\top (S - I) v.
	\end{equation*}
	
	We continue by showing that $\mathbb{P}\{\rho(S - I) > \frac{34}{35} \varepsilon/(1+\varepsilon) \}$ is small, where $\rho(S-I)$ is the spectral radius of $S-I$. It is known that
	\begin{equation}
		\rho(S-I) = \sqrt{\lambda_{\max}([S-I]^2)} \leq \sqrt{\tr([S-I]^2)}.
		\label{eq:ProofCovQualitySpectralRadius}
	\end{equation}
	To apply Markov's inequality, we will bound $\mathbb{E} [\tr([S-I]^2)] = \mathbb{E} [\tr(S^2 - 2S)] + n$. By Lemma \ref{lemma:DivineIntervention}, for each $Y^{(j)}$ we can find a $Z^{(j)}$ following a Boltzmann distribution with parameter $\theta_1$ such that $\mathbb{P}\{ Z^{(j)} \neq Y^{(j)} \} \leq q$. Then, using \eqref{eq:DivineInterventionExpectation},
	\begin{align}
		\mathbb{E}[ \tr(S)] &= \mathbb{E} \Bigg[ \sum_{r=1}^n \frac{1}{N} \sum_{j=1}^N \left(Y^{(j)}_r \right)^2 \Bigg] \nonumber\\
		&\geq \mathbb{E} \Bigg[ \sum_{r=1}^n \frac{1}{N} \sum_{j=1}^N \left(Z^{(j)}_r \right)^2 \Bigg] - qn \left(2R\lambda_{\max}(\Sigma(\theta_1)^{-1/2}) \right)^2 \nonumber\\
		&= n - qn \left(2R\lambda_{\max}(\Sigma(\theta_1)^{-1/2}) \right)^2,
		\label{eq:ProofCovQualityTrace}
	\end{align}
	where the last inequality uses that the origin is contained in the support of the Boltzmann distribution with parameter $\theta_1$, since we assumed the origin is the mean of this distribution.
	
	Finding a bound on $\mathbb{E}[\tr (S^2)]$ requires more work. Note that
	\begin{equation*}
		\mathbb{E} [\tr (S^2)] = \frac{1}{N^2} \mathbb{E} \Bigg[ \sum_{j=1}^N \sum_{k=1}^N \left( (Y^{(j)})^\top Y^{(k)} \right)^2 \Bigg]
		= \frac{1}{N^2} \sum_{j=1}^N \mathbb{E} \left[\| Y^{(j)} \|^4 \right] + \frac{1}{N^2} \mathbb{E} \Bigg[ \sum_{j=1}^N \sum_{k\neq j} \left( (Y^{(j)})^\top Y^{(k)} \right)^2 \Bigg].
	\end{equation*}
	By another application of Lemma \ref{lemma:DivineIntervention} and \eqref{eq:DivineInterventionExpectation}, $\mathbb{E} [\| Y^{(j)} \|^4 ]$ can be bounded as %Using Corollary \ref{cor:ExpectedNorm},
	\begin{equation*}
		\mathbb{E}\left[\| Y^{(j)} \|^4 \right] %\leq \mathbb{E}\left[\| Z^{(j)} \|^4 \right] + q \mathbb{E}\left[ \| Y^{(j)} \|^4 - \| Z^{(j)} \|^4 \Big| Y^{(j)} \neq Z^{(j)} \right]
		\leq \mathbb{E}\left[\| Z^{(j)} \|^4 \right] + q \left(2\lambda_{\max}(\Sigma(\theta_1)^{-1/2})R \right)^4.
	\end{equation*}
	Observe that since the Boltzmann distribution with parameter $\theta_1$ is isotropic, $\mathbb{E}[\| Z^{(j)} \|^2 ] = n$ for all $j$. Hence, by Lemma \ref{lemma:LogConcaveTail}, we have $\mathbb{P}\{ \| Z^{(j)} \| > t \sqrt{n} \} \leq e^{1-t}$ for all $t > 1$. By a change of variables $s = n^2 t^4$,
	\begin{align*}
		\mathbb{E} [ \| Z^{(j)} \|^4 ] &= \int_0^\infty \mathbb{P} \{ \| Z^{(j)} \|^4 > s \} \diff s\\
		&= \int_0^{\infty} \mathbb{P} \{ \| Z^{(j)} \| > t \sqrt{n} \} 4 t^3 n^2 \diff t\\
		&\leq \int_0^1 4 t^3 n^2 \diff t + \int_1^\infty 4 t^3 n^2 e^{1-t} \diff t = 65 n^2.
	\end{align*}
	
	It remains to bound $\mathbb{E}[ ( (Y^{(j)})^\top Y^{(k)})^2]$ for $k \neq j$. If follows from the $6q$-independence of $Y^{(j)}$ and $Y^{(k)}$, combined with Lemma \ref{lemma:NearIndependentProduct} and Lemma \ref{lemma:NearIndependenceExpectation}, that
	\begin{align*}
		\mathbb{E}\left[ \left( (Y^{(j)})^\top Y^{(k)} \right)^2 \right] &= \mathbb{E} \left[ \sum_{r=1}^n \sum_{s=1}^n \left( Y^{(j)}_r Y^{(k)}_r \right) \left( Y^{(j)}_s Y^{(k)}_s \right) \right]\\
		&\leq \sum_{r=1}^n \sum_{s=1}^n \left( \mathbb{E} \left[ Y^{(j)}_r Y^{(j)}_s \right] \mathbb{E} \left[ Y^{(k)}_r Y^{(k)}_s \right] + 4(6q)\left(2\lambda_{\max}(\Sigma(\theta_1)^{-1/2})R \right)^2 \right).
	\end{align*}
	Applying the now familiar Lemma \ref{lemma:DivineIntervention} and \eqref{eq:DivineInterventionExpectation} once more, we find for all $j$ that
	\begin{equation*}
		\mathbb{E} \left[ Y^{(j)}_r Y^{(j)}_s \right] \leq \mathbb{E} \left[ Z^{(j)}_r Z^{(j)}_s \right] + q \left(2\lambda_{\max}(\Sigma(\theta_1)^{-1/2})R \right)^2.
	\end{equation*}
	Note that because $Z^{(j)}$ is isotropic, $\mathbb{E} [ Z^{(j)}_r Z^{(j)}_s ]$ is one if $r=s$ and zero otherwise. Therefore,
	\begin{align*}
		&\sum_{r=1}^n \sum_{s=1}^n \mathbb{E} \left[ Y^{(j)}_r Y^{(j)}_s \right] \mathbb{E} \left[ Y^{(k)}_r Y^{(k)}_s \right]\\
		&\leq n \left( 1 + q \left(2\lambda_{\max}(\Sigma(\theta_1)^{-1/2})R \right)^2 \right)^2 + n(n-1) \left( q \left(2\lambda_{\max}(\Sigma(\theta_1)^{-1/2})R \right)^2 \right)^2.
	\end{align*}
	
	In summary, the values of $N$ and $q$ in \eqref{eq:CovQualityNumberSamples} give us
	\begin{align*}
		\mathbb{E} [\tr (S^2)]
		&\leq \frac{1}{N} \left( 65n^2 + q \left(2\lambda_{\max}(\Sigma(\theta_1)^{-1/2})R \right)^4 \right)\\
		&+ \frac{N(N-1)}{N^2} \Bigg[ n \left( 1 + q \left(2\lambda_{\max}(\Sigma(\theta_1)^{-1/2})R \right)^2 \right)^2\\
		&+ n(n-1) \left( q \left(2\lambda_{\max}(\Sigma(\theta_1)^{-1/2})R \right)^2 \right)^2 + 4 n^2(6q)\left(2\lambda_{\max}(\Sigma(\theta_1)^{-1/2})R \right)^2 \Bigg]\\
		&\leq n + 0.14 p \varepsilon^2.
	\end{align*}
	Combined with \eqref{eq:ProofCovQualityTrace}, the above yields the following upper bound on $\mathbb{E} [\tr([S-I]^2)] = \mathbb{E} [\tr(S^2 - 2S)] + n$:
	\begin{equation*}
		\mathbb{E} [\tr([S-I]^2)] \leq n + 0.14 p \varepsilon^2 - 2n + 2qn \left(2R\lambda_{\max}(\Sigma(\theta_1)^{-1/2}) \right)^2 + n
		\leq \frac{ p \varepsilon^2}{7}.
	\end{equation*}
%	\begin{align*}
%		\mathbb{E} [\tr([S-I]^2)] &\leq n + 0.14 p \varepsilon^2 - 2n + 2qn \left(2R\lambda_{\max}(\Sigma(\theta_1)^{-1/2}) \right)^2 + n\\
%		&= 0.14 p \varepsilon^2 + 2qn \left(2R\lambda_{\max}(\Sigma(\theta_1)^{-1/2}) \right)^2
%		\leq \frac{ p \varepsilon^2}{7}.
%	\end{align*}
	
	By \eqref{eq:ProofCovQualitySpectralRadius} and Markov's inequality,
	\begin{equation*}
		\mathbb{P}\left\{ \rho(S-I) > \frac{34}{35} \frac{\varepsilon}{1+\varepsilon} \right\}
		\leq \frac{\mathbb{E} [\tr([S-I]^2)]}{\left(\frac{34}{35} \frac{\varepsilon}{1+\varepsilon}\right)^2}
		\leq \frac{p \varepsilon^2/7}{ \left(\frac{34}{35} \frac{\varepsilon}{1+\varepsilon}\right)^2 } = p(1+\varepsilon)^2 \frac{175}{1156}.
	\end{equation*}
	Using Theorem \ref{thm:SampleMeanNorm} for $\alpha^2 = \frac{\varepsilon}{35(1+\varepsilon)}$, we find that the $N$ and $q$ in \eqref{eq:CovQualityNumberSamples} yield the quality guarantee
	\begin{equation*}
		\mathbb{P}\left\{ \| \widehat{Y} \| > \sqrt{\frac{\varepsilon}{35(1+\varepsilon)}} \right\} \leq \frac{2 p\varepsilon^2}{490 n \frac{\varepsilon}{35(1+\varepsilon)}} \leq \frac{p\varepsilon(1+\varepsilon)}{7}.
	\end{equation*}
	Thus, we have for all $\varepsilon \in (0,1)$,
	\begin{equation*}
		\mathbb{P}\left\{ \rho(S-I) \leq \frac{34}{35} \frac{\varepsilon}{1+\varepsilon} \wedge \| \widehat{Y} \| \leq \sqrt{\frac{\varepsilon}{35(1+\varepsilon)}} \right\} \geq 1 - p(1+\varepsilon)^2 \frac{175}{1156} - p\frac{\varepsilon(1+\varepsilon)}{7} \geq 1-p.
	\end{equation*}
	
	We can now verify the inequalities in \eqref{eq:ProofCovQualityCondition2} to complete the proof. With probability at least $1-p$,
	\begin{equation*}
		\frac{1}{1+\varepsilon} + \| \widehat{Y} \|^2 \leq \frac{1 + \varepsilon/35}{1+\varepsilon} = 1 - \frac{34}{35} \frac{\varepsilon}{1+\varepsilon} \leq 1 - \rho(S-I),
	\end{equation*}
	thereby verifying the first inequality in \eqref{eq:ProofCovQualityCondition2} for all unit vectors $v$. The second inequality in \eqref{eq:ProofCovQualityCondition2} can be shown by noting that with probability at least $1-p$,
	\begin{equation*}
		1 + \rho(S-I) \leq 1 + \frac{\varepsilon}{1 + \varepsilon} \leq \frac{1}{1-\varepsilon} \leq \frac{1}{1-\varepsilon} + (v^\top \widehat{Y})^2,
	\end{equation*}
	for all $v \in \mathbb{R}^n$.
\end{proof}

One might expect that if all eigenvalue-eigenvector pairs are approximated well enough for \eqref{eq:CovQualityStatement} to hold, this also implies a similar statement relating $\widehat{\Sigma}^{-1}$ and $\Sigma(\theta_1)^{-1}$. This is confirmed by the following `folklore' result from linear algebra.
%The following lemma shows that this is indeed the case.
\begin{lemma}
	\label{lemma:InverseQuality}
	Let $\Sigma$ and $\widehat{\Sigma}$ be two symmetric invertible matrices such that
	\begin{equation*}
		(1-\varepsilon) v^\top \widehat{\Sigma} v \leq v^\top \Sigma v \leq (1+\varepsilon) v^\top \widehat{\Sigma} v \qquad \forall v \in \mathbb{R}^n,
%		\label{eq:LemmaInverseQualityAssumption}
	\end{equation*}
	for some constant $0 \leq \varepsilon < \frac{1}{2}(\sqrt{5} - 1) \approx 0.618$. Then, for all $u \in \mathbb{R}^n$,
	\begin{equation*}
		\left(1 - \frac{2\varepsilon^2 + \varepsilon}{1 + \varepsilon^2 + \varepsilon} \right) u^\top \widehat{\Sigma}^{-1} u \leq u^\top \Sigma^{-1} u \leq \left(1 + \frac{2\varepsilon^2 + 3\varepsilon}{1 - \varepsilon^2 - \varepsilon} \right) u^\top \widehat{\Sigma}^{-1} u.
%		\label{eq:LemmaInverseQualityResult}
	\end{equation*}
\end{lemma}

We can combine this lemma with Theorem \ref{thm:CovQuality} to bound the number of samples required to approximate both $\Sigma(\theta_1)$ and $\Sigma(\theta_1)^{-1}$ to a desired accuracy.
\begin{corollary}
	\label{cor:CovAndInverseQuality}
	Let $K \subseteq \mathbb{R}^n$ be a convex body, and let $\langle \cdot, \cdot \rangle$ be the Euclidean inner product. Suppose $K$ is contained in a Euclidean ball with radius $R > 0$. Let $p \in (0,1)$, $\epsilon \in [0, 1)$, and $0 < \epsilon_1 \leq \sqrt{13} - 3 \approx 0.606$.
	Let $\theta_0, \theta_1 \in \mathbb{R}^n$ such that $\Delta \theta := \| \theta_1 - \theta_0 \|_{\theta_0} < 1$ and $\Delta x := \| x(\theta_1) - x(\theta_0) \|_{x(\theta_0)}^* < 1$. Suppose we have an invertible matrix $\widehat{\Sigma}(\theta_0)$ such that
	\begin{equation*}
		(1-\epsilon) y^\top \widehat{\Sigma}(\theta_0)^{-1} y \leq y^\top \Sigma(\theta_0)^{-1} y \leq (1+\epsilon) y^\top \widehat{\Sigma}(\theta_0)^{-1} y \qquad \forall y \in \mathbb{R}^n.
	\end{equation*}
	Pick
	\begin{equation*}
		N \geq \frac{7840 n^2}{p \epsilon_1^2}, \qquad q \leq \frac{p \epsilon_1^2}{799680 n^2 R^4} \lambda_{\min}(\Sigma(\theta_1))^2.
	\end{equation*}
	Let $X_0$ be a random starting point drawn from a Boltzmann distribution with parameter $\theta_0$. Let $Y^{(1)}, ..., Y^{(N)}$ be the end points of $N$ hit-and-run random walks applied to the Boltzmann distribution with parameter $\theta_1$ having starting point $X_0$, where the directions are drawn from a $\mathcal{N}(0,\widehat{\Sigma}(\theta_0))$-distribution, and each walk has length $\ell$ given by \eqref{eq:WalkLength}. (Note that $\ell$ depends on $\epsilon$, $n$, $q$, $\Delta\theta$, and $\Delta x$.)
	Then, with probability $1-p$, the empirical covariance matrix $\widehat{\Sigma} \approx \Sigma(\theta_1)$ as defined in \eqref{eq:DefinitionEmpiricalCovariance} satisfies
	\begin{align}
		(1-\epsilon_1) v^\top \widehat{\Sigma} v &\leq v^\top \Sigma(\theta_1) v \leq (1+\epsilon_1) v^\top \widehat{\Sigma} v &&\forall v \in \mathbb{R}^n, \label{eq:CorCovRegular}\\
		(1-\epsilon_1) v^\top \widehat{\Sigma}^{-1} v &\leq v^\top \Sigma(\theta_1)^{-1} v \leq (1+\epsilon_1) v^\top \widehat{\Sigma}^{-1} v &&\forall v \in \mathbb{R}^n.
		\label{eq:CorCovInverse}
	\end{align}
\end{corollary}
\begin{proof}
	We apply Theorem \ref{thm:CovQuality} with $\varepsilon = \frac{1}{4} \epsilon_1 < \epsilon_1$. Thus, \eqref{eq:CorCovRegular} holds. For $\epsilon_1 \leq \sqrt{13} - 3$,
	\begin{equation*}
		\tfrac{1}{4} \epsilon_1 \leq \frac{\sqrt{(3+\epsilon_1)^2 + 4 \epsilon_1(2+\epsilon_1)} - (3 + \epsilon_1)}{2(2+\epsilon_1)},
	\end{equation*}
	where the right hand side is chosen such that Lemma \ref{lemma:InverseQuality} shows that \eqref{eq:CorCovInverse} holds.
\end{proof}

\section{Short-Step IPM Using the Entropic Barrier}
\label{sec:ShortStepIPM}
From now on, let the reference inner product $\langle \cdot, \cdot \rangle$ be the Euclidean dot product.
Before we show how the results from the previous sections may be applied to interior point methods, we have to fix some notation. With $c$ as in \eqref{eq:MainProblem}, define
\begin{equation*}
	f^{*,\eta}(x) = \eta \langle c, x \rangle + f^*(x),
\end{equation*}
and let $z(\eta)$ be the minimizer of $f^{*,\eta}$, that is, $g^*(z(\eta)) = -\eta c$.
%Since Lemmas \ref{lemma:PropertiesDerivativesConjugate} and \ref{lemma:InteriorKImageG} imply $g^*(x) = g^{-1}(x)$ for $x \in \interior K$, we will often write $\theta(x)$ for $g^*(x)$ to keep the notation intuitive.
Moreover, it follows from Lemma \ref{lemma:PropertiesDerivativesConjugate} that $H^*(z(\eta)) = [H(g^*(z(\eta)))]^{-1} = \Sigma(-\eta c)^{-1}$.
%As noted before, $H(\theta)$ can be interpreted as the covariance matrix of a Boltzmann distribution with parameter $\theta$ under the Euclidean inner product, and therefore we will often write $\Sigma(\theta)$ for $H(\theta)$.

\begin{algorithm}[!ht]
	\caption{Short-step IPM using the entropic barrier from de Klerk, Glineur and Taylor \cite{deklerk2017worst}}
	\label{alg:ShortStepIPM}
	\begin{algorithmic}[1]
		\Input Tolerances $\widetilde{\epsilon}, \bar{\epsilon} > 0$ and $\epsilon \in (0, \sqrt{13}-3)$; entropic barrier parameter $\vartheta \leq n + o(n)$; objective $c \in \mathbb{R}^n$; radius $r > 0$ of Euclidean ball contained in $K$; radius $R > 0$ of Euclidean ball containing $K$; central path proximity parameter $\delta > 0$; an $x_0 \in K$ and $\eta_0 > 0$ such that $\| x_0 - z(\eta_0)\|_{z(\eta_0)}^* \leq \frac{1}{2} \delta$; approximation $\widehat{\Sigma}(-\eta_0 c)$ of $\Sigma(-\eta_0 c)$ satisfying \eqref{eq:ApproximateCovCondition} and \eqref{eq:ApproximateInverseCovCondition}; $X_0 \in K$ drawn randomly from a Boltzmann distribution with parameter $-\eta_0 c$;
		growth rate $\beta > 0$.
%		uniform sample $X_0 \in K$; approximation $\widehat{\Sigma}(0) \approx \Sigma(0)$ satisfying \eqref{eq:ApproximateCovCondition}; $\eta_1 > 0$; membership oracle $\mathcal{O}_K$;
		\Output $x_k$ such that $\langle c, x_k \rangle - \min_{x \in K} \langle c, x \rangle \leq \bar{\epsilon}$ at terminal iteration $k$.
		\Statex
		
		\State $\widehat{\epsilon} \gets \widetilde{\epsilon} \sqrt{\frac{1+\epsilon}{1-\epsilon}} + \sqrt{\frac{2 \epsilon}{1 - \epsilon}}$
		\State $d \gets \left(\frac{1}{2} \delta + \beta + \frac{3 \beta^2}{(1-\beta)^3} \right) \Big/ \left(1 - \left( \beta + \frac{3 \beta^2}{(1-\beta)^3} \right)\right)$
		\State $\gamma \gets \frac{2(1-d)^4 - \widehat{\epsilon}(1+(1-d)^4)}{(1 - \widehat{\epsilon})(1-d)^2((1-d)^4 + 1)}$
%		\For{$j \in \{1, ..., N' \}$}
%			\State $Y_j \gets \Call{HitAndRun}{x \mapsto 1, \mathcal{O}_K, \widehat{\Sigma}(0), X_{0}, \ell}$
%		\EndFor
%		\State $x_0 \gets \frac{1}{N'} \sum_{j=1}^{N'} Y_j$
		\State $m \gets \left\lceil \log\left( \frac{\vartheta(1 + \delta/2)}{\eta_0 \bar{\epsilon}} \right) \big/ \log \left( 1 + \frac{\beta}{\sqrt{\vartheta}} \right) \right\rceil$
		\State $N \gets \left\lceil \frac{15680 n^2 m}{p \epsilon^2} + 2\right\rceil$
		\State $q \gets \min \left\{ \frac{p}{m N}, \frac{N-2}{N}\frac{p \epsilon^2}{1599360 m n^2 R^4} \frac{1}{4096} \left( \frac{r}{n+1} \right)^4 \left( \frac{\bar{\epsilon}}{4R\vartheta(1+\delta/2)(1+\beta/\sqrt{\vartheta})} \right)^{8 \sqrt{\vartheta} + 4} \right\}$
		\State $k \gets 0$
		\While{$\vartheta(1+\frac{1}{2}\delta) / \eta_k > \bar{\epsilon}$} \label{line:IPMWhile}
%			\State Find approximation $\widehat{\Sigma}(-\eta_k c)$ of $\Sigma(-\eta_k c)$ satisfying \eqref{eq:ApproximateCovCondition} and \eqref{eq:ApproximateInverseCovCondition}
			\State \algmultiline{Generate hit-and-run samples $Y^{(1)}, ..., Y^{(N)}$, and $X_{k+1}$ from the Boltzmann distribution with parameter $-\eta_{k+1} c$, starting from $X_k$, drawing directions from $\mathcal{N}(0, \widehat{\Sigma}(-\eta_k c))$, using $\ell$ steps as in \eqref{eq:WalkLength} with bounds on $\Delta x$ and $\Delta\theta$ given in \eqref{eq:WalkLengthMainDeltaX} and \eqref{eq:WalkLengthMainDeltaTheta}} \label{line:IPMSample}
			\State \algmultiline{Find approximation $\widehat{\Sigma}(-\eta_{k+1} c)$ of $\Sigma(-\eta_{k+1} c)$ satisfying \eqref{eq:ApproximateCovCondition} and \eqref{eq:ApproximateInverseCovCondition} using samples $Y^{(1)}, ..., Y^{(N)}$}
			\State Find approximation $\widehat{\theta}(x_k)$ of $\theta(x_k)$ satisfying \eqref{eq:ThetaCondition} using $X_{k+1}$ \label{line:IPMTheta}
			\State $x_{k+1} \gets x_{k} - \gamma \widehat{\Sigma}(-\eta_k c) [\eta_{k+1} c + \widehat{\theta}(x_{k})]$ \Comment{$x_{k+1} \approx z(\eta_{k+1})$}
			\State $\eta_{k+1} \gets (1 + \frac{\beta}{\sqrt{\vartheta}}) \eta_k$
			\State $k \gets k+1$
		\EndWhile
		\State \Return $x_k$
	\end{algorithmic}
\end{algorithm}

A more detailed description of the algorithm by de Klerk, Glineur and Taylor \cite{deklerk2017worst} is given in Algorithm \ref{alg:ShortStepIPM}. We note that an approximation of $\Sigma(-\eta_0 c)$ can be obtained by hit-and-run sampling using the algorithm by Kalai and Vempala \cite{kalai2006simulated}. This algorithm also generates an $X_0$ following the Boltzmann distribution with parameter $-\eta_0 c$.

The assumptions in de Klerk, Glineur and Taylor \cite{deklerk2017worst} include that one can find an estimate
$\widehat{\Sigma}(-\eta_{k+1} c)$ of $\Sigma(-\eta_{k+1} c)$ such that
\begin{align}
	(1-\epsilon) y^\top \widehat{\Sigma}(-\eta_{k+1} c) y &\leq y^\top \Sigma(-\eta_{k+1} c) y \leq (1+\epsilon) y^\top \widehat{\Sigma}(-\eta_{k+1} c) y && \forall y \in \mathbb{R}^n, \label{eq:ApproximateCovCondition}\\
	(1-\epsilon) y^\top \widehat{\Sigma}(-\eta_{k+1} c)^{-1} y &\leq y^\top \Sigma(-\eta_{k+1} c)^{-1} y \leq (1+\epsilon) y^\top \widehat{\Sigma}(-\eta_{k+1} c)^{-1} y && \forall y \in \mathbb{R}^n \label{eq:ApproximateInverseCovCondition}.
\end{align}
Note that the number of samples required to find such a $\widehat{\Sigma}(-\eta_{k+1} c)$ is given by Corollary \ref{cor:CovAndInverseQuality}.
%$\widehat{\Sigma}(-\eta_k c)$ of $\Sigma(-\eta_k c)$ such that
%\begin{align}
%	(1-\epsilon) y^\top \widehat{\Sigma}(-\eta_k c) y &\leq y^\top \Sigma(-\eta_k c) y \leq (1+\epsilon) y^\top \widehat{\Sigma}(-\eta_k c) y && \forall y \in \mathbb{R}^n, \label{eq:ApproximateCovCondition}\\
%	(1-\epsilon) y^\top \widehat{\Sigma}(-\eta_k c)^{-1} y &\leq y^\top \Sigma(-\eta_k c)^{-1} y \leq (1+\epsilon) y^\top \widehat{\Sigma}(-\eta_k c)^{-1} y && \forall y \in \mathbb{R}^n \label{eq:ApproximateInverseCovCondition}.
%\end{align}
%Note that the number of samples required to find such a $\widehat{\Sigma}(-\eta_k c)$ is given by Corollary \ref{cor:CovAndInverseQuality}.
Moreover, at some point $x_k \in \interior K$ we want to find an approximate gradient $\widehat{g}_{z(\eta_{k+1})}^{*,\eta_{k+1}}(x_k)$ of $f^{*,\eta_{k+1}}$ in the sense that
\begin{equation*}
%	\| \widehat{g}_{z(\eta)}^{*,\eta}(x) - g_{z(\eta)}^{*,\eta}(x) \|^*_{z(\eta)} \leq \epsilon' \| g_{z(\eta)}^{*,\eta}(x) \|^*_{z(\eta)},
	\| \widehat{g}_{z(\eta_{k+1})}^{*,\eta_{k+1}}(x_k) - g_{z(\eta_{k+1})}^{*,\eta_{k+1}}(x_k) \|^*_{z(\eta_{k+1})} \leq \widetilde{\epsilon} \| g_{z(\eta_{k+1})}^{*,\eta_{k+1}}(x_k) \|^*_{z(\eta_{k+1})},
\end{equation*}
for some $\widetilde{\epsilon} > 0$. Equivalently, we are looking for an approximation $\widehat{\theta}(x_k)$ of $\theta(x_k)$ such that
\begin{equation*}
	\left\langle \widehat{\theta}(x_k) - \theta(x_k), H^*(z(\eta_{k+1}))^{-1} \left[\widehat{\theta}(x_k) - \theta(x_k) \right] \right\rangle \leq (\widetilde{\epsilon})^2 \left\langle \eta_{k+1} c + \theta(x_k), H^*(z(\eta_{k+1}))^{-1} [\eta_{k+1} c + \theta(x_k)] \right\rangle.
%	&\left\langle \widehat{\theta}(x_k) - \theta(x_k), H^*(z(\eta_{k+1}))^{-1} \left[\widehat{\theta}(x_k) - \theta(x_k) \right] \right\rangle \\
%	&\leq (\widetilde{\epsilon})^2 \left\langle \eta_{k+1} c + \theta(x_k), H^*(z(\eta_{k+1}))^{-1} [\eta_{k+1} c + \theta(x_k)] \right\rangle.
\end{equation*}
%where $ H^*(z(\eta_{k+1}))^{-1} = \Sigma(-\eta_{k+1} c)$ can be considered invariant for the purpose of finding $\widehat{\theta}(x_k)$.
Since $\theta(z(\eta_{k+1})) = -\eta_{k+1} c$, it will suffice to find a $\widehat{\theta}(x_k)$ such that
\begin{equation}
	\label{eq:ThetaCondition}
	\left\| \widehat{\theta}(x_k) - \theta(x_k) \right\|_{\Sigma(-\eta_{k+1} c)} \leq \widetilde{\epsilon} \left\| \theta(z(\eta_{k+1})) - \theta(x_k) \right\|_{\Sigma(-\eta_{k+1} c)}.
\end{equation}
%With these assumptions satisfied, they prove the following result.
%\begin{theorem}[Theorem 7.6]
%	
%\end{theorem}

%\subsection{Hessian Approximation}
%The number of samples required to find an approximation of $\Sigma(-\eta_k c)$ that satisfies \eqref{eq:ApproximateCovCondition} was given by Theorem \ref{thm:CovQuality}.
%It follows easily from Lemma \ref{lemma:InverseQuality} that \eqref{eq:ApproximateCovCondition} and \eqref{eq:ApproximateInverseCovCondition} both hold if one sets
%\begin{equation*}
%	\varepsilon \leq \frac{\sqrt{(3+\epsilon)^2 + 4 \epsilon(2+\epsilon)} - (3 + \epsilon)}{2(2+\epsilon)},
%\end{equation*}
%in Theorem \ref{thm:CovQuality}. If $\epsilon \leq \sqrt{13} - 3 \approx 0.606$, then $\varepsilon \leq \frac{1}{4} \epsilon$ also suffices.

\subsection{Gradient Approximation}
%This section will focus on .
One immediate problem with finding a $\widehat{\theta}(x_k)$ satisfying \eqref{eq:ThetaCondition} is that the right hand side of \eqref{eq:ThetaCondition} is unknown: it depends on $\theta(x_k)$, the very vector we are trying to approximate. Hence, we will first present a gradient descent algorithm to find a $\widehat{\theta}(x_k)$ such that $\| \widehat{\theta}(x_k) - \theta(x_k) \|_{\Sigma(-\eta_{k+1} c)} \leq b$ for any given $b > 0$. In the next section, we will then lower bound the right hand side of \eqref{eq:ThetaCondition} by known values.
%attempt to lower bound $\left\| \theta(z(\eta_{k+1})) - \theta(x_k) \right\|_{\Sigma(-\eta_{k+1} c)}$, and then present a gradient descent algorithm that yields the desired approximation $\widehat{\theta}(x_k)$ in the next section.

To approximate $\theta(x_k)$ in Algorithm \ref{alg:ShortStepIPM}, we will use the approach proposed by Abernethy and Hazan \cite{abernethy2015faster}, which was also discussed in Section \ref{subsec:AbernethyHazan}. Let $x \in \interior K$, and consider the unconstrained minimization problem
\begin{equation*}
	\min_{\theta \in \mathbb{R}^n} \Psi(\theta), \quad \text{where} \quad \Psi(\theta) = f(\theta) - \langle \theta, x \rangle.
\end{equation*}
Since the Fr\'echet derivative of $\Psi$ is $D \Psi(\theta) = \mathbb{E}_\theta[X] - x$,
%\begin{equation*}
%	$D \Psi(\theta) = \mathbb{E}_\theta[X] - x$, and $D^2 \Psi(\theta) = \Sigma(\theta)$,
%\end{equation*}
the gradient of $\Psi$ is zero at the $\theta \in \mathbb{R}^n$ such that $g(\theta) = \mathbb{E}_\theta[X] = x$, which is by definition $\theta(x)$.
Moreover, $D^2 \Psi(\theta) = \Sigma(\theta)$, and thus $\Psi$ is strictly convex, meaning we can use unconstrained minimization techniques to approximate $\theta(x)$. Note that because the Hessians of $f$ and $\Psi$ are the same, $\Psi$ is self-concordant. Although the gradient and Hessian of $\Psi$ are not readily available, they can be estimated through sampling. We can then again apply the inexact gradient descent results from de Klerk, Glineur and Taylor \cite{deklerk2017worst}. This is formalized in Algorithm \ref{alg:ThetaSubroutine}. To distinguish parameters in Algorithm \ref{alg:ThetaSubroutine} from those with a similar function in Algorithm \ref{alg:ShortStepIPM}, we add a prime to those parameters in Algorithm \ref{alg:ThetaSubroutine} (e.g. we use $\epsilon'$ in stead of $\epsilon$).
Among other things, Algorithm \ref{alg:ThetaSubroutine} depends on two vectors $\theta_0$ and $\overline{\theta}$. Intuitively, $\theta_0$ is our initial estimate for $\theta(x)$, while $\overline{\theta}$ is used to fix an inner product. (As we will see in Section \ref{subsec:Complexity}, we cannot take $\overline{\theta} = \theta_0$ to analyze Algorithm \ref{alg:ShortStepIPM}.)

\begin{algorithm}[!ht]
	\caption{Approximation routine for $\theta(x)$ %\textproc{ApproxTheta}($x, \theta_0, \overline{\theta}, B, b, \widehat{\Sigma}(\overline{\theta})^{-1}, \epsilon, \epsilon'$)
	}
	\label{alg:ThetaSubroutine}
	\begin{algorithmic}[1]
		\Input $x \in \interior K$; $0 < b < B$; $\theta_0, \overline{\theta} \in \mathbb{R}^n$ such that $\| \theta_0 - \theta(x) \|_{\overline{\theta}} \leq B$; $d' < \frac{1}{3}$ such that $2 B + \| \theta_0 - \overline{\theta} \|_{\overline{\theta}} \leq d'$; approximation $\widehat{\Sigma}(\overline{\theta})^{-1}$ of $\Sigma(\overline{\theta})^{-1}$ and $\epsilon' \in (0,1)$ satisfying \eqref{eq:ApproximateCovCondition} and \eqref{eq:ApproximateInverseCovCondition}; $\overline{X} \in K$ drawn randomly from a Boltzmann distribution with parameter $\overline{\theta}$; radius $r > 0$ of Euclidean ball contained in $K$; radius $R > 0$ of Euclidean ball containing $K$; relative gradient error $\widetilde{\epsilon}' > 0$ such that $\widetilde{\epsilon}'\sqrt{\frac{1+\epsilon'}{1-\epsilon'}} + \sqrt{\frac{2\epsilon'}{1-\epsilon'}} < \frac{2(1-d')^4}{1 + (1-d')^4}$; $C \in (0,1)$ such that $\| x - x(\overline{\theta}) \|_{\Sigma(\overline{\theta})^{-1}} \leq C$; failure probability $p' \in (0,1)$.
		\Output $\theta_i$ such that $\| \theta_i - \theta(x) \|_{\overline{\theta}} \leq b$ at terminal iteration $i$.
		\Statex
		
		\State $\widehat{\epsilon}' \gets \widetilde{\epsilon}'\sqrt{\frac{1+\epsilon'}{1-\epsilon'}} + \sqrt{\frac{2\epsilon'}{1-\epsilon'}}$
		\State $\gamma' \gets \frac{2(1-d')^4 - \widehat{\epsilon}'(1+(1-d')^4)}{(1 - \widehat{\epsilon}')(1-d')^2((1-d')^4 + 1)}$
		\State \label{line:DefM'} $m' \gets \left\lceil \log( b / B ) \big/ \log \left( \frac{1 - (1-d')^4}{1 + (1-d')^4} + \widehat{\epsilon}' \right) \right\rceil$
%		\State $N' \gets \left\lceil n \ln^2(em'/p') \left(\frac{1 + b(1-C)}{b (1-C)^2} \left( 1 + \frac{1+ \widetilde{\epsilon}'}{\widetilde{\epsilon}'} \sqrt{\frac{1+\epsilon'}{1-\epsilon'}} \right) \frac{1-2d'}{1-3d'} \right)^2 \right\rceil$
		\State \label{line:DefN'} $N' \gets \left\lceil 2n \frac{m'}{p'} \left(\frac{1 + b(1-C)}{b (1-C)^2} \left( 1 + \frac{1+ \widetilde{\epsilon}'}{\widetilde{\epsilon}'} \sqrt{\frac{1+\epsilon'}{1-\epsilon'}} \right) \frac{1-2d'}{1-3d'} \right)^2 \right\rceil$
		\State \label{line:DefQ'}
		\algmultiline{$q' \gets \frac{p'}{204 m' n R^2} \left(\frac{b (1-C)^2}{1 + b(1-C)} \frac{\widetilde{\epsilon}' \sqrt{1 - \epsilon'}}{(1+ \widetilde{\epsilon}')\sqrt{1+\epsilon'} + \widetilde{\epsilon}'\sqrt{1-\epsilon'}} \frac{1-3d'}{1-2d'}\right)^2 \times$\\
		$\frac{1}{64} \left( \frac{r}{n+1} \right)^2 \min \left\{ 1, \left( 4R \| \overline{\theta} \| + \frac{32 R d' (n+1)}{r} \max\{ 1, (4R \| \overline{\theta} \|)^{2 \sqrt{\vartheta} + 1} \} \right)^{-4 \sqrt{\vartheta} - 2} \right\}$}
		\State $i \gets 0$
		\While{$ B\left( \frac{1 - (1-d')^4}{1 + (1-d')^4} + \widehat{\epsilon}' \right)^i > b$}
			\State \algmultiline{Generate hit-and-run samples $Y^{(1)}, ..., Y^{(N')}$ from the Boltzmann distribution with parameter $\theta_i$, starting from $\overline{X}$, drawing directions from $\mathcal{N}(0, \widehat{\Sigma}(\overline{\theta}))$, using $\ell'$ steps as in \eqref{eq:WalkLengthTheta}}
%			\EndWhile
			\State Compute sample mean: $\widehat{x}(\theta_i) \gets \frac{1}{N'} \sum_{j=1}^{N'} Y^{(j)}$
			\If{$\| \widehat{x}(\theta_i) - x \|_{\widehat{\Sigma}(\overline{\theta})^{-1}} \leq \frac{b (1-C)^2}{1 + b(1-C)} \frac{1+ \widetilde{\epsilon}' }{(1+ \widetilde{\epsilon}')\sqrt{1+\epsilon'} + \widetilde{\epsilon}'\sqrt{1-\epsilon'}}$}
			\label{line:TestIf}
				\State \Return $\theta_i$ \Comment{$\| \theta_i - \theta(x) \|_{\overline{\theta}} \leq b$}
			\Else \label{line:TestElse}
				\State $\theta_{i+1} \gets \theta_i - \gamma' \widehat{\Sigma}(\overline{\theta})^{-1} (\widehat{x}(\theta_i) - x)$
				\State $i \gets i + 1$
			\EndIf
		\EndWhile
		\State \Return $\theta_i$
	\end{algorithmic}
\end{algorithm}

The following theorem shows that Algorithm \ref{alg:ThetaSubroutine} produces the desired result with high probability.
\begin{theorem}
	\label{thm:ThetaRoutineAnalysis}
	Let $p' \in (0,1)$.
	With probability at least $1-p'$, Algorithm \ref{alg:ThetaSubroutine} produces a $\widehat{\theta}(x)$ such that $\| \widehat{\theta}(x) - \theta(x) \|_{\overline{\theta}} \leq b$. The number of required samples hit-and-run steps is $m' N' \ell'$, if
	\begin{equation}
		\ell' = \left\lceil \frac{1+\epsilon'}{1-\epsilon'}\frac{64 e^2 n^3 10^{30}}{(1 - 2d')^2} \ln^2 \left( \sqrt{\frac{1+\epsilon'}{1-\epsilon'}} \frac{16 e n\sqrt{n} \exp(-2d')}{(q')^2 (1-d')^2(1-2d')} \right) \ln^3 \left( \frac{2 \exp(-2d')}{(q')^2  (1-d')^2} \right) \right\rceil,
		\label{eq:WalkLengthTheta}
	\end{equation}
	and $m'$, $N'$, $\epsilon'$, $q'$, and $d'$ are defined in Algorithm \ref{alg:ThetaSubroutine}, where $m'$, $q'$, and $N'$ depend on $p'$ and $b$ as in lines \ref{line:DefM'} -\ref{line:DefQ'} of Algorithm \ref{alg:ThetaSubroutine}.
%	\begin{equation*}
%		n \left(\frac{1 + b(1-C)}{b (1-C)^2} \left( 1 + \frac{1+ \widetilde{\epsilon}'}{\widetilde{\epsilon}'} \sqrt{\frac{1+\epsilon'}{1-\epsilon'}} \right) \frac{1-2d'}{1-3d'} \right)^2 \left\lceil \frac{\log( b / B )}{\log \left( \frac{1 - (1-d')^4}{1 + (1-d')^4} + \widehat{\epsilon}' \right)} \right\rceil \ln^2(em'/p'),
%	\end{equation*}
%	and the complexity of the algorithm is $O(m' N' \ell')$.
%	\begin{equation*}
%		O \left( n \left[\frac{1 + b(1-C)}{b (1-C)^2}
%		\left( 1 + \frac{1+ \widetilde{\epsilon}'}{\widetilde{\epsilon}'} \sqrt{\frac{1+\epsilon'}{1-\epsilon'}} \right)
%		\frac{1-2d'}{1-3d'} \right]^2 \left\lceil  \frac{\log( b / B )}{\log \left( \frac{1 - (1-d')^4}{1 + (1-d')^4} + \widehat{\epsilon}' \right)} \right\rceil \ln^2(em'/p')  \right).
%	\end{equation*}
\end{theorem}
\begin{proof}
	First, we want to show that all iterates $\theta_i$ lie close to $\overline{\theta}$ in the sense that $\| \theta_i - \overline{\theta} \|_{\overline{\theta}} \leq d'$ for all $i$. The argument is inductive: if $\| \theta_i - \overline{\theta} \|_{\overline{\theta}} \leq d'$, then the remainder of this proof will show that $\| \theta_{i+1} - \theta(x) \|_{\overline{\theta}} \leq \| \theta_i - \theta(x) \|_{\overline{\theta}}$, i.e. we are making progress to the minimizer, and therefore
	\begin{align}
		\| \theta_{i+1} - \overline{\theta} \|_{\overline{\theta}} &\leq \| \theta_{i+1} - \theta(x) \|_{\overline{\theta}} + \| \theta(x) - \theta_0 \|_{\overline{\theta}} + \| \theta_0 - \overline{\theta} \|_{\overline{\theta}} \nonumber\\
		&\leq 2 \| \theta(x) - \theta_0 \|_{\overline{\theta}} + \| \theta_0 - \overline{\theta} \|_{\overline{\theta}} \leq 2B + \| \theta_0 - \overline{\theta} \|_{\overline{\theta}} \leq d'.
		\label{eq:ProofThetaRoutineAnalysisProximityTheta}
	\end{align}
	
%	Before we can approximate the mean of the Boltzmann distribution with parameter $\theta_i$, Theorem \ref{thm:SampleMeanNorm} tells us we need to know a suitable lower bound on $\lambda_{\min}(\Sigma(\theta_i))$.
	We want to invoke Theorem \ref{thm:SampleMeanNorm} for $\theta_0 = \overline{\theta}$ and $\theta_1 = \theta_i$ to show that we can approximate the mean of the Boltzmann distribution with parameter $\theta_i$. Before we can do that, we need appropriate bounds on $\Delta \theta = \| \theta_i - \overline{\theta} \|_{\overline{\theta}}$, $\Delta x = \| x(\theta_i) - x(\overline{\theta}) \|^*_{x(\overline{\theta})}$ and $\lambda_{\min}(\Sigma(\theta_i))$.
	Note that $\Delta \theta = \| \theta_i - \overline{\theta} \|_{\overline{\theta}} \leq d' < 1$ by assumption, and Corollary \ref{cor:ProximityThetaTwoPoints} shows
	\begin{equation*}
		\Delta x = \| x(\theta_i) - x(\overline{\theta}) \|^*_{x(\overline{\theta})} \leq \frac{\| \theta_i - \overline{\theta} \|_{\overline{\theta}}}{1 - \| \theta_i - \overline{\theta} \|_{\overline{\theta}}} \leq \frac{d'}{1-d'} < 1.
	\end{equation*}
	By Theorem \ref{thm:LowerBoundSigma}, for all $i$,
	\begin{equation}
		\| \theta_i \| \leq \| \overline{\theta} \| + \| \theta_i - \overline{\theta} \| \leq \| \overline{\theta} \| + \frac{\| \theta_i - \overline{\theta} \|_{\overline{\theta}} }{\sqrt{\lambda_{\min}(\Sigma(\overline{\theta}))}} \leq \| \overline{\theta} \| + \frac{8 d' (n+1)}{r} \max\{ 1, (4R \| \overline{\theta} \|)^{2 \sqrt{\vartheta} + 1} \},
		\label{eq:ProofThetaRoutineAnalysis0}
	\end{equation}
	and therefore another application of Theorem \ref{thm:LowerBoundSigma} yields
	\begin{align*}
		\lambda_{\min}(\Sigma(\theta_i))
		&\geq \frac{1}{64} \left( \frac{r}{n+1} \right)^2 \min \left\{ 1, \left( 4R \| \theta_i \| \right)^{-4 \sqrt{\vartheta} - 2} \right\}\\
		&\geq \frac{1}{64} \left( \frac{r}{n+1} \right)^2 \min \left\{ 1, \left( 4R \| \overline{\theta} \| + \frac{32 R d' (n+1)}{r} \max\{ 1, (4R \| \overline{\theta} \|)^{2 \sqrt{\vartheta} + 1} \} \right)^{-4 \sqrt{\vartheta} - 2} \right\},
	\end{align*}
	where the second inequality is from \eqref{eq:ProofThetaRoutineAnalysis0}.
%	We continue with the approximation of $\mathbb{E}_{\theta_i}[X]$.
	We will now apply Theorem \ref{thm:SampleMeanNorm} with $p = p'/m'$ and
	\begin{equation*}
		\alpha = \left(\frac{b (1-C)^2}{1 + b(1-C)} \right) \left(\frac{\widetilde{\epsilon}' \sqrt{1 - \epsilon'}}{(1+ \widetilde{\epsilon}')\sqrt{1+\epsilon'} + \widetilde{\epsilon}'\sqrt{1-\epsilon'}} \right) \left(\frac{1-3d'}{1-2d'}\right).
	\end{equation*}
	Note that the values of $N'$ and $q'$ in Algorithm \ref{alg:ThetaSubroutine} are chosen such that Theorem \ref{thm:SampleMeanNorm} now shows that in each iteration $i$, it holds with probability at least $1-p'/m'$ that the sample mean $\widehat{x}(\theta_i)$ satisfies
	\begin{equation}
		\left\| \widehat{x}(\theta_i) - x(\theta_i) \right\|_{x(\theta_i)}^* \leq \left(\frac{b (1-C)^2}{1 + b(1-C)} \right) \left(\frac{\widetilde{\epsilon}' \sqrt{1 - \epsilon'}}{(1+ \widetilde{\epsilon}')\sqrt{1+\epsilon'} + \widetilde{\epsilon}'\sqrt{1-\epsilon'}} \right) \left(\frac{1-3d'}{1-2d'}\right),
		\label{eq:ProofThetaRoutineAnalysisEXi}
	\end{equation}
	where we recall that Lemma \ref{lemma:PropertiesDerivativesConjugate} demonstrates $\left\| \widehat{x}(\theta_i) - x(\theta_i) \right\|_{\Sigma(\theta_i)^{-1}} = \left\| \widehat{x}(\theta_i) - x(\theta_i) \right\|_{x(\theta_i)}^*$. We want to show that it follows from \eqref{eq:ProofThetaRoutineAnalysisEXi} that the $\| \cdot \|^*_{x(\overline{\theta})}$-norm of $\widehat{x}(\theta_i) - x(\theta_i)$ is also small.
	Since $\| \overline{\theta}- \theta_i \|_{\overline{\theta}} \leq d' < \frac{1}{3}$ by \eqref{eq:ProofThetaRoutineAnalysisProximityTheta}, the second inequality in \eqref{eq:SCf} shows
	\begin{equation*}
		\| \overline{\theta}- \theta_i \|_{\theta_i} \leq \frac{\| \overline{\theta}- \theta_i \|_{\overline{\theta}}}{1 - \| \overline{\theta}- \theta_i \|_{\overline{\theta}}} \leq \frac{d'}{1-d'} < \tfrac{1}{2}.
	\end{equation*}
	Therefore, the first inequality in \eqref{eq:ProximityThetaTwoPointsX} implies that
	\begin{equation}
		\| x(\overline{\theta}) - x(\theta_i) \|_{x(\theta_i)}^*
		\leq \frac{\| \overline{\theta}- \theta_i \|_{\theta_i} }{1 - \| \overline{\theta}- \theta_i \|_{\theta_i}}
		\leq \frac{\frac{d'}{1-d'}}{1 - \frac{d'}{1-d'}}
%		\leq \frac{ \frac{\| \overline{\theta}- \theta_i \|_{\overline{\theta}}}{1 - \| \overline{\theta}- \theta_i \|_{\overline{\theta}}} }{1 - \frac{\| \overline{\theta}- \theta_i \|_{\overline{\theta}}}{1 - \| \overline{\theta}- \theta_i \|_{\overline{\theta}}}}
		= \frac{d'}{1 - 2 d'} < 1.
		\label{eq:ProofThetaRoutineAnalysisThetaToTheta}
	\end{equation}
	Consequently, we can now use the second inequality in \eqref{eq:SCf*}, along with \eqref{eq:ProofThetaRoutineAnalysisEXi} and \eqref{eq:ProofThetaRoutineAnalysisThetaToTheta} to see that
	\begin{align}
		\left\| \widehat{x}(\theta_i) - x(\theta_i) \right\|_{x(\overline{\theta})}^*
		&\leq \frac{\left\| \widehat{x}(\theta_i) - x(\theta_i) \right\|_{x(\theta_i)}^*}{1 - \left\| {x}(\overline{\theta}) - x(\theta_i) \right\|_{x(\theta_i)}^*} \nonumber\\
		&\leq \frac{\left(\frac{b (1-C)^2}{1 + b(1-C)} \right) \left(\frac{\widetilde{\epsilon}' \sqrt{1 - \epsilon'}}{(1+ \widetilde{\epsilon}')\sqrt{1+\epsilon'} + \widetilde{\epsilon}'\sqrt{1-\epsilon'}} \right) \left(\frac{1-3d'}{1-2d'}\right) }{1 - \frac{d'}{1 - 2 d'}} \nonumber \\
		&= \left(\frac{b (1-C)^2}{1 + b(1-C)} \right)  \left(\frac{\widetilde{\epsilon}' \sqrt{1 - \epsilon'}}{(1+ \widetilde{\epsilon}')\sqrt{1+\epsilon'} + \widetilde{\epsilon}'\sqrt{1-\epsilon'}} \right).
		\label{eq:ProofThetaRoutineAnalysisEXbar}
	\end{align}
%	and therefore
%	\begin{align*}
%		\left\| \widehat{\mathbb{E}}_{\theta_i}[X] - \mathbb{E}_{\theta_i}[X] \right\|_{\Sigma(\overline{\theta})^{-1}}
%		&= \left\| \widehat{\mathbb{E}}_{\theta_i}[X] - \mathbb{E}_{\theta_i}[X] \right\|_{\mathbb{E}_{\overline{\theta}}[X]}^*\\
%		&\leq \frac{\| \widehat{\mathbb{E}}_{\theta_i}[X] - \mathbb{E}_{\theta_i}[X] \|_{\mathbb{E}_{\theta_i}[X]}^*}{ 1 - \| \mathbb{E}_{\overline{\theta}}[X] - \mathbb{E}_{\theta_i}[X] \|_{\mathbb{E}_{\theta_i}[X]}^* }\\
%		&\leq \frac{\frac{b (1-C)^2}{1 + b(1-C)} \frac{\widetilde{\epsilon}' \sqrt{1 - \epsilon'}}{(1+ \widetilde{\epsilon}')\sqrt{1+\epsilon'} + \widetilde{\epsilon}'\sqrt{1-\epsilon'}} \frac{1-3d'}{1-2d'} }{1 - \frac{d'}{1 - 2 d'}} \\
%		&= \frac{b (1-C)^2}{1 + b(1-C)} \frac{\widetilde{\epsilon}' \sqrt{1 - \epsilon'}}{(1+ \widetilde{\epsilon}')\sqrt{1+\epsilon'} + \widetilde{\epsilon}'\sqrt{1-\epsilon'}}.
%	\end{align*}
	
	Algorithm \ref{alg:ThetaSubroutine} distinguishes the following two cases in lines \ref{line:TestIf} and \ref{line:TestElse}:
	\begin{enumerate}[(a)]
		\item \label{case:Less}
		$\| \widehat{x}(\theta_i) - x \|_{\widehat{\Sigma}(\overline{\theta})^{-1}} \leq \left(\frac{b (1-C)^2}{1 + b(1-C)} \right) \left(\frac{1+ \widetilde{\epsilon}' }{(1+ \widetilde{\epsilon}')\sqrt{1+\epsilon'} + \widetilde{\epsilon}'\sqrt{1-\epsilon'}} \right)$;
		\item \label{case:More}
		$\| \widehat{x}(\theta_i) - x \|_{\widehat{\Sigma}(\overline{\theta})^{-1}} > \left(\frac{b (1-C)^2}{1 + b(1-C)} \right) \left( \frac{1+ \widetilde{\epsilon}' }{(1+ \widetilde{\epsilon}')\sqrt{1+\epsilon'} + \widetilde{\epsilon}'\sqrt{1-\epsilon'}} \right)$.
	\end{enumerate}
	In Case \eqref{case:Less}, we have
	\begin{align*}
		\| x(\theta_i) - x \|_{\Sigma(\overline{\theta})^{-1}}
		&\leq \| \widehat{x}(\theta_i) -  x(\theta_i) \|_{\Sigma(\overline{\theta})^{-1}} + \| \widehat{x}(\theta_i) - x \|_{\Sigma(\overline{\theta})^{-1}} && \text{(triangle inequality)}\\
		&\leq \| \widehat{x}(\theta_i) -  x(\theta_i) \|_{\Sigma(\overline{\theta})^{-1}} + \sqrt{1+\epsilon'} \| \widehat{x}(\theta_i) - x \|_{\widehat{\Sigma}(\overline{\theta})^{-1}} && \text{(by \eqref{eq:ApproximateInverseCovCondition})}\\
%		&\leq \frac{b (1-C)^2}{1 + b(1-C)} \left( \frac{\widetilde{\epsilon}' \sqrt{1 - \epsilon'}}{(1+ \widetilde{\epsilon}')\sqrt{1+\epsilon'} + \widetilde{\epsilon}'\sqrt{1-\epsilon'}} + \frac{1+ \widetilde{\epsilon}' }{(1+ \widetilde{\epsilon}')\sqrt{1+\epsilon'} + \widetilde{\epsilon}'\sqrt{1-\epsilon'}} \right)\\
		&\leq \frac{b (1-C)^2}{1 + b(1-C)}. && \text{(by \eqref{eq:ProofThetaRoutineAnalysisEXbar} and Case \eqref{case:Less})}
	\end{align*}
	Since $\| x - x(\overline{\theta}) \|_{x(\overline{\theta})}^* \leq C < 1$ by assumption, it follows that
	\begin{equation*}
		\| x - x(\theta_i) \|_{x(\overline{\theta})}^* + \| x - x(\overline{\theta}) \|_{x(\overline{\theta})}^* \leq \frac{b (1-C)^2}{1 + b(1-C)} + C = \frac{C + b(1-C)}{1 + b(1 - C)} < 1.
	\end{equation*}
	Thus, we may apply Lemma \ref{lemma:ProximityThetaGeneral} to see that
	\begin{align*}
		\| \theta_i - \theta(x) \|_{\overline{\theta}} &\leq \frac{1}{1 - \| x - x(\overline{\theta}) \|_{x(\overline{\theta})}^*} \frac{\| x - x(\theta_i) \|_{x(\overline{\theta})}^*}{1 - \| x - x(\theta_i) \|_{x(\overline{\theta})}^* - \| x - x(\overline{\theta}) \|_{x(\overline{\theta})}^*}\\
		&\leq \frac{1}{1-C} \frac{ \frac{b(1-C)^2}{1 + b(1-C)} }{1 - \frac{b(1-C)^2}{1 + b(1-C)} - C} = b.
	\end{align*}
	In conclusion, the stopping criterion for Algorithm \ref{alg:ThetaSubroutine} is met in Case \eqref{case:Less}.
	In Case \eqref{case:More}, we can expand \eqref{eq:ProofThetaRoutineAnalysisEXbar} to
	\begin{align}
		\left\| \widehat{x}(\theta_i) - x(\theta_i) \right\|_{x(\overline{\theta})}^*
		&\leq \left(\frac{b (1-C)^2}{1 + b(1-C)} \right)  \left(\frac{\widetilde{\epsilon}' \sqrt{1 - \epsilon'}}{(1+ \widetilde{\epsilon}')\sqrt{1+\epsilon'} + \widetilde{\epsilon}'\sqrt{1-\epsilon'}} \right) \nonumber\\
		&= \frac{\widetilde{\epsilon}'}{1+\widetilde{\epsilon}'} \sqrt{1-\epsilon'} \left(\frac{b (1-C)^2}{1 + b(1-C)} \right) \left( \frac{1+ \widetilde{\epsilon}' }{(1+ \widetilde{\epsilon}')\sqrt{1+\epsilon'} + \widetilde{\epsilon}'\sqrt{1-\epsilon'}} \right) \nonumber\\
		&< \frac{\widetilde{\epsilon}'}{1+\widetilde{\epsilon}'} \sqrt{1-\epsilon'} \| \widehat{x}(\theta_i) - x \|_{\widehat{\Sigma}(\overline{\theta})^{-1}} && \text{(Case \eqref{case:More})} \nonumber\\
		&\leq \frac{\widetilde{\epsilon}'}{1+\widetilde{\epsilon}'} \| \widehat{x}(\theta_i) - x \|^*_{x(\overline{\theta})} && \text{(by \eqref{eq:ApproximateInverseCovCondition})} \nonumber\\
		&\leq \frac{\widetilde{\epsilon}'}{1+\widetilde{\epsilon}'} \left(\| x(\theta_i) - x \|^*_{x(\overline{\theta})} + \| \widehat{x}(\theta_i) - x(\theta_i) \|^*_{x(\overline{\theta})} \right) && \text{(triangle inequality)},
		\label{eq:ProofThetaRoutineAnalysis1}
	\end{align}
	where we have used $\Sigma(\overline{\theta})^{-1} = H^*(x(\overline{\theta}))$. Again using $\Sigma(\overline{\theta})^{-1} = H^*(x(\overline{\theta}))$, the outermost inequality in \eqref{eq:ProofThetaRoutineAnalysis1} can be rewritten as
	\begin{equation*}
		\left\|  \widehat{x}(\theta_i) - x(\theta_i) \right\|_{\Sigma(\overline{\theta})^{-1}}
		\leq \widetilde{\epsilon}' \left\| x(\theta_i) - x \right\|_{\Sigma(\overline{\theta})^{-1}}.
	\end{equation*}
	Since we also have an approximation $\widehat{\Sigma}(\overline{\theta}) \approx \Sigma(\overline{\theta})$ satisfying \eqref{eq:ApproximateCovCondition} and \eqref{eq:ApproximateInverseCovCondition} for some $\epsilon' > 0$, one can show in a manner analogous to Corollary 7.5 in de Klerk, Glineur and Taylor \cite{deklerk2017worst} that
	\begin{equation*}
		\left\| \widehat{\Sigma}(\overline{\theta})^{-1} (\widehat{x}(\theta_i) - x) - \Sigma(\overline{\theta})^{-1} (x(\theta_i) - x) \right\|_{\Sigma(\overline{\theta})}
		\leq \widehat{\epsilon}' \left\| \Sigma(\overline{\theta})^{-1} (x(\theta_i) - x) \right\|_{\Sigma(\overline{\theta})},
	\end{equation*}
	where $\widehat{\epsilon}' = \widetilde{\epsilon}'\sqrt{\frac{1+\epsilon'}{1-\epsilon'}} + \sqrt{\frac{2\epsilon'}{1-\epsilon'}}$. Hence, we have an approximate gradient with respect to $\langle \cdot, \cdot \rangle_{\overline{\theta}}$, and we can apply approximate gradient descent. Recall that $\Psi$ is self-concordant with minimizer $\theta(x)$.
	Note that Theorem 3.6 in \cite{deklerk2017worst} shows that the spectrum of $H_{\overline{\theta}}(\theta_i)$ is contained in $[(1-d')^2, \frac{1}{(1-d')^2}]$ for all $i$. Theorem 5.3 in \cite{deklerk2017worst} shows that if $\widehat{\epsilon}' \leq \frac{2(1-d')^4}{1 + (1-d')^4}$ and $\gamma'$ takes the value it is given in Algorithm \ref{alg:ThetaSubroutine}, then it holds for $\theta_{i+1} = \theta_i - \gamma' \widehat{\Sigma}(\overline{\theta})^{-1} (\widehat{x}(\theta_i) - x)$ that
	\begin{equation*}
		\| \theta_{i+1} - \theta(x) \|_{\overline{\theta}} \leq \left( \frac{1 - (1-d')^4}{1 + (1-d')^4} + \widehat{\epsilon}' \right) \| \theta_i - \theta(x) \|_{\overline{\theta}}.
	\end{equation*}
	Hence, if our starting point $\theta_0$ satisfies $\| \theta_0 - \theta(x) \|_{\overline{\theta}} \leq B$, then after
	\begin{equation*}
		m' = \left\lceil \frac{\log( b / B )}{\log \left( \frac{1 - (1-d')^4}{1 + (1-d')^4} + \widehat{\epsilon}' \right)} \right\rceil
	\end{equation*}
	iterations, we have $\| \theta_{m'} - \theta(x) \|_{\overline{\theta}} \leq b$. Recall that the probability of failure, where not finding a good enough approximation of $x(\theta_i) = \mathbb{E}_{\theta_i}[X]$ constitutes a failure of the algorithm, is at most $p'/m'$ in each iteration. By the union bound, the probability of success after $m'$ iterations is therefore at least $1-p'$.
%	
%	The total error in the gradient of $\Psi$ at $\theta_i$ with respect to $\langle \cdot, \cdot \rangle_{\overline{\theta}}$ depends on the error in our approximation of $\Sigma(\overline{\theta})$ and our approximation of $\mathbb{E}_{\theta_i}[X]$. If we have an approximation $\widehat{\Sigma}(\overline{\theta}) \approx \Sigma(\overline{\theta})$ satisfying \eqref{eq:ApproximateCovCondition} and \eqref{eq:ApproximateInverseCovCondition} for some $\epsilon > 0$, and we have an $ \widehat{\mathbb{E}}_{\theta_i}[X] \approx \mathbb{E}_{\theta_i}[X]$ such that
%	\begin{equation}
%		\left\|  \widehat{\mathbb{E}}_{\theta_i}[X] - \mathbb{E}_{\theta_i}[X] \right\|_{\Sigma(\overline{\theta})^{-1}}
%		\leq \epsilon' \left\| \mathbb{E}_{\theta_i}[X] - x \right\|_{\Sigma(\overline{\theta})^{-1}},
%		\label{eq:ThetaApproximationGradientCondition}
%	\end{equation}
%	for some $\epsilon' > 0$,  Naturally, the right hand side of \eqref{eq:ThetaApproximationGradientCondition} is unknown, since we do not know $\mathbb{E}_{\theta_i}[X]$ and $\Sigma(\overline{\theta})^{-1}$ exactly. However, if $\left\| \mathbb{E}_{\theta_i}[X] - x \right\|_{\Sigma(\overline{\theta})^{-1}}$ were very small, $\theta_i$ would already be a good approximation of $\theta(x)$. Formally, if
%	$\left\| \mathbb{E}_{\theta_i}[X] - x \right\|_{\Sigma(\overline{\theta})^{-1}} \leq \frac{b(1 - C)^2}{1 + b(1-C)}$,
%	then by Lemma \ref{lemma:ProximityThetaGeneral}
%
%	
%	We can now analyze gradient descent applied to $\Psi$.
\end{proof}
%The analysis will be done with respect to the local inner product at some $\overline{\theta} \in \mathbb{R}^n$ (in the main algorithm, this will be the vector $-\eta_{k+1} c$ in iteration $k$).
%We need this reference point to be close to all iterates, in the sense that $\| \theta_i - \overline{\theta} \|_{\overline{\theta}} \leq \delta' < 1$ for all $i$. To enforce this ex ante, note that if gradient descent indeed improves the distance of the iterate $\theta_i$ to $\theta(x)$ in each iteration, then for all $i$,

%If the right hand side is at most some $\delta' < 1$, then Theorem 3.6 in \cite{deklerk2017worst} shows that the spectrum of $H_{\overline{\theta}}(\theta_i)$ is contained in $[(1-\delta')^2, \frac{1}{(1-\delta')^2}]$ for all $i$. (In our main algorithm, we will bound $\| \theta(x_k) - \theta_0 \|_{\overline{\theta}}$ using the fact that $x$ lies close to the central path and $\theta_0 = \theta(z(\eta_k))$, and we will bound $\| \theta_0 - \overline{\theta} \|_{\overline{\theta}}$ using the fact that both $\theta_0$ and $\overline{\theta} = \theta(z(\eta_{k+1}))$ are known.)

\subsection{Parameter Values}
\label{subsec:ParameterValues}
%Recall that we are looking for a $\widehat{\theta}(x_k)$ satisfying \eqref{eq:ThetaCondition}. One immediate problem is that the right hand side of \eqref{eq:ThetaCondition} is unknown, since it depends on $\theta(x_k)$, the very vector we are trying to approximate. Hence, we will first attempt to lower bound $\left\| \theta(z(\eta_{k+1})) - \theta(x_k) \right\|_{\Sigma(-\eta_{k+1} c)}$, and then present a gradient descent algorithm that yields the desired approximation $\widehat{\theta}(x_k)$ in the next section.
The algorithm in the previous section allows us to approximate $\theta(x_k)$ up to any desired accuracy $b > 0$. However, the requirement \eqref{eq:ThetaCondition} is stated in terms of the unknown expression $\| \theta(z(\eta_{k+1})) - \theta(x_k) \|_{\Sigma(-\eta_{k+1} c)}$. The purpose of this section is to provide a lower bound on this expression, which may then be used in the role of $b$ in Algorithm \ref{alg:ThetaSubroutine}.

The triangle inequality shows that
\begin{align}
	\left\| \theta(z(\eta_{k+1})) - \theta(x_k) \right\|_{\Sigma(-\eta_{k+1} c)}
	&\geq \left\| \theta(z(\eta_{k+1})) - \theta(z(\eta_k)) \right\|_{\Sigma(-\eta_{k+1} c)} - \left\| \theta(z(\eta_k)) - \theta(x_k) \right\|_{\Sigma(-\eta_{k+1} c)} \nonumber\\
	&\geq (\eta_{k+1} - \eta_k) \left\| c \right\|_{\Sigma(-\eta_{k+1} c)} - \frac{\left\| \theta(z(\eta_k)) - \theta(x_k) \right\|_{\Sigma(-\eta_k c)}}{1 - (\eta_{k+1} - \eta_k) \left\| c \right\|_{\Sigma(-\eta_k c)}},
	\label{eq:LowerBoundThetaCriterionRHS}
\end{align}
where the second inequality uses \eqref{eq:SCf}.
%\begin{figure}
%	\centering
%	\usepgfplotslibrary{fillbetween}
%	\begin{tikzpicture}
%		\begin{axis}[axis equal image, domain=0:0.5, domain y = 0:1,
%%		xmax=1/3, ymax = 1/3, enlarge x limits, enlarge y limits,
%%		samples = 100,
%		ylabel = {$\| \theta(x) - \theta(z) \|_{\theta(z)}$}, xlabel = {$\| x - z \|_z^*$}, ylabel near ticks]
%%			\addplot[mark=none, thick, name path = B] (x - x^2 + x^3/3, x );
%%			\addplot[mark=none, thick, name path = A] {x - x^2 + x^3/3};
%			\addplot[mark=none, thick, name path = A] {x/(1-x)};
%			\addplot[mark=none, thick, name path = B] {x/(1+x)};
%			
%%			\draw[thick, black, name path = C] (1/3,1) -- (1/3, 2);
%%			\draw[thick, black, name path = D] (1,1/3) -- (2, 1/3);
%			
%%			\addplot[gray!70] fill between [of=A and B];
%%			\addplot[gray!70] fill between [of=C and D];
%		\end{axis}
%	\end{tikzpicture}
%	\caption{Admissible combinations of $\| x - z \|_z^*$ and $\| \theta(x) - \theta(z) \|_{\theta(z)}$ according to Lemma \ref{lemma:ProximityTheta}}
%	\label{fig:ProximityTheta}
%\end{figure}
%
%Our next step in bounding \eqref{eq:LowerBoundThetaCriterionRHS} is to bound $\left\| \theta(z(\eta_k)) - \theta(x_k) \right\|_{\Sigma(-\eta_k c)}$ from above.
%
%
%It follows from this lemma that if $\| x - z \|_z^* \leq \frac{1}{3}$, then
%\begin{equation*}
%	\| \theta(x) - \theta(z) \|_{\theta(z)} \leq 1 - \sqrt[3]{1 - 3 \| x - z \|_z^*}.
%\end{equation*}
%
It follows from Corollary \ref{cor:ProximityThetaTwoPoints} that if $\| x_k - z(\eta_k) \|_{z(\eta_k)}^* < 1$, we can further bound \eqref{eq:LowerBoundThetaCriterionRHS} by
\begin{equation}
	\left\| \theta(z(\eta_{k+1})) - \theta(x_k) \right\|_{\Sigma(-\eta_{k+1} c)}
	\geq (\eta_{k+1} - \eta_k) \left\| c \right\|_{-\eta_{k+1} c} - \frac{\frac{\left\| x_k - z(\eta_k) \right\|_{z(\eta_k)}^*}{ 1 - \left\| x_k - z(\eta_k) \right\|^*_{z(\eta_k)}}}{1 - (\eta_{k+1} - \eta_k) \left\| c \right\|_{-\eta_k c}}.
	\label{eq:LowerBoundThetaCriterionRHS2}
\end{equation}
%Since $c$ is a fixed unit vector, one might hope to lower bound $\left\| c \right\|_{\Sigma(-\eta_{k+1} c)}$ using the smallest eigenvalue of $\Sigma(-\eta_{k+1} c)$.

For our main algorithm, we can configure how closely we want to follow the central path, i.e. how small $\| x_k - z(\eta_k) \|_{z(\eta_k)}^*$ should be in each iteration $k$. Moreover, with Lemma \ref{lemma:LowerBoundTheta} and \eqref{eq:UpperBoundSigma}, we have lower and upper bounds on $\| c \|_{-\eta_{k+1} c}$ and $\| c \|_{-\eta_k c}$, respectively. The only remaining ingredient to finding an explicit lower bound on $\| \theta(z(\eta_{k+1})) - \theta(x_k) \|_{\Sigma(-\eta_{k+1} c)}$ is therefore to set explicit parameter values.

The parameter values in Algorithm \ref{alg:ShortStepIPM} we still have to determine are the approximation error $\widehat{\epsilon}$, the growth factor $\beta$ and the central path proximity $\delta$. For most short-step interior point methods, this is a matter of keeping $\beta$ relatively small and $\delta$ large enough to guarantee that after one takes a step from the current iterate, the resulting point will lie within distance $\frac{1}{2} \delta$ of the next minimizer on the central path. More formally, suppose $\| x_k - z(\eta_k) \|^*_{z(\eta_k)} \leq \frac{1}{2} \delta$, and we want to apply an approximate gradient descent step to estimate $z(\eta_{k+1})$. First note that, as in the proof of Theorem 7.6 in de Klerk, Glineur and Taylor \cite{deklerk2017worst},
\begin{equation}
	\| z(\eta_{k+1}) - z(\eta_k) \|_{z(\eta_k)}^* \leq \beta + \frac{3 \beta^2}{(1-\beta)^3},
	\label{eq:DistanceMinimizers}
\end{equation}
and therefore, by the second inequality in \eqref{eq:SCf*},
\begin{equation}
	\| x_k - z(\eta_{k+1}) \|_{z(\eta_{k+1})}^* \leq \frac{\| x_k - z(\eta_k) \|_{z(\eta_k)}^* + \| z(\eta_{k+1}) - z(\eta_k) \|_{z(\eta_k)}^*}{1 - \| z(\eta_{k+1}) - z(\eta_k) \|_{z(\eta_k)}^*}
	\leq \frac{\frac{1}{2} \delta + \beta + \frac{3 \beta^2}{(1-\beta)^3}}{1 - \left( \beta + \frac{3 \beta^2}{(1-\beta)^3} \right)}. \label{eq:DistanceXToNextMinimizer}
\end{equation}
Supposing we have a $\widehat{g}^{*,\eta_{k+1}}(x_k)$ and $\widehat{\Sigma}(-\eta_{k+1} c)$ that satisfy
\begin{equation*}
	\left\| \widehat{\Sigma}(-\eta_{k+1}c) \widehat{g}^{*,\eta_{k+1}}(x_k) - g_{z(\eta_{k+1})}^{*,\eta_{k+1}}(x_k) \right\|_{z(\eta_{k+1})}^* \leq \widehat{\epsilon} \left\| g_{z(\eta_{k+1})}^{*,\eta_{k+1}}(x_k) \right\|_{z(\eta_{k+1})}^*,
\end{equation*}
we can now apply Corollary 6.2 in \cite{deklerk2017worst}. After a damped gradient descent step, we have
\begin{equation}
	\| x_{k+1} - z(\eta_{k+1}) \|_{z(\eta_{k+1})}^* \leq \left( \frac{1 - \kappa}{1 + \kappa} + \widehat{\epsilon} \right) \| x_k - z(\eta_{k+1}) \|_{z(\eta_{k+1})}^*
	\leq \left( \frac{1 - \kappa}{1 + \kappa} + \widehat{\epsilon} \right) \frac{\frac{1}{2} \delta + \beta + \frac{3 \beta^2}{(1-\beta)^3}}{1 - \left( \beta + \frac{3 \beta^2}{(1-\beta)^3} \right)},
	\label{eq:DistanceToNewMinimizer}
\end{equation}
where
\begin{equation*}
	\kappa = \left( 1 - \frac{\delta + 2\beta + \frac{6 \beta^2}{(1-\beta)^3}}{1 - \left( \beta + \frac{3 \beta^2}{(1-\beta)^3} \right)} \right)^4.
\end{equation*}
Thus, \eqref{eq:DistanceToNewMinimizer} should be at most $\frac{1}{2} \delta$ to maintain a distance of at most $\frac{1}{2} \delta$ to the central path.

However, in our algorithm, we are also looking for a $\widehat{\theta}(x_k)$ that satisfies \eqref{eq:ThetaCondition}. The right hand side of this condition depends on the unknown quantity $\| \theta(z(\eta_{k+1})) - \theta(x_k) \|_{\Sigma(-\eta_{k+1} c)}$, and therefore we would like to bound it from below by some known, fixed quantity $Q > 0$. If we can find a $\widehat{\theta}(x_k)$ such that  $\| \widehat{\theta}(x_k) - \theta(x_k) \|_{\Sigma(-\eta_{k+1} c)} \leq \widetilde{\epsilon} Q$, then \eqref{eq:ThetaCondition} holds as well. Recall that we already made some progress in bounding $\| \theta(z(\eta_{k+1})) - \theta(x_k) \|_{\Sigma(-\eta_{k+1} c)}$ from below in \eqref{eq:LowerBoundThetaCriterionRHS2}.
To make further developing this bound easier, one would like to increase $\beta$ as much as possible (such that $\eta_{k+1}$ and $\eta_k$ are further from each other), and decrease $\delta$ (such that $x_k$ lies closer to $z(\eta_k)$). This desire thus conflicts with our wish to make \eqref{eq:DistanceToNewMinimizer} smaller than or equal to $\frac{1}{2} \delta$.
%At the very least, the lower bound \eqref{eq:LowerBoundThetaCriterionRHS2} should be positive for all iterations $k$ where we still have to make a step.

The following lemma gives parameter values where both demands are met. These values depend on a number $S$ with $1/S = O(\frac{n \sqrt{\vartheta}}{r \eta_0}) = O(\frac{n \sqrt{n}}{r \eta_0})$. We first show that \eqref{eq:DistanceToNewMinimizer} holds, which is a purely algebraic property.
\begin{lemma}
	\label{lemma:ParameterValues}
	Let $r, R > 0$ such that $R \geq r$.
	Let $n, \vartheta \geq 1$, $\eta_0 > 0$, and $\bar{\epsilon} \geq 0$.
	Define the parameter
	\begin{equation*}
		S := \frac{1}{4} \frac{1}{\sqrt{\vartheta} + \frac{7}{176}} \frac{r \eta_0}{2(n+ 1) + r \eta_0},
	\end{equation*}
	and set
	\begin{equation*}
		\widehat{\epsilon} = \frac{S}{32 + 7 S}, \qquad
		\beta = \min\left\{ \frac{21S}{640 + 332S}, \frac{7}{176}\frac{32+5S}{32+7S}, \frac{\bar{\epsilon}}{4R \sqrt{\vartheta}(1+\frac{7}{1408})} \right\}, \qquad
		\delta = S \beta.
	\end{equation*}
	Then,
	\begin{equation*}
		\left( \frac{1 - \left(1- \frac{\delta + 2 \overline{\beta}}{1-\overline{\beta}} \right)^4}{1 + \left(1 - \frac{\delta + 2 \overline{\beta}}{1-\overline{\beta}} \right)^4} + \widehat{\epsilon} \right)
		\frac{\overline{\beta} + \frac{1}{2} \delta }{1 - \overline{\beta} } \leq \tfrac{1}{2} \delta,
	\end{equation*}
	where $\overline{\beta} := \beta + \frac{3\beta^2}{(1-\beta)^3}$.
\end{lemma}
\begin{proof}
	See Appendix \ref{app:Proofs}.
\end{proof}

Next, we show that the right hand side \eqref{eq:LowerBoundThetaCriterionRHS2} can be bounded away from zero for the parameter values in Lemma \ref{lemma:ParameterValues}.

\begin{lemma}
	\label{lemma:ParameterValuesBound}
	Let $K \subseteq \mathbb{R}^n$ be a convex body that contains a Euclidean ball of radius $r > 0$ and is contained in a Euclidean ball of radius $R \geq r$.
	Define $f$ as the log partition function $f(\theta) = \ln \int_{K} e^{\langle \theta, x \rangle} \diff x$, where $\langle \cdot, \cdot \rangle$ is the Euclidean inner product, and denote its Hessian by $\Sigma(\theta)$. Let $f^*$ be the entropic barrier for $K$ with complexity parameter $\vartheta \geq 1$, and let $\eta_0, \bar{\epsilon} > 0$.
	Let the values of $S$, $\widehat{\epsilon}$, $\beta$ and $\delta$ be as in Lemma \ref{lemma:ParameterValues}.
	For all positive integers $k \leq \left\lceil \log(\frac{\vartheta (1 + \delta / 2)}{\eta_0 \bar{\epsilon}}) / \log(1 + \beta / \sqrt{\vartheta}) \right\rceil$, define $\eta_k = (1 + \beta/\sqrt{\vartheta})^k \eta_0$. Then, for all $k \leq \left\lceil \log(\frac{\vartheta (1 + \delta / 2)}{\eta_0 \bar{\epsilon}}) / \log(1 + \beta / \sqrt{\vartheta}) \right\rceil - 1$,
	\begin{equation*}
	(\eta_{k+1} - \eta_k) \left\| c \right\|_{-\eta_{k+1} c} - \frac{\frac{\delta/2}{ 1 - \delta/2}}{1 - (\eta_{k+1} - \eta_k) \left\| c \right\|_{-\eta_k c}}
	\geq \frac{1}{2} \frac{\beta}{\sqrt{\vartheta} + \beta} \frac{r \eta_{k+1}}{2(n+ 1) + r \eta_{k+1}}.
	\end{equation*}
\end{lemma}
\begin{proof}
	Let $k \leq \left\lceil \log(\frac{\vartheta (1 + \delta / 2)}{\eta_0 \bar{\epsilon}}) / \log(1 + \beta / \sqrt{\vartheta}) \right\rceil - 1$.
	We observe that $\frac{\delta/2}{1-\delta/2} \leq \delta = S \beta$, where the inequality holds for $\delta \in [0, 1]$. Using the definition of $S$, we get
%	for all $k\leq \left\lceil \log(\frac{\vartheta (1 + \delta / 2)}{\eta_0 \bar{\epsilon}}) / \log(1 + \beta / \sqrt{\vartheta}) \right\rceil - 1$,
	\begin{equation*}
	\frac{\delta/2}{1-\delta/2} \leq \delta = \frac{1}{4} \frac{\beta}{\sqrt{\vartheta} + \frac{7}{176}} \frac{r \eta_0}{2(n+ 1) + r \eta_0} \leq \frac{1}{4} \frac{\beta}{\sqrt{\vartheta} + \beta} \frac{r \eta_{k+1}}{2(n+ 1) + r \eta_{k+1}}.
	\end{equation*}
	Through application of Lemma \ref{lemma:LowerBoundTheta} and \eqref{eq:UpperBoundSigma}, the above thus gives us	
	\begin{equation}
	(\eta_{k+1} - \eta_k) \left\| c \right\|_{-\eta_{k+1} c} - \frac{\frac{\delta/2}{ 1 - \delta/2}}{1 - (\eta_{k+1} - \eta_k) \left\| c \right\|_{-\eta_k c}}
	\geq \frac{r (\eta_{k+1} - \eta_k)}{2(n+ 1) + r \eta_{k+1}} - \frac{\frac{1}{4} \frac{\beta}{\sqrt{\vartheta} + \beta} \frac{r \eta_{k+1}}{2(n+ 1) + r \eta_{k+1}}}{1 - 2R(\eta_{k+1} - \eta_k)}.
	\label{eq:LemmaParameterValues2}
	\end{equation}
	Since the value of $k$ implies $\eta_k \leq \vartheta(1 + \frac{1}{2}\delta) / \bar{\epsilon}$,
%	for all iterations $k \leq \left\lceil \log(\frac{\vartheta (1 + \delta / 2)}{\eta_0 \bar{\epsilon}}) / \log(1 + \beta / \sqrt{\vartheta}) \right\rceil - 1$
	and $\delta = S \beta \leq \frac{1}{4} \frac{7}{176} = \frac{7}{704}$,
	\begin{equation*}
		2R(\eta_{k+1} - \eta_k) = \frac{2R \beta}{\sqrt{\vartheta}} \eta_k \leq \frac{2R \beta}{\sqrt{\vartheta}} \frac{\vartheta (1 + \delta /2)}{\bar{\epsilon}} \leq \tfrac{1}{2},
	\end{equation*}
	by our choice of $\beta$. With this fact, and $\eta_{k+1} \leq (1 + \frac{\beta}{\sqrt{\vartheta}})\vartheta(1 + \frac{1}{2}\delta) / \bar{\epsilon}$,
%	 for all $k\leq \left\lceil \log(\frac{\vartheta (1 + \delta / 2)}{\eta_0 \bar{\epsilon}}) / \log(1 + \beta / \sqrt{\vartheta}) \right\rceil - 1$,
	we can further bound \eqref{eq:LemmaParameterValues2} by
	\begin{align*}
	&(\eta_{k+1} - \eta_k) \left\| c \right\|_{-\eta_{k+1} c} - \frac{\frac{\delta/2}{ 1 - \delta/2}}{1 - (\eta_{k+1} - \eta_k) \left\| c \right\|_{-\eta_k c}}\\
	&\geq \frac{r \eta_{k+1} (1- \eta_k/\eta_{k+1})}{2(n+ 1) + r \eta_{k+1}} - \frac{1}{2} \frac{\beta}{\sqrt{\vartheta} + \beta} \frac{r \eta_{k+1}}{2(n+ 1) + r \eta_{k+1}} = \frac{1}{2} \frac{\beta}{\sqrt{\vartheta} + \beta} \frac{r \eta_{k+1}}{2(n+ 1) + r \eta_{k+1}}.
	\end{align*}
\end{proof}

It is important to note that the value of $S$ is the same for $K$ as for $\alpha K$, where $\alpha > 0$. To see this, we need to show that the product $r \eta_0$ does not change as we scale $K$.
Recall that we start Algorithm \ref{alg:ShortStepIPM} with an $x_0$ close to $z(\eta_0)$, where $z(\eta_0)$ is the minimizer of $x \mapsto \eta_0 \langle c,x \rangle + f^*(x)$. Let $\bar{f}$ be the log partition function associated with $\alpha K$, and $\bar{f}^*$ the entropic barrier over $\alpha K$. We claim that $\alpha z(\eta_0)$ is the minimizer of $y \mapsto \frac{\eta_0}{\alpha} \langle c, y \rangle + \bar{f}^*(y)$.

By definition, $\bar{f}(\theta) = \ln \int_{\alpha K} e^{\langle \theta, y \rangle} \diff y = \ln \int_{K} e^{\langle \theta, \alpha x \rangle} \alpha^n \diff x$. It follows that
\begin{equation*}
	\bar{g}(\theta) = \frac{\int_{K} \alpha x e^{\langle \theta, \alpha x \rangle} \alpha^n \diff x}{\int_{K} e^{\langle \theta, \alpha x \rangle} \alpha^n \diff x} = \alpha \frac{\int_{K} x e^{\langle \alpha \theta, x \rangle} \diff x}{\int_{K} e^{\langle \alpha \theta, x \rangle} \diff x} = \alpha g(\alpha \theta).
\end{equation*}
Using $g(-\eta_0 c) = z(\eta_0)$, it is easily seen that $\bar{g}(-\frac{\eta_0}{\alpha} c) = \alpha g(-\eta_0 c) = \alpha z(\eta_0)$, which shows that $\alpha z(\eta_0)$ is the minimizer of $y \mapsto \frac{\eta_0}{\alpha} \langle c, y \rangle + \bar{f}^*(y)$. In other words, $\eta_0$ is scaled by a factor $\frac{1}{\alpha}$ if $K$ is scaled by a factor $\alpha$.
Since $\alpha K$ contains a Euclidean ball of radius $\alpha r$, the term $r \eta_0$ is invariant under scaling of $K$, and therefore $S$ is scaling-invariant as well.

\subsection{Complexity Analysis}
\label{subsec:Complexity}
With the parameter values in Lemma \ref{lemma:ParameterValues}, we can show that Algorithm \ref{alg:ShortStepIPM} follows the central path with high probability.
\begin{lemma}
	\label{lemma:FollowCentralPath}
	Let $p \in (0,1)$ and $\eta_0 > 0$, and the values of $S$, $\widehat{\epsilon}$, $\beta$ and $\delta$ as in Lemma \ref{lemma:ParameterValues}. Let $ \widetilde{\epsilon} > 0$ and $\epsilon \in (0, \sqrt{13}-3)$ be such that $\widetilde{\epsilon} \sqrt{\frac{1+\epsilon}{1-\epsilon}} + \sqrt{\frac{2 \epsilon}{1 - \epsilon}} \leq \widehat{\epsilon}$. Set $B$, $b$, $d'$ and $C$ as follows:
	\begin{align*}
		B &= \frac{\delta}{2-\delta} \frac{1 - \beta(4 - \beta(3-\beta)(1-\beta) )}{1 - \beta (5 - \beta(3-\beta)(1 - 2\beta) )}, &
		b &= \frac{\widetilde{\epsilon}}{2}\frac{\beta}{\sqrt{\vartheta} + \beta} \frac{r \eta_0}{2(n+ 1) + r \eta_0},\\
		d' &= 2 B + \frac{\beta + 3 \beta^3 -\beta^4}{1 - \beta(5- \beta(3-\beta)(1-2\beta))}, &
		C &= \frac{\frac{1}{2} \delta + \beta + \frac{3 \beta^2}{(1-\beta)^3}}{1 - \left( \beta + \frac{3 \beta^2}{(1-\beta)^3} \right)},
	\end{align*}
	and pick $\epsilon', \widetilde{\epsilon}' > 0$ such that $\widetilde{\epsilon}'\sqrt{\frac{1+\epsilon'}{1-\epsilon'}} + \sqrt{\frac{2\epsilon'}{1-\epsilon'}} < \frac{2(1-d')^4}{1 + (1-d')^4}$ and $\epsilon' \geq \epsilon$.
	Let $m$ and $N$ be as defined in Algorithm \ref{alg:ShortStepIPM}.
	If $\| x_k - z(\eta_k) \|_{z(\eta_k)}^* \leq \frac{1}{2} \delta$ in some iteration $k$ of Algorithm \ref{alg:ShortStepIPM}, then, with probability at least $1-p/m$, we have $\| x_{k+1} - z(\eta_{k+1}) \|_{z(\eta_{k+1})}^* \leq \frac{1}{2} \delta$.	
\end{lemma}
\begin{proof}
%	Let us first show that the requirements of Algorithm \ref{alg:ThetaSubroutine} are satisfied, and that this algorithm indeed yields a $\widehat{\theta}(x_k)$ satisfying \eqref{eq:ThetaCondition}, as Algorithm \ref{alg:ShortStepIPM} requires. We will call Algorithm \ref{alg:ThetaSubroutine} with $\overline{\theta} = -\eta_{k+1}c$, $\theta_0 = -\eta_k c$ and $x = x_k$.
	
	First, we need to find an approximation $\widehat{\Sigma}(-\eta_{k+1} c)$ of $\Sigma(-\eta_{k+1} c)$ satisfying \eqref{eq:ApproximateCovCondition} and \eqref{eq:ApproximateInverseCovCondition}.
	Note that by \eqref{eq:DistanceMinimizers} and the fact that $\beta \leq \frac{7}{176}$,
	\begin{equation}
		\label{eq:WalkLengthMainDeltaX}
		\Delta x = \| z(\eta_{k+1}) - z(\eta_k) \|_{z(\eta_k)}^* \leq \beta + \frac{3 \beta^2}{(1-\beta)^3} < 0.0452.
	\end{equation}
	Then, the first inequality in \eqref{eq:ProximityThetaTwoPointsTheta} from Corollary \ref{cor:ProximityThetaTwoPoints} shows that
	\begin{equation}
		\label{eq:WalkLengthMainDeltaTheta}
		\Delta \theta = \| -\eta_{k+1} c + \eta_k c \|_{-\eta_k c} \leq \frac{\Delta x}{1 - \Delta x} < 0.0473.
	\end{equation}
	Since
	\begin{equation*}
		\|-\eta_{k+1} c\| = \eta_{k+1} \leq \eta_m = \eta_0 \left(  1+\frac{\beta}{\sqrt{\vartheta}}\right)^{\left\lceil \log\left( \frac{\vartheta(1 + \delta/2)}{\eta_0 \bar{\epsilon}} \right) \big/ \log \left( 1 + \frac{\beta}{\sqrt{\vartheta}} \right) \right\rceil} \leq \frac{\vartheta(1 + \delta/2) (1+\beta/\sqrt{\vartheta})}{ \bar{\epsilon}},
	\end{equation*}
	by the value of $m$, Theorem \ref{thm:LowerBoundSigma} shows
	\begin{align*}
		\lambda_{\min}(\Sigma(-\eta_{k+1} c))
		&\geq \frac{1}{64} \left( \frac{r}{n+1} \right)^2 \min \left\{ 1, \left( 4R \eta_{k+1} \right)^{-4 \sqrt{\vartheta} - 2} \right\}\\
		&\geq \frac{1}{64} \left( \frac{r}{n+1} \right)^2 \min \left\{ 1, \left( \frac{\bar{\epsilon}}{4R \vartheta(1 + \delta/2)(1+\beta/\sqrt{\vartheta})} \right)^{4 \sqrt{\vartheta} + 2} \right\}.
	\end{align*}
	We can now apply Corollary \ref{cor:CovAndInverseQuality} for $\theta_0 = -\eta_k c$, $\theta_1 = -\eta_{k+1} c$, and $\epsilon_1 = \epsilon$. It can be shown that values of $N$ and $q$ given in Algorithm \ref{alg:ShortStepIPM} satisfy
	\begin{equation*}
		N \geq \frac{7840 n^2}{\epsilon^2 \frac{p}{m}\left( \frac{1}{2} - \frac{1}{N} \right)}, \qquad q \leq \left( \frac{\epsilon^2 \lambda_{\min}(\Sigma(-\eta_{k+1} c))^2}{799680 n^2 R^4} \right) \left(\frac{p}{m} \left(\frac{1}{2}-\frac{1}{N}\right) \right).
	\end{equation*}
%	We find that the number of hit-and-run samples $N$ and the value of $q$ given in Algorithm \ref{alg:ShortStepIPM}
	Therefore, Corollary \ref{cor:CovAndInverseQuality} guarantees that $\widehat{\Sigma}(-\eta_{k+1} c)$ satisfies \eqref{eq:ApproximateCovCondition} and \eqref{eq:ApproximateInverseCovCondition} with probability at least
	\begin{equation}
		1 - \frac{p}{m} \left(\frac{1}{2}-\frac{1}{N}\right).
		\label{eq:ProofMainProbCov}
	\end{equation}
	Note that Algorithm \ref{alg:ShortStepIPM} moreover sets $q \leq p/(mN)$. Lemma \ref{lemma:DivineIntervention} thus guarantees that with probability $1 - q \geq 1 - p/(mN)$, the sample $X_{k+1}$ follows the Boltzmann distribution with parameter $-\eta_{k+1} c$.
	
	Second, we need an approximation $\widehat{\theta}(x_k)$ of $\theta(x_k)$ satisfying \eqref{eq:ThetaCondition}. It will suffice to show that the chosen parameters satisfy the requirements for Algorithm \ref{alg:ThetaSubroutine}, and that this algorithm yields the desired result.	Thus, we will call Algorithm \ref{alg:ThetaSubroutine} with $\overline{\theta} = -\eta_{k+1} c$. Regarding the first condition of Algorithm \ref{alg:ThetaSubroutine}, using the second inequality in \eqref{eq:SCf}, Corollary \ref{cor:ProximityThetaTwoPoints} and \eqref{eq:DistanceMinimizers}, we find
	\begin{align*}
		\| \theta(x) - \theta_0 \|_{\overline{\theta}}
%		&\leq \frac{1}{1 - \| z(\eta_k) - z(\eta_{k+1} ) \|_{z(\eta_{k+1})}^*} \frac{\| x_k - z(\eta_k) \|_{z(\eta_{k+1})}^*}{1 - \| x_k - z(\eta_k) \|_{z(\eta_{k+1})}^* - \| z(\eta_k) - z(\eta_{k+1} ) \|_{z(\eta_{k+1})}^*}\\
%		&\leq \frac{1}{1 - \| z(\eta_k) - z(\eta_{k+1} ) \|_{z(\eta_{k+1})}^*} \frac{\| x_k - z(\eta_k) \|_{z(\eta_{k+1})}^*}{1 - \| x_k - z(\eta_k) \|_{z(\eta_{k+1})}^* - \| z(\eta_k) - z(\eta_{k+1} ) \|_{z(\eta_{k+1})}^*}
		&\leq \frac{\| \theta(x) - \theta_0  \|_{\theta_0}}{1 - \| \overline{\theta} - \theta_0 \|_{\theta_0}}
		\leq \frac{ \frac{\| x_k - z(\eta_k) \|_{z(\eta_k)}^*}{1 - \| x_k - z(\eta_k) \|_{z(\eta_k)}^*} }{1 - \frac{\| z(\eta_k) - z(\eta_{k+1} ) \|_{z(\eta_k)}^*}{1 - \| z(\eta_k) - z(\eta_{k+1} ) \|_{z(\eta_k)}^*}}
		\leq \frac{ \frac{\delta/2}{1 - \delta/2} }{1 - \frac{\beta + \frac{3 \beta^2}{(1-\beta)^3}}{1 - \beta - \frac{3 \beta^2}{(1-\beta)^3}}} \\
		&= \frac{\delta}{2-\delta} \frac{1 - \beta(4 - \beta(3-\beta)(1-\beta) )}{1 - \beta (5 - \beta(3-\beta)(1 - 2\beta) )} = B.
	\end{align*}
%	Furthermore, since we make progress towards $\theta(x)$ in each iteration of Algorithm \ref{alg:ThetaSubroutine}, it holds for all $i$ that
%	\begin{equation*}
%		\| \theta_i - \overline{\theta} \|_{\overline{\theta}} \leq \| \theta_i - \theta(x) \|_{\overline{\theta}} + \| \theta(x) - \theta_0 \|_{\overline{\theta}} + \| \theta_0 - \overline{\theta} \|_{\overline{\theta}} \leq 2 \| \theta(x) - \theta_0 \|_{\overline{\theta}} + \| \theta_0 - \overline{\theta} \|_{\overline{\theta}}.
%	\end{equation*}
%	We can bound $\| \theta(x) - \theta_0 \|_{\overline{\theta}}$ by $B$, and again using self-concordance of $f$ and \eqref{eq:DistanceMinimizers},
	Furthermore, we need to show that $2B + \| \theta_0 - \overline{\theta} \|_{\overline{\theta}} \leq d'$. Again using \eqref{eq:SCf}, Corollary \ref{cor:ProximityThetaTwoPoints} and \eqref{eq:DistanceMinimizers}, we find
	\begin{align*}
%		\| \theta_i - \overline{\theta} \|_{\overline{\theta}}
		2B + \| \theta_0 - \overline{\theta} \|_{\overline{\theta}}
		&\leq 2B + \frac{\| \theta_0 - \overline{\theta} \|_{\theta_0}}{1 - \| \theta_0 - \overline{\theta} \|_{\theta_0}} \leq 2B + \frac{\frac{\| z(\eta_k) - z(\eta_{k+1} ) \|_{z(\eta_k)}^*}{1 - \| z(\eta_k) - z(\eta_{k+1} ) \|_{z(\eta_k)}^*}}{1 - \frac{\| z(\eta_k) - z(\eta_{k+1} ) \|_{z(\eta_k)}^*}{1 - \| z(\eta_k) - z(\eta_{k+1} ) \|_{z(\eta_k)}^*}}\\
		&\leq 2B +  \frac{ \frac{\beta + \frac{3 \beta^2}{(1-\beta)^3}}{1 - \beta - \frac{3 \beta^2}{(1-\beta)^3}} }{1 - \frac{\beta + \frac{3 \beta^2}{(1-\beta)^3}}{1 - \beta - \frac{3 \beta^2}{(1-\beta)^3}}} = 2B+ \frac{\beta + 3 \beta^3 -\beta^4}{1 - \beta(5- \beta(3-\beta)(1-2\beta))} = d'.
	\end{align*}
%	for all $i$.
	
	Since $\epsilon \leq \epsilon'$, we can use the approximation $\widehat{\Sigma}(-\eta_{k+1} c)$ of $\Sigma(-\eta_{k+1}c)$ from the main algorithm in Algorithm \ref{alg:ThetaSubroutine}. Similarly, we already have a sample $\overline{X} = X_{k+1}$ from the main algorithm.
	
	Finally, as was shown in \eqref{eq:DistanceXToNextMinimizer},
	\begin{equation*}
		\| x - x(\overline{\theta}) \|_{\Sigma(\overline{\theta})^{-1}} = \| x_k - z(\eta_{k+1}) \|_{z(\eta_{k+1})}^* \leq \frac{\frac{1}{2} \delta + \beta + \frac{3 \beta^2}{(1-\beta)^3}}{1 - \left( \beta + \frac{3 \beta^2}{(1-\beta)^3} \right)} = C.
	\end{equation*}
	
	Hence, Theorem \ref{thm:ThetaRoutineAnalysis} shows that Algorithm \ref{alg:ThetaSubroutine} returns an approximation $\widehat{\theta}(x_k)$ of $\theta(x_k)$ such that, with probability $1-p' = 1-p/(2m)$, $\| \widehat{\theta}(x_k) - \theta(x_k) \|_{\Sigma(-\eta_{k+1} c)} \leq b$.
	Note that by \eqref{eq:LowerBoundThetaCriterionRHS2} and Lemma \ref{lemma:ParameterValuesBound}, the right hand side of \eqref{eq:ThetaCondition} is lower bounded by our value for $b$. We conclude $\widehat{\theta}(x_k)$ satisfies \eqref{eq:ThetaCondition}.
	Since $\widehat{\Sigma}(-\eta_k c)$ satisfies \eqref{eq:ApproximateCovCondition} and \eqref{eq:ApproximateInverseCovCondition}, Corollary 7.5 in de Klerk, Glineur and Taylor \cite{deklerk2017worst} shows
	\begin{equation*}
		\left\| \widehat{\Sigma}(-\eta_k c) [\eta_{k+1} c + \widehat{\theta}(x_{k})] -  g_{z(\eta_{k+1})}^{*,\eta_{k+1}}(x_k) \right\|_{z(\eta_{k+1})}^* \leq \widehat{\epsilon} \left\|  g_{z(\eta_{k+1})}^{*,\eta_{k+1}}(x_k) \right\|_{z(\eta_{k+1})}^*,
%		&\left\| \widehat{\Sigma}(-\eta_k c) [\eta_{k+1} c + \widehat{\theta}(x_{k})] - \Sigma(-\eta_k c) [\eta_{k+1} c + \theta(x_{k})] \right\|_{z(\eta_{k+1})}^*\\
%		&\leq \widehat{\epsilon} \left\| \Sigma(-\eta_k c) [\eta_{k+1} c + \theta(x_{k})] \right\|_{z(\eta_{k+1})}^*,
	\end{equation*}
	which means the distance of $x_{k+1}$ to $z(\eta_{k+1})$ is bounded by \eqref{eq:DistanceToNewMinimizer}. By Lemma \ref{lemma:ParameterValues}, we have $\| x_{k+1} - z(\eta_{k+1}) \|_{z(\eta_{k+1})}^* \leq \frac{1}{2} \delta$.
	
	The probability that $\widehat{\Sigma}(-\eta_{k+1} c)$ satisfies \eqref{eq:ApproximateCovCondition} and \eqref{eq:ApproximateInverseCovCondition} was given by \eqref{eq:ProofMainProbCov}. Moreover, the sample $X_{k+1}$ follows the Boltzmann distribution with parameter $-\eta_{k+1} c$ with probability at least $1-p/(mN)$. Finally, $\| \widehat{\theta}(x_k) - \theta(x_k) \|_{\Sigma(-\eta_{k+1} c)} \leq b$ with probability at least $1-p/(2m)$. Hence, the probability that these three events occur simultaneously is at least $1-p/m$ by the union bound.
\end{proof}
	
%With the parameter values in Lemma \ref{lemma:ParameterValues}, we can analyze the worst-case complexity of Algorithm \ref{alg:ShortStepIPM}. The following theorem shows that the complexity of Algorithm \ref{alg:ShortStepIPM} is $O^*(\frac{n^{19}}{p(r\eta_0)^8})$, where the $O^*$ notation suppresses polylogarithmic terms in the problem parameters.
%In each of the $m$ iterations of Algorithm \ref{alg:ShortStepIPM}, we find an approximation of a covariance matrix which satisfies \eqref{eq:ApproximateCovCondition} and \eqref{eq:ApproximateInverseCovCondition} with probability at least \eqref{eq:ProofMainProbCov}. This requires $N$ samples, where $N$ is given in Algorithm \ref{alg:ShortStepIPM}. One additional sample is generated in each iteration, which is from the correct distribution with probability at least $1-p/(mN)$. Moreover, in each of $m$ iterations, Algorithm \ref{alg:ThetaSubroutine} has success probability $1 - p/(2m)$.
Using this result, it is not hard to see that Algorithm \ref{alg:ShortStepIPM} converges with high probability. By the union bound, the probability that $\| x_k - z(\eta_k) \|_{z(\eta_k)}^* \leq \frac{1}{2} \delta$ for all $k \in \{1, ..., m\}$ is at least $1 - p$. After $m$ iterations, it follows from  Renegar \cite[relation (2.14)]{renegar2001mathematical} that $\langle c, x_m \rangle - \min_{x \in K} \langle c, x \rangle \leq \bar{\epsilon}$.

%	As in Algorithm \ref{alg:ThetaSubroutine}, let
%	\begin{equation*}
%		m' = \left\lceil \frac{ \log( b / B ) }{ \log \left( \frac{1 - (1-d')^4}{1 + (1-d')^4} + \widehat{\epsilon}' \right)} \right\rceil.
%	\end{equation*}
We now turn to a run-time analysis of Algorithm \ref{alg:ShortStepIPM} by breaking down the complexity of its most crucial steps. The complexities of the other lines are straightforward, and not relevant for the algorithm's total complexity.
\begin{description}
	\item[Line \ref{line:IPMWhile} in Algorithm \ref{alg:ShortStepIPM}] By construction, the main loop in Algorithm \ref{alg:ShortStepIPM} has $m$ iterations. Since $\vartheta = n + o(n)$,
	\begin{equation*}
		m = O \left( \frac{\sqrt{\vartheta}}{\beta} \log \left(\frac{\sqrt{\vartheta}}{\eta_0 \bar{\epsilon}} \right) \right) = O \left( \frac{\sqrt{n}}{\beta} \log \left( \frac{\sqrt{n}}{\eta_0 \bar{\epsilon}} \right)\right).
	\end{equation*}
	Note that if $\beta$ were a constant, this would essentially be the usual worst-case iteration complexity for interior point methods. However, we now have $1/\beta = \max\{ O(1/S), O(R\sqrt{n}/\bar{\epsilon}) \}$, where $1/S = O(\frac{n \sqrt{\vartheta}}{r \eta_0}) = O(\frac{n \sqrt{n}}{r \eta_0})$. Therefore,
	\begin{equation*}
		m = O \left( \sqrt{n} \max\left\{ \frac{n \sqrt{n}}{r \eta_0}, \frac{R\sqrt{n}}{\bar{\epsilon}} \right\} \log \left( \frac{\sqrt{n}}{\eta_0 \bar{\epsilon}} \right) \right).
	\end{equation*}
	
	\item[Line \ref{line:IPMSample} in Algorithm \ref{alg:ShortStepIPM}] Since $1/\epsilon = O((1/\widehat{\epsilon})^2) = O(1/S^2)$, the number of hit-and-run samples $N$ in every iteration satisfies
	\begin{equation*}
		\qquad N = O \left( \frac{n^2 m}{p \epsilon^2} \right) = O \left( \frac{n^2 m}{p S^4} \right)
		= O \left( \frac{n^{8.5}}{p (r\eta_0)^4} \max\left\{ \frac{n \sqrt{n}}{r \eta_0}, \frac{R\sqrt{n}}{\bar{\epsilon}} \right\} \log \left( \frac{\sqrt{n}}{\eta_0 \bar{\epsilon}} \right) \right).
	\end{equation*}
	Each of these samples originates from a random walk with $\ell$ steps, where
	\begin{equation*}
		\ell = O\left( n^3 \log^2\left( \frac{n\sqrt{n}}{q^2} \right) \log^3
		(1/q^2)
		\right), \quad \text{where} \quad \log(1/q^2) = O\left( \sqrt{n} \log\left( \frac{R n}{\bar{\epsilon}} \right) + \log\left( \frac{m^2n^{12}R^8}{p^2 r^8 S^8} \right) \right).
	\end{equation*}
	
	\item[Line \ref{line:IPMTheta} in Algorithm \ref{alg:ShortStepIPM}] The number of hit-and-run steps required in every iteration to find a suitable approximation of $\theta(x_k)$ by Algorithm \ref{alg:ThetaSubroutine} was proven in Theorem \ref{thm:ThetaRoutineAnalysis} to be $m' N' \ell'$.
	Note that $B = \Theta(\delta) = \Theta(S\beta)$ and $b = o(\widehat{\epsilon} \beta (S\sqrt{n})) = o(S^2 \beta \sqrt{n})$. Therefore, we have $m' = O(\log(\frac{1}{S\sqrt{n}}))$. It also follows that
	\begin{equation*}
		N' = O\left( \frac{nm'}{p' b^2} \right) = O\left( \frac{m' }{p S^4 \beta^2} \right)
		= O\left( \frac{n^6}{p (r\eta_0)^4} \max\left\{ \frac{n^3}{(r \eta_0)^2}, \frac{R^2 n}{\bar{\epsilon}^2} \right\} \log\left(\frac{1}{S\sqrt{n}}\right) \right).
	\end{equation*}
	The complexity of $\ell'$ is the same as that for $\ell$, but with $q$ replaced by $q'$. Since each time Algorithm \ref{alg:ThetaSubroutine} is called it holds that $\| \overline{\theta} \| = \| -\eta_{k+1} c \| \leq \eta_m = O(\vartheta/\bar{\epsilon})$,
	\begin{equation*}
		\log(1/(q')^2) = O\left( \log\left( \frac{m^2 (m')^2 R^4 n^4 }{p^2 b^4 (\widetilde{\epsilon}')^4 r^4} \right) + n \log\left( \frac{Rn}{\bar{\epsilon}} + \frac{R^2 d' n^2}{r \bar{\epsilon}}\right) \right).
	\end{equation*}
\end{description}

Let the $O^*$ notation suppresses polylogarithmic terms in the problem parameters.
It follows from the above that the quantity $m m' N' \ell'$ determines the number of hit-and-run steps, and this quantity has complexity
\begin{equation*}
	m m' N' \ell' = O^*\left( \sqrt{n} \frac{n^6}{p (r\eta_0)^4} n^3 n^5 \max\left\{ \frac{n^{4.5}}{(r \eta_0)^3}, \frac{R^3 n^{1.5}}{\bar{\epsilon}^3} \right\} \right) = O^*\left(  \frac{n^{14.5}}{p(r\eta_0)^4} \max\left\{ \frac{n^{4.5}}{(r\eta_0)^3}, \frac{R^3 n^{1.5}}{\bar{\epsilon}^3} \right\} \right).
\end{equation*}

The number of hit-and-run steps also determines the complexity of the total number of operations in Algorithm \ref{alg:ShortStepIPM}. The following theorem summarizes the preceding discussion.

\begin{theorem}
	Let $p \in (0,1)$ and $\eta_0 > 0$, and the values of $S$, $\widehat{\epsilon}$, $\beta$ and $\delta$ as in Lemma \ref{lemma:ParameterValues}. Let $ \widetilde{\epsilon} > 0$ and $\epsilon \in (0, \sqrt{13}-3)$ be such that $\widetilde{\epsilon} \sqrt{\frac{1+\epsilon}{1-\epsilon}} + \sqrt{\frac{2 \epsilon}{1 - \epsilon}} \leq \widehat{\epsilon}$.
	Set $B$, $b$, $d'$ and $C$ as in Lemma \ref{lemma:FollowCentralPath}, and pick $\epsilon', \widetilde{\epsilon}' > 0$ such that $\widetilde{\epsilon}'\sqrt{\frac{1+\epsilon'}{1-\epsilon'}} + \sqrt{\frac{2\epsilon'}{1-\epsilon'}} < \frac{2(1-d')^4}{1 + (1-d')^4}$ and $\epsilon' \geq \epsilon$.
	Let $m$ and $N$ be as defined in Algorithm \ref{alg:ShortStepIPM}.
	
	Then, with probability at least $1-p$, Algorithm \ref{alg:ShortStepIPM} returns an $x_m \in K$ such that $\langle c, x_m \rangle - \min_{x \in K} \langle c, x \rangle \leq \bar{\epsilon}$ after $m$ iterations.	
	If $\bar{\epsilon}/R$ is fixed, the total complexity of Algorithm \ref{alg:ShortStepIPM} is determined by
	\begin{equation}
		\label{eq:AsymptoticComplexityFinal}
		O^*\left( \frac{n^{19}}{p(r\eta_0)^7} \right),
	\end{equation}
	hit-and-run steps.
\end{theorem}
%Note that \eqref{eq:AsymptoticComplexityFinal} is $O^*(n^{12}/(r\eta_0)^7)$. Recall from the previous section that $r \eta_0$ is scaling-invariant.

%
%\begin{equation*}
%	O \left( \frac{n^{12}}{(r\eta_0)^7} \log\left(\tfrac{\sqrt{n}}{\eta_0 \bar{\epsilon}} \right) \log\left(\tfrac{n \sqrt{n}}{r \eta_0} \right) \log^2 \left( \tfrac{1}{p}\log\left(\tfrac{n \sqrt{n}}{r \eta_0} \right) \right)
%	+ \frac{n^{10}}{(r\eta_0)^5} \log\left(\tfrac{\sqrt{n}}{\eta_0 \bar{\epsilon}} \right) \log^2\left( \tfrac{n^2}{p} \log(\tfrac{\sqrt{n}}{\eta_0 \bar{\epsilon}}) \right) \right),
%\end{equation*}
%which is $O^*(n^{12}/(r\eta_0)^7)$, where the $O^*$ notation suppresses polylogarithmic terms in the problem parameters. Recall from the previous section that $r \eta_0$ is scaling-invariant.

\section{Concluding Remarks}
Our algorithm uses the starting conditions that a point $x_0$ is known with $\| x_0 - z(\eta_0) \|^*_{z(\eta_0)}$ small, and an approximation $\widehat{\Sigma}(-\eta_0 c)$ of $\Sigma(-\eta_0 c)$ is known.
A natural question is then how to find such $x_0$ and $\widehat{\Sigma}(-\eta_0 c)$. As was shown in Lov\'asz and Vempala \cite{lovasz2006hit}, it is possible to generate samples from the uniform distribution under some mild assumptions. Thus, one can approximate the mean and the covariance matrix of the uniform distribution. If $\eta_0$ is sufficiently small, self-concordance ensures that these approximations can serve as $x_0$ and $\widehat{\Sigma}(-\eta_0 c)$ to start the algorithm.

Alternatively, if $\eta_0$ is too large to use this approach, we could use the algorithm by Kalai and Vempala \cite{kalai2006simulated} to generate approximations of $\Sigma(-\eta c)$ for increasing $\eta > 0$. Once $\eta$ is large enough, the results in this paper show that one can apply hit-and-run sampling to the Boltzmann distribution with parameter $-\eta_0 c$, and thereby approximate its mean to find a point $x_0$ with $\| x_0 - z(\eta_0) \|^*_{z(\eta_0)}$ as small as desired.

The complexity bound \eqref{eq:AsymptoticComplexityFinal} is much worse than the bounds normally associated with interior point methods. The main cause is the high number of hit-and-run samples required to approximate the covariance matrices to sufficient accuracy.
%the discussion outlined in Section \ref{subsec:ParameterValues}: we need a positive lower bound on the right hand side of \eqref{eq:ThetaCondition} for the analysis, and consequently, the steps the algorithm takes are very small. In practice however, one can estimate the right hand side of \eqref{eq:ThetaCondition} with the current estimate $\widehat{\theta}(x_k)$ of $\theta(x_k)$ and conclude if further progress is necessary.
The practical applicability of sampling-based interior point methods will thus largely depend on the number of samples needed. In practice, possible approaches to reduce this workload could include mixing acceleration (see e.g. Kaufman and Smith \cite{kaufman1998direction}), or long-step and predictor-corrector methods (to reduce the number of interior point iterations). We hope the results in this paper may further motivate the development of such improved sampling methods.

%\begin{align*}
%	\left\| \frac{1}{L-\mu} (L y - \mu x + g(x) - g(y)) - y \right\| &= \frac{1}{L-\mu} \left\|  \mu y - \mu x + g(x) - g(y) \right\|\\
%	&\leq \frac{\mu \| x - y \| + L \| x - y \|}{L-\mu}\\
%	&\leq \frac{L + \mu}{L-\mu} \delta
%\end{align*}

\section*{Acknowledgments}
The authors would like to thank Nikolaus Schweizer for many insightful discussions and for proofreading parts of this manuscript. We also thank Daniel Dadush for sharing his expertise.

\appendix
\section{Remaining Proofs}
\label{app:Proofs}

\begin{proof}[Proof of Lemma \ref{lemma:ParameterValues}]
%	The proof depends on a number of crude approximations of various complicated expressions.
	Since by definition,
	\begin{equation*}
%		\beta \leq \frac{7(7S(1-\widehat{\epsilon}) - 32 \widehat{\epsilon})}{8(7S(12 - \widehat{\epsilon}) + 32(5-\widehat{\epsilon}))},
		\beta \leq \frac{21S}{640 + 332S} = \frac{7 (7 S (\widehat{\epsilon} - 1) + 32 \widehat{\epsilon})}{8 (7 S (\widehat{\epsilon} - 12) + 32 (\widehat{\epsilon} - 5))},
	\end{equation*}
	$\beta$ is less than or equal to the $t > 0$ at which the line $t \mapsto t S$ and the function $t \mapsto \frac{32 t(8(5 - \widehat{\epsilon})t + 7\widehat{\epsilon})}{49 - 49 \widehat{\epsilon} - 56t (12 - \widehat{\epsilon})}$
%	\frac{32 t(8(5 - \widehat{\epsilon})t + 7\widehat{\epsilon})}{49 - 49 \widehat{\epsilon} - 56t (12 - \widehat{\epsilon})}
	intersect, and thus $\delta \geq \frac{32 \beta(8(5 - \widehat{\epsilon})\beta + 7\widehat{\epsilon})}{49 - 49 \widehat{\epsilon} - 56\beta (12 - \widehat{\epsilon})} $. Note that
	\begin{equation}
		\delta
		\geq \frac{32 \beta(8(5 - \widehat{\epsilon})\beta + 7\widehat{\epsilon})}{49 - 49 \widehat{\epsilon} - 56\beta (12 - \widehat{\epsilon})}
		= \frac{4 (\frac{8}{7}\beta)((5 - \widehat{\epsilon}) (\frac{8}{7}\beta) + \widehat{\epsilon})}{1 - \widehat{\epsilon} - (\frac{8}{7}\beta) (12 - \widehat{\epsilon})}
		\geq \frac{4 (\beta + \frac{3\beta^2}{(1-\beta)^3})((5 - \widehat{\epsilon}) (\beta + \frac{3\beta^2}{(1-\beta)^3}) + \widehat{\epsilon})}{1 - \widehat{\epsilon} - (\beta + \frac{3\beta^2}{(1-\beta)^3}) (12 - (\beta + \frac{3\beta^2}{(1-\beta)^3}) - \widehat{\epsilon})},
		\label{eq:LemmaParameterValues1}
	\end{equation}
	where the second inequality holds for all $\beta \in [0, \frac{1}{24}]$.
	
	Because the function $\phi_b(t) = b - \sqrt{b^2 - t}$ is convex for all $b \in \mathbb{R}$, we can upper bound $\phi_b$ on $[0, b^2]$ by a line segment connecting the points $(0,0)$ and $(b^2, b)$. Thus, for all $t \in [0,b^2]$,
	\begin{equation*}
		b - \sqrt{b^2 - t} \leq \frac{t}{b}.
	\end{equation*}
	Applying the inequality above to \eqref{eq:LemmaParameterValues1}, where we write $\overline{\beta} := \beta + \frac{3\beta^2}{(1-\beta)^3}$, one finds
	\begin{align*}
		\delta &\geq \frac{1}{5} \frac{20 \overline{\beta}((5 - \widehat{\epsilon}) \overline{\beta} + \widehat{\epsilon})}{1 - \widehat{\epsilon} - \overline{\beta} (12 - \overline{\beta} - \widehat{\epsilon})}\\
		&\geq \frac{1}{5} \left( 1 - \widehat{\epsilon} - \overline{\beta} (12 - \overline{\beta} - \widehat{\epsilon}) - \sqrt{(1 - \widehat{\epsilon} - 12 \overline{\beta} + \overline{\beta} (\overline{\beta} + \widehat{\epsilon}))^2 - 20 \overline{\beta}((5 - \widehat{\epsilon}) \overline{\beta} + \widehat{\epsilon})} \right),
	\end{align*}
	and moreover, $\overline{\beta} \leq \frac{8}{7} \beta \leq \frac{8}{7} \frac{7}{176}(1 - 2 \widehat{\epsilon}) = \frac{1 - 2 \widehat{\epsilon}}{22} \leq 11-\widehat{\epsilon}-\sqrt{20}\sqrt{6 - \widehat{\epsilon}}$.
	Our assumptions imply that $S \leq \frac{1}{4}$, and therefore $\delta \leq \frac{1}{4} \beta \leq \frac{1}{4}\frac{7}{176} = \frac{7}{704}$. Hence, it can be shown that
	\begin{equation*}
		\delta \leq \frac{1}{5} \left( 1 - \widehat{\epsilon} - \overline{\beta} (12 - \overline{\beta} - \widehat{\epsilon}) + \sqrt{(1 - \widehat{\epsilon} - \overline{\beta} (12 - \overline{\beta} - \widehat{\epsilon}))^2 - 20 \overline{\beta}((5 - \widehat{\epsilon}) \overline{\beta} + \widehat{\epsilon})} \right),
	\end{equation*}
	and we can therefore conclude that $\delta$ is a solution to the quadratic equation
	\begin{equation*}
		\left(5\frac{\overline{\beta} + \frac{1}{2} \delta }{1 - \overline{\beta} } + \widehat{\epsilon}\right) \frac{\overline{\beta} + \frac{1}{2} \delta }{1 - \overline{\beta} } \leq \tfrac{1}{2} \delta.
	\end{equation*}
	Since for all $t \geq 0$, one has
	\begin{equation*}
		\frac{1 - (1-t)^4}{1 + (1-t)^4} \leq \tfrac{5}{2} t,
	\end{equation*}
	it follows that the chosen values for $\widehat{\epsilon}$, $\beta$ and $\delta$ satisfy
	\begin{equation*}
		\left( \frac{1 - \left(1- \frac{\delta + 2 \overline{\beta}}{1-\overline{\beta}} \right)^4}{1 + \left(1 - \frac{\delta + 2 \overline{\beta}}{1-\overline{\beta}} \right)^4} + \widehat{\epsilon} \right)
		\frac{\overline{\beta} + \frac{1}{2} \delta }{1 - \overline{\beta} }
		\leq \left(5\frac{\overline{\beta} + \frac{1}{2} \delta }{1 - \overline{\beta} } + \widehat{\epsilon}\right) \frac{\overline{\beta} + \frac{1}{2} \delta }{1 - \overline{\beta} } \leq \tfrac{1}{2} \delta.
	\end{equation*}
\end{proof}

{\small
\bibliographystyle{abbrv}

}
%\bibliography{Bibliography}

\end{document}